\pdfoutput=1
\documentclass[11pt, a4paper, oneside, reqno]{elsarticle}
\usepackage[english]{babel}
\usepackage[T1]{fontenc}
\usepackage{soul}

\usepackage[usenames,dvipsnames]{xcolor}
\usepackage[colorlinks=true,linkcolor=NavyBlue,urlcolor=RoyalBlue,citecolor=PineGreen,
	hypertexnames=false]{hyperref}
\usepackage{amsfonts}
\usepackage{amsmath}
\usepackage{amsthm}
\usepackage{stmaryrd}
 \usepackage{mathpazo}
 \usepackage[scaled]{beramono}
\usepackage{mathtools}
\usepackage{bm}
\usepackage{mathbbol}
\usepackage{shuffle}
\usepackage{rotating}
\usepackage{tikz}
\usepackage{tikz-cd} \usepackage{subfigure}
\usepackage{verbatim}
\usepackage[margin=0.7in]{geometry}
\usepackage[font=small]{caption}
\usepackage{adjustbox}
\usepackage{enumerate}
\usepackage{cleveref}
\usepackage{svg}
\usepackage{mathabx}
\usepackage{tcolorbox}

\def\cS{\mathcal{S}}

\def\bS{ {\bf S}}
\def\bT{{\bf T}}

\def\1{\mathbf{1}}

\def\a{\mathfrak{a}}

\newcommand{\x}{{\bf x}}
\theoremstyle{theorem}
\newtheorem{theorem}{Theorem}[subsection]
\newtheorem{example}[theorem]{Example}
\newtheorem{proposition}[theorem]{Proposition}
\newtheorem*{proposition*}{Proposition}
\newtheorem{corollaire}[theorem]{Corollary}
\newtheorem*{theorem*}{Theorem}
\newtheorem{lemma}[theorem]{Lemma}
\newtheorem{remark}[theorem]{Remark}
\newtheorem{definition}[theorem]{Definition}

\newtheorem{definitionproposition}[theorem]{Definition-Proposition}

\newcommand{\unit}{{\rm I}}
\newcommand{\op}{{\mathcal{P}}}
\newcommand{\CC}{{\mathbb{C}}}
\newcommand{\algop}{{\mathbb{C}\!\left[\op\right]}}

\newcommand{\mult}{\centerdot}

\newcommand{\liediff}[2]{\llbracket #1, #2\rrbracket}
\newcommand{\lieinv}[2]{[#1, #2]_{\mult}}

\newcommand{\End}{\rm End}

\newcommand{\tl}{\vartriangleleft}

\newcommand{\coloni}{\,\colon}

\newcommand{\g}{\mathfrak{g}}
\newcommand{\univalgebra}{{\mathcal{U}(\g)}}
\newcommand{\preliediff}{\vartriangleleft\!\!\!\vartriangleleft}

\newcommand{\antipode}{\mathcal{S}}
\usepackage{amssymb}

\makeatletter
\newcommand{\subjclass}[2][2020]{%
  \let\@oldtitle\@title%
  \gdef\@title{\@oldtitle\footnotetext{#1 \emph{Mathematics subject classification:} #2}}%
}
\newcommand{\keywords}[1]{%
  \let\@@oldtitle\@title%
  \gdef\@title{\@@oldtitle\footnotetext{\emph{Key words and phrases:} #1.}}%
}
\makeatother

\title{Post-Hopf algebras and non-commutative probability theory}

\author[1]{Nicolas Gilliers}
\affiliation[1]{organization={New York University Abu Dhabi, Division of Science, Mathematics}, city = {Abu Dhabi}, country = { United Arab Emirates}}
\ead{nag9000@nyu.edu}
\date{\vspace{-5ex}}

\subjclass[]{17B60,46L53}
\keywords{Non-commutative probability, post-Lie algebras, Post-Groups, Kupershmidt's operators}

\begin{document}

\begin{abstract}
We study $\mathcal{O}$-operators and post-Lie products over the same Lie algebra compatible in a certain sense. We prove that the group product corresponding to the formal integration of the Lie algebra, which is adjacent to the sum of two compatible post-Lie products, can be factorized in a way that is reminiscent of the classical Semenov-Tian-Shanskii factorization. In the second part, we explore applications in non-commutative probability. We introduce new transforms that facilitate the computation of conditionally free and conditionally monotone multiplicative convolutions involving operator-valued non-commutative distributions.
\end{abstract}

\maketitle

\tableofcontents

\section{Introduction}

This article explores relations between operator-valued (conditionally) free multiplicative convolution of non-commutative distributions and the theory of operated sets. We prove a factorisation that is satisfied by matching (Hopf, group) $\mathcal{O}$-operators (whose definitions are original to this work), which is a generalisation of the well-known \emph{Semenov-Tian-Shanskii factorisation} and discuss \emph{its implications to non-commutative probability theory}. This work continues the research presented in \cite{ebrahimi2021twisted} and builds upon the observation that a post-Lie structure arises in the computation of the free multiplicative convolution of the two non-commutative distributions.

Voiculescu, in his 1987 paper \cite{voiculescu1987multiplication}, introduced the $T$-transform as a non-commutative analogue to the classical Fourier-Melin transform. The goal was to give an algorithmic solution to computing the spectral distribution of the self-adjoint operator $\sqrt{a}b\sqrt{a}$, where $a \geq 0$ and $b \geq 0$ are two free operators. Dykema extended Voiculescu's work in 2006 \cite{dykema2006stransform} and defined the operator-valued $T$-transform of a non-commutative distribution. 
Dykema observed a curious factorisation of this transform over the multiplication of two random variables free with amalgamation, known as the twisted factorisation of the $T$-transform \cite{dykema2006stransform,ebrahimi2021twisted}, which was interpretated in \cite{ebrahimi2021twisted} as a group product adjacent to a post-Lie operation. The second observation that we make is the peculiar structure of this post-Lie operation; it is the sum of two other post-Lie operations \emph{compatible in a certain sense}. To the best of our knowledge, this is the first example of such pair of post-Lie operations and we start their study in the first part of the present work. We explain how we can use an abstract extension of Dykema's $T$-transform to perform computations in the realm of deformations of post-Lie algebras. 

In the second part, we then use this abstract formalism to study operator-valued conditionally free and operator-valued conditionally monotone multiplicative convolutions of non-commutative distributions. We refer to the works of Młotkowski and Bożejko \cite{mlotkowski2002operator,bozejko1996convolution} on operator-valued conditionally free convolutions and the work of Popa \cite{popa2012realization} on operator-valued conditionally monotone multiplicative convolutions.

Our work also delves into Rota-Baxter operators and their generalisations introduced by Kupershmidt, known as $\mathcal{O}$-operators. Baxter initially introduced these operators in the realm of probability theory, following the work by Spitzer. Rota formalised the notion of what is nowadays known as Rota-Baxter operator on an associative algebra. This ties in with the Rota-Baxter operators on a Lie algebra already considered in \cite{semenov1983classical} in the context of the classical Yang--Baxter equation. Kupershmidt studied a class of generalised Rota--Baxter operators defined over Lie modules, which we refer to as Lie $\mathcal{O}$-operators and are of prime interest in this work. Recently, there has been interest in developing a theory of operated groups, starting with the work by Guo et al. \cite{guo2021integration}. This theory was further developed in \cite{bardakov2021rota} and \cite{bardakov2022rota}, drawing connections with the quantum Yang-Baxter equation using the theory of (skew-left) braces. These structures appear also in our work. Various algebraic structures accompany a Lie $\mathcal{O}$-operator, which we refer to as adjacent structures. For instance, to any  Lie $\mathcal{O}$-operator corresponds a post-Lie algebra \cite{curry2019post}. Post-Lie algebras are transversal to many fields in mathematics and are in particular central to the theory of numerical integration of differential equations on homogeneous spaces, with vanishing torsion. It is our intention here to demonstrate their utility in the realm of non-commutative probability as well.

Let us now be a bit more precise about our contributions.
A post-Lie algebra yields a second Lie algebra besides the "controlling'' one of the post-Lie operation (which is the standard Jacobi bracket of vector fields in differential geometry and this additional bracket equals the curvature). One notable feature is the possibility of realising the universal envelope of this additional bracket over the universal envelope of the controlling Lie algebra (see for example \cite{jacques2023post} for an important application to multiindices). For free pre- and post-Lie algebras, the duals of the resulting Hopf algebras are known as Butcher--Connes--Kreimer (BCK) and Munthe-Kaas--Wright (MKW) Hopf algebras, respectively. Building the envelope of a post-Lie algebra implies extending the post-Lie product to polynomials on Lie elements; this is done by unfolding the \emph{Guin--Oudom recursion} \cite{oudom2008lie}. It results in an operation defined on the universal envelope of the controlling Lie algebra constitutive of a \emph{$D$-algebra} \cite{al2022algebraic} -- or \emph{Post-Hopf algebra} \cite{li2022post, li2024sub, catoire2024cartier}. The author in \cite{goncharov2021rota} introduced and studied the notion of Rota--Baxter operator on a cocommutative Hopf algebra; we call them Rota--Baxter--Hopf operators. Following Kupershmidt, we propose the notion of a Hopf $\mathcal{O}$-operator. In a sense, it is the operator form of a Post-Hopf algebra like a Lie $\mathcal{O}$-operator is the operator form of a post-Lie algebra and group $\mathcal{O}$-operator is the operator form of a Post-Group. In this operator form, and this is the primary observation of the present work, \emph{unfolding the Guin--Oudom recursion is equivalent to exponentiating with respect to an associative product} (denoted $\#$ in this work; this is not to be confused with the compositional product of \cite{agrachev1980chronological})) the Lie $\mathcal{O}$-operator. 

We propose a definition of a matching Lie $\mathcal{O}$-operator. It is, in particular, a pair of a Lie $\mathcal{O}$-operator satisfying to some compatibility, implying that the sum of these two operators is a Lie $\mathcal{O}$-operator. We prove that such a pair is a commutative pair with respect to $\#$ and consequently that their exponential factorises; and we explain how this factorisation can be viewed as an abstract form of the Semenov-Tian-Shanskii factorisation.

We explore applications in non-commutative probability of this product $\#$ and matching $\mathcal{O}$-operators, focusing on free and monotone conditional probability. We aim to identify a matching group $\mathcal{O}$-operator that governs the relations between the various transforms used to calculate non-commutative convolutions. 
In particular, we manage to cast almost all relations between the various transforms in non-commutative probability as \emph{double inversions with respect to Post-Group products}. By using our new formalism and the concept of crossed product of groups, we propose a definition for the conditional operator-valued $T$-transform and prove that it factorises over the product of conditionally free with amalgamation random variables.

Let $(\a, \vartriangleleft, \g)$ be a Lie module. 
We list our contributions, first for the abstract part, 
\begin{enumerate}
\item We define a product $\#$ (see Section \ref{sec:gogroup}) and prove that its restriction to co-algebraic maps (this yields a group which we call the Guin--Oudom group) between the envelopes of $\a$ and $\g$ yields an algebraic Lie group, which turns out to be itself a $\mathcal{O}$-group. We give properties of the subset of Hopf $\mathcal{O}$-operators with respect to this product.
\item We show how to recover the Semenov--Tian--Shanskii (see Theorem \ref{thm:semenovtianshanskii}) factorisation from the commutativity of two Rota--Baxter--Hopf operators)
\item We define the notion of \emph{matching Hopf $\mathcal{O}$-operators} (see Section \ref{sec:MatchOope}) and provide several examples.
\end{enumerate}
Secondly, for applications to the operator-valued non-commutative probability theory:
\begin{enumerate}
\item We define the conditional $T$-transform and prove the transform factorises over the product of conditional free random variables in Theorem \ref{thm:factorisationconditionalttransform},
\item We define the conditional $H$-transform and prove the transform factorises over the product of conditional monotone random variables in Theorem \ref{thm:conditionalhtransform}
\item We state subordination equations in Theorem \ref{thm:subordinatoin}.
\end{enumerate}

\section{Background}
\label{sec:bckgnd}

\subsection{Notations}
\label{sec:notation}

In this work, the letters $H$ and $K$ consistently denote \emph{co-commutative Hopf algebras}. We use subscripts to distinguish between structural morphisms in the case where there is more than one Hopf algebra at stake, for example $\eta_H$ (unit), $\Delta_H$ (coproduct), $\varepsilon_H$ (counit),\ldots We use ${\rm Prim}(H)$ to denote the Lie algebra of the \emph{ primitive elements of a Hopf algebra $H$},
$$
{\rm Prim} (H) = \{ X \in H\colon \Delta_H(x)=x\otimes \eta_{H}+\eta_H \otimes x \}.
$$
We generally reserve small-case letters to denote "infinitesimal" objects, such as Lie--Rota--Baxter operators (defined further down) or Lie algebras. Uppercase letters denote "integrated" objects, such as universal envelopes, Hopf algebras, etc.


\subsection{Basic settings} 
\label{sec:setting}

For completeness, several basic definitions are recalled in the following.

\begin{enumerate}[\hspace{0.65cm} $\bullet$]

\item A \emph{right Lie module} is a triple $(\mathfrak{g},\vartriangleleft, \mathfrak{a})$ where

\begin{enumerate}

 \item $\mathfrak{g}$ and $\mathfrak{a}$ are Lie algebras,

 \item $\vartriangleleft$ is an action of the Lie algebra $\mathfrak{a}$ on the Lie algebra $\mathfrak{g}$ by derivations:
$$
\begin{aligned}
	[x,y]_{\g} \vartriangleleft a 
    &= [x \vartriangleleft a,y]_{\g} + [x,y \vartriangleleft a]_{\g},\quad
    x \vartriangleleft [a,b]_{\mathfrak{a}} 
    &= [(\cdot)\vartriangleleft a, (\cdot)\vartriangleleft b]_{\textrm{End}(\g)}(x),
\end{aligned}
$$
for any $x,y \in \g, a,b \in \mathfrak{a}$.
\end{enumerate}
A morphism between two right Lie modules $(\g,\vartriangleleft,\a)$ and $(\g^{\prime}, \vartriangleleft^{\prime},\a^{\prime})$ is a pair $(\varphi, \theta)$ of morphisms of Lie algebras $\varphi:\g\to\g^{\prime}$ and $\theta:\a\to\a^{\prime}$ such that for any $x\in \g$ and $a\in\a$:
$$
    \varphi(x) \vartriangleleft' \theta(a) = \varphi(x\vartriangleleft a).
$$

\item A \emph{right group module} is a triple $(G,\cdot,\Gamma)$ where
\begin{enumerate}
    \item $G$ and $\Gamma$ are groups,
    \item $\cdot$ is a right action of $\Gamma$ on $G$ by automorphisms:
        \begin{align*}
	&(xy) \cdot a 
        = (x\cdot a)(y\cdot a), \quad (x\cdot a) \cdot b = x \cdot (ab), \quad x \cdot 1 = x
        \end{align*}
for any $x,y \in G,~ a,b \in \Gamma$.
\end{enumerate}
A morphism between two right group modules $(G,\cdot,\Gamma)$ and $(G^{\prime}, \cdot^{\prime},\Gamma^{\prime})$ is a pair $(\Phi, \Theta)$ of group morphisms $\Phi:G\to G^{\prime}$ and $\Theta:\Gamma\to\Gamma^{\prime}$ such that for any $x \in G$ and $k \in \Gamma$:
$$
    \Phi(h) \cdot^{\prime} \Theta(k) = \Phi(h\cdot k).
$$    \item A \emph{right Hopf module} is a triple $(H,\preliediff,K)$ where
    \begin{enumerate}
	\item $(H,\Delta_{H},\eta_{H},\mathcal{S}_H)$ and $(K,\Delta_{K},\eta_{K},\mathcal{S}_K)$ are \emph{cocommutative} Hopf algebras.
	\item $\preliediff\colon H\otimes K \to H$ is a right Hopf action of $K$ on $H$, which means,  that $\preliediff$ is a morphism of coalgebras, using Sweedler's notations,
	\begin{align}
	(h_{(1)} \preliediff k_{(1)})\otimes(h_{(2)} \preliediff k_{(2)}) 
            &=(h\preliediff k)_{(1)} \otimes (h \preliediff k)_{(2)} \label{eqn:rightPost-Hopfone}\\
	\varepsilon_{H}(h \preliediff k) &=  \varepsilon_H(h)\varepsilon_K(k)\label{eqn:rightPost-Hopftwo}
	\end{align}
	and in addition, we require:
	    \begin{align}
	    (h \preliediff k_1) \preliediff k_2 &= h \preliediff (k_1k_2)  \label{eqn:rightPost-Hopfsix}\\
	(h_1h_2) \preliediff k &= (h_1 \preliediff k_{(1)})(h_2\preliediff k_{(2)}) \label{eqn:rightPost-Hopfthree}\\
        		h \preliediff 1 &= h \label{eqn:rightPost-Hopffour}\\ 
	(1 \preliediff k) &=  1 \preliediff \varepsilon_K(k)  	 \label{eqn:rightPost-Hopffive}
	\end{align}
	for any $h,h_1,h_2 \in H,~ k,k_1,k_2 \in K$.
    \end{enumerate}
\begin{remark}
The axioms of a right Hopf module together with \ref{eqn:rightPost-Hopfthree} imply that, for any $h_1, h_2 \in H$: 
$$
\antipode_H(h_1 \preliediff h_1) = \antipode_H(h_1) \preliediff h_2
$$
And also that the right multiplication map $\alpha_{\preliediff} \in {\rm End}(K,{\rm End}(H,H))$ is invertible respectively to the convolution product with inverse $\beta_{\preliediff}$ given by
\begin{equation}
\beta_{\preliediff}(k)(h) = h \preliediff S_{K}(k), \quad h,k \in H.
\end{equation}
It is not difficult to prove that our set our axioms is in fact equivalent to requiring that $\preliediff$ is a morphism of coalgebras that \eqref{eqn:rightPost-Hopfsix}, \eqref{eqn:rightPost-Hopfthree} hold and $\alpha_{\preliediff}$ is invertible \cite{catoire2024cartier}. 
\end{remark}
A morphism between right Hopf modules, $(H,\preliediff,K)$ and $(H^{\prime}, \preliediff^{\prime},K^{\prime})$, is a pair $(\Phi, \Theta)$ of Hopf algebra morphisms, $\Phi:H\to H^{\prime}$ and $\Theta:K\to K^{\prime}$, such that for any $h\in H$ and $k\in K$,
$$
    \Phi(h) \preliediff^{\prime} \Theta(k) 
    = \Phi(h \preliediff k).
$$
\end{enumerate}

We explain the notion of \emph{universal envelope of a right Lie module} $(\g,\vartriangleleft, \a)$. This is an object $(H,\preliediff, K)$ in the category of right Hopf modules together with an injection $(\iota_\g,\iota_\a)$ of right Lie modules:
$$
    \iota_\g:\g \to {\rm Prim}(H),
    \qquad
    ~\iota_\a :\a \to {\rm Prim}(K),
$$
which is the solution to the following universal problem: let $(H^{\prime},\preliediff,K^{\prime})$ be a right Hopf module and $(\varphi,\theta)$ a pair a morphism of right Lie modules, $\varphi : \g \to {\rm Prim}(H^{\prime})$ and $\theta : \a \to {\rm Prim}(K^{\prime})$. Then there exists a morphism $(\Phi,\Theta)$ of right Hopf modules, $\Phi:H\to H^{\prime}$ and $\Theta:K\to K^{\prime}$, such that
$$
    \Phi \circ \iota_{H} 
    = \varphi,
    \qquad 
    \Theta \circ \iota_{K} 
    = \theta.
$$
See Fig.~\ref{fig:universalproblem} for a diagrammatic formulation.
\begin{figure}
	\[\begin{tikzcd}
		H && {H^{\prime}} & K && {K^{\prime}} \\
		{{\rm Prim}(H)} && {{\rm Prim}(H^{\prime})} & {{\rm Prim}(K)} && {{\rm Prim}(K^{\prime})} \\
		& {\mathfrak{g}} &&& {\mathfrak{a}}
		\arrow["\Phi", from=1-1, to=1-3]
		\arrow["{\iota_{\mathfrak{g}}}", from=3-2, to=2-1]
		\arrow[from=2-1, to=1-1]
		\arrow[from=2-3, to=1-3]
		\arrow["\varphi"', from=3-2, to=2-3]
		\arrow["\Theta", from=1-4, to=1-6]
		\arrow[from=2-4, to=1-4]
		\arrow[from=2-6, to=1-6]
		\arrow["{\iota_{a}}", from=3-5, to=2-4]
		\arrow["\theta"', from=3-5, to=2-6]
	\end{tikzcd}\]
\caption{\label{fig:universalproblem} Envelope of a right Lie module.}
\end{figure}
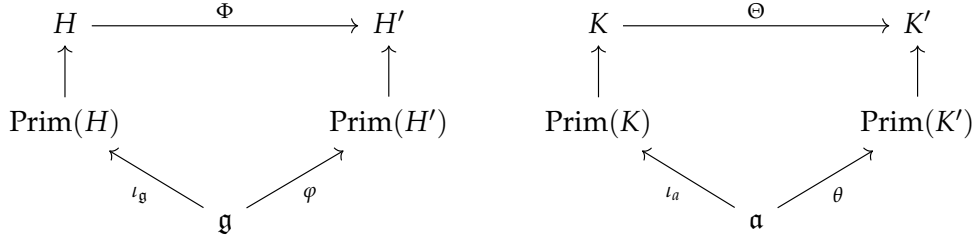
It is not difficult to build a universal envelope of a right Lie module $(\g,\vartriangleleft,\a)$. We call $\mathcal{U}(\g)$ and $\mathcal{U}(\a)$ the universal enveloping algebras of $\g$ and $\a$, respectively. We "extend/co-extend" the action of $\a$ on $\g$ to an action of $\mathcal{U}(\a)$ on $\mathcal{U}(\g)$. First, pick $a\in \a$ and set for any $x_1 \cdots x_n \in  \mathcal{U}(\g)$:
$$
    x_1\cdots x_n \preliediff a 
    := \sum_{i=1}^n x_1\cdots x_{i-1}(x_i \vartriangleleft a)x_{i+1}\cdots x_n.
$$
We need to check that this is well-defined, which amounts to verifying the equality
$$
    [x_1,x_2]_{\g} \preliediff a = x_1 x_2 \preliediff a - x_{2}x_1 \preliediff a.
$$
It is implied by the fact that $\a$ acts by derivations on $\g$. Notice that $\preliediff$ is a morphism from the Lie algebra $\a$ to the Lie algebra ${\rm End}_{\g}(\mathcal{U}(\g))$. Therefore, we extend $\preliediff$ to an algebra morphism from $\mathcal{U}(\a)$ to ${\rm End}(\mathcal{U}(\g))$. A simple induction shows that $X_1X_2 \preliediff A = (X_1 \preliediff A_{(1)})(X_2 \preliediff A_{(2)})$, for any $A \in \mathcal{U}(\a)$ and $X_1,X_2 \in \mathcal{U}(\g)$.

We let $(\varphi,\theta)$ be a morphism of right Lie modules, $\iota_{\mathcal{U}(\g)}$, $\iota_{H^{\prime}}$ and $\iota_{\mathcal{U}(\mathfrak{a})}$, $\iota_{K^{\prime}}$ are the canonical injections of the Lie algebras of primitive elements of $\mathcal{U}(\g)$, $H^{\prime}$, $\mathcal{U}(\a)$, and $K^{\prime}$, respectively. Since $\iota_{H^{\prime}} \circ \varphi$ and $\iota_{K^{\prime}} \circ \theta$  are Lie algebra morphisms, we extend them to algebra morphisms $\Phi: \mathcal{U}(\g)\to H^{\prime}$ and $\Theta: \mathcal{U}(\a)\to K'$. With these definitions, the diagrams in Fig.~\ref{fig:universalproblem} are commutative.

We denote by LieM, GroupM, and HopfM the categories of Lie modules, Group modules, and Hopf modules, respectively. The constructions exposed above produce a functor
$$
    \mathcal{U}: {\rm LieM} \to {\rm HopfM},
$$
which is the left adjoint of a functor ${\rm{HLM}}$ that associates with any Hopf module a Lie module via a restriction to primitive elements:
$$
    {\rm HLM} : {\rm HopfM} \to {\rm LieM}.
$$

\subsection{On the Semenov--Tian-Shanskii factorisation}
We refer the reader to \cite{semenov2003classical}.
We state in this section the Semenov--Tian--Shanskii factorisation for \emph{factorizable Lie bialgebras}.
We recall that a factorizable Lie bialgebra is the data of a Lie algebra $\g$ together with a tensor $\bm{r} \in g \otimes g$ such that: 
\begin{enumerate}
\item the symmetric tensor $I := \bm{r}_{12} + \bm{r}_{21}$ is \emph{invertible} and \emph{ad-invariant},
\item $\bm{r}$ satisfies the classical Yang--Baxter equation:
\begin{equation}
\label{eqn:cyb}
[r_{12},r_{13}] + [r_{12},r_{23}]+[r_{13},r_{23}]=0.
\end{equation}
\end{enumerate}
Recall that \eqref{eqn:cyb} is equivalent to:
\begin{equation*}
[\bm{r}(x^{\star}),\bm{r}(y^{\star})]_{\g}=\bm{r}(ad^{\star}(\bm{r}(x^{\star}))(y^{\star})-ad^{\star}(\bm{r}^{\star}(y))(x^{\star})).
\end{equation*}
for any $x^{\star},y^{\star} \in \g^{\star}$. The above equation implies that $[-,-]_{\star}\colon \g^{\star}\otimes\g^{\star}\to\g^{\star}$ defined by 
$$
[-,-]_{\star}(x^{\star},y^{\star})= ad^{\star}(\bm{r}(x^{\star}))(y^{\star})-ad^{\star}(\bm{r}^{\star}(y))(x^{\star}),~x^{\star},y^{\star} \in \g^{\star}
$$
is a Lie bracket, and therefore that $\bm{r}$ is a morphism of Lie algebras.
In the following, $\bm{r}_+ = \bm{r}$ and $\bm{r}_-=-r_{21}$. Therefore, $\bm{r}_+ = I + \bm{r}_-$ and $\bm{r}_-$ satisfy the classical Yang--Baxter equation.
We use the same symbols $\bm{r}_+$ (resp. $\bm{r}_-$) for the extension and coextension of $\bm{r}_+$ (resp. $\bm{r}_-$) to the envelope $\mathcal{U}(\g_+, [-,-]_+)$ (resp. $\mathcal{U}(\g),[-,-]_-$) with values in $\mathcal{U}(\g)$.
We denote by $G$ the group-like elements of $\mathcal{U}(\g)$ and $G_+$ (resp. $G_-$) the group-like elements of $\mathcal{U}(\g_+, [-,-]_+)$ (resp. $\mathcal{U}(\g_-, [-,-]_-)$)).

\begin{theorem}[Semenov-Tian-Shanskii factorisations \cite{semenov2003classical}]
Let $(\mathfrak{g},\bm{r})$ be a factorizable Lie bialgebra. 
\begin{enumerate}
\item Then, for any $x \in \g$, there exists a unique couple $(x_-,x_+)$ with $x_- \in {\rm Im}(\bm{r}_-)$ and $x_+ \in {\rm Im}(\bm{r}_+)$ such that:
$$
x = x_+ - x_-.
$$
\item Then any $X \in \mathcal{U}(\g)$ admits an unique representation:
$$
X = \sum_{i} X^{(i)}_+ \otimes \mathcal{S}_{\mathcal{U}(\g)}(X^{(i)}_-),
$$
such that:
$$
\sum_i X^{(i)}_+ \otimes X^{(i)}_- =  (\bm{r}_+ \otimes \bm{r}_-) \circ \Delta \circ (I^{-1}(X)).
$$
\item Let $g \in G$ be a group-like element in $\mathcal{U}(\g)$. Then $g$ admits a unique representation:
$$
g = g_+g_-^{-1},
$$
where $g_+ \in \tilde{G}_+={\rm Im}(\bm{r}_+ : G_+\to G)$ and $g_- \in \tilde{G}_- = {\rm Im}(\bm{r}_-: G_- \to G)$.
\end{enumerate}
\end{theorem}
\section{Kupershmidt's $\mathcal{O}$-operators and beyond}
\label{sec:Kupershmidt}
We recall the definitions of an Lie $\mathcal{O}$-operator on a Lie module $(\g, \vartriangleleft, \a)$ \cite{Kupershmidt1999} and that of a group $\mathcal{O}$-operator \cite{guo2021integration} on a right group module $(G,\cdot,\Gamma)$ (see the previous section). We also \emph{formalize} the notion of Hopf $\mathcal{O}$-operators  in the context of a right Hopf module $(H,\preliediff, K)$ already suggested in \cite{goncharov2021rota}. 
\subsection{Types of  $\mathcal{O}$-operators}
\label{sec:typesooperators}
\begin{definition}[Lie $\mathcal{O}$-operator \cite{Kupershmidt1999}]
\label{def:LieOop}
Let $(\mathfrak{g}, \vartriangleleft, \mathfrak{a})$ be a right Lie module. We call a Lie $\mathcal{O}$-operator a map ${\bf t} : \mathfrak{g} \to \mathfrak{a}$ satisfying for any $x,y \in \g$:
\begin{equation}
\label{eqn:olie}
    [{\bf t}(x),{\bf t}(y)]_{\mathfrak{a}} = {\bf t}(x \vartriangleleft {\bf t}(y)) - {\bf t}(y \vartriangleleft {\bf t}(x)) + {\bf t}([x,y]_\mathfrak{g}).
\end{equation}
We call \emph{$\mathcal{O}$ Lie algebra} the tuple $(\mathfrak{g}, \vartriangleleft, \mathfrak{a}, {\bf t})$. A morphism between two Lie $\mathcal{O}$-operator algebras, $(\mathfrak{g}, \vartriangleleft, \mathfrak{a}, {\bf t})$ and $(\mathfrak{g}^{\prime},\vartriangleleft^{\prime}, \mathfrak{g}^{\prime}, {\bf t}^{\prime})$, is a pair $(\varphi,\theta)$ of Lie algebra morphisms, $\varphi: \g \to \g^{\prime}$ and $\theta : \a \to a^{\prime}$, such that:
$$
    \theta \circ {\bf t} 
    = {\bf t^{\prime}} \circ \varphi.
$$
\end{definition}

\begin{remark}
Note that equation \eqref{eqn:olie} is not linear; the sum of two Lie $\mathcal{O}$-operators is, in general, not a Lie $\mathcal{O}$-operator.
\end{remark}

\begin{example}
If $\mathfrak{a} = \mathfrak{g}$ and $\vartriangleleft$ is the adjoint action of the Lie algebra $\g$ over itself:
\begin{equation*}
    x \vartriangleleft y =[x,y],~x,y \in \g,
\end{equation*}
then a Lie $\mathcal{O}$-operator is a Rota--Baxter operator (of weight one) on the Lie algebra $\g$.
\end{example}

\begin{definition}[group $\mathcal{O}$-operator \cite{guo2021integration}] Let $(G,\cdot, \Gamma)$ be a right group module. We call \emph{an group $\mathcal{O}$-operator } a map ${\bf T}:G \to \Gamma$ satisfying for any $h,g \in G$,
\begin{equation}
    {\bf T}(g){\bf T}(h) 
    = {\bf T}((g \cdot {\bf T}(h))h).
\end{equation}
We call $\mathcal{O}$ group the tuple $(G,\cdot,\Gamma, {\bf T})$. A morphism between two $\mathcal{O}$-groups, $(G,\cdot,\Gamma, {\bf T})$ and $(G^{\prime},\cdot^{\prime},\Gamma^{\prime}, {\bf T}^{\prime})$, is a pair $(\Phi,\Theta)$ of group morphisms $\Phi : G \to G^{\prime}$ and $\Theta : \Gamma \to \Gamma^{\prime}$ such that
$$
    \Theta \circ {\bf T}^{\prime} 
    = {\bf T}\circ \Phi.
$$
\end{definition}

\begin{example}
If $\Gamma = G$ and $\cdot$ is the right adjoint action $Ad$ of $G$ on itself, then an group $\mathcal{O}$-operator  is a Rota--Baxter group operator and $(G,Ad,G,{\bf T})$ is a Rota--Baxter group (of weight one) \cite{guo2021integration}.
\end{example}

\begin{definition}[Hopf $\mathcal{O}$-operator \cite{li2022post}]
Let $(H, \preliediff, K)$ be a right Hopf module. We call an Hopf $\mathcal{O}$-operator a \emph{coalgebra map} ${\bf T}: H \to K$ satisfying for any $e,f \in H$,
\begin{equation}
\label{def:Ohopfoperator}
    {\bf T}(e){\bf T}(f)
    ={\bf T}((e \preliediff {\bf T}(f_{(1)}))f_{(2)}).
\end{equation}
and 
\begin{equation}
\varepsilon_{H} = \varepsilon_{K} \circ \bT 
\end{equation}
We call  $\mathcal{O}$ Hopf algebra the tuple $(H,\preliediff, K, {\bf T})$. A morphism between two Hopf $\mathcal{O}$ algebras $(H,\preliediff,K, {\bf T})$ and $(H^{\prime},\preliediff^{\prime},K^{\prime}, {\bf T}^{\prime})$, is a pair $(\Phi,\Theta)$ of Hopf algebra morphisms, $\Phi : H \to H^{\prime}$ and $\Theta : K \to K^{\prime}$, such that
$$
    \Theta \circ {\bf T}^{\prime} 
    = {\bf T}\circ \Phi.
$$
\end{definition}
\subsection{Adjacent structures}
\label{ssec:adjstructure}
We propose here to review algebraic structures adjacent to an $\mathcal{O}$-operator of any type, that is, Lie, Hopf, group. Recall first the definition of a post-Lie algebra; see, for example, \cite{curry2019post}.
\begin{definition} 
\label{def:post-Lie}
A post-Lie algebra $(\g,\vartriangleleft)$ consists of a Lie algebra $(\g,[-,-])$ and a binary operation $\vartriangleleft: \g \otimes \g \to \g$ satisfying for any $x,y,z \in \g$:
\begin{align}
    [x,y]\vartriangleleft z 
    &= [x\vartriangleleft z,y]+[x,y\vartriangleleft z] \label{post-Lie1}\\ 
    x \vartriangleleft [y,z]
    &= (x \vartriangleleft y) \vartriangleleft z - x \vartriangleleft (y\vartriangleleft z) 
    - ((x \vartriangleleft z) \vartriangleleft y - x \vartriangleleft (z \vartriangleleft y)) \label{post-Lie2}
\end{align}
\end{definition}

The following proposition is well known.
\begin{proposition}
\label{prop:post-Lie}
	Let ${\bf t} : \g \to \a$ be an Lie $\mathcal{O}$-operator over $(\g,\vartriangleleft,\a)$. The binary operation $\vartriangleleft_{\bf t}\colon \g\otimes\g \to \g$ defined by  
	\begin{equation*}
	x \vartriangleleft_{\bf t} y = x \vartriangleleft {\bf t}(y), \quad x,y \in \g,
	\end{equation*}
  yields a post-Lie algebra $(\g,\vartriangleleft_{\bf t})$.
\end{proposition}
The post-Lie axioms imply that the bracket defined by
$$
    [x,y]_{\bf t} 
    := x \vartriangleleft_{\bf t} y - y \vartriangleleft_{\bf t} x  + [x,y]_{\g},\quad x,y \in \g,
$$
 is a Lie bracket. We denote by $\g_{\bf t}$ the Lie algebra with the same vector space as $\g$ but with bracket $[-,-]_{\bf t}$. Then, $(\g, \vartriangleleft_{\bf t})$ is post-Lie if and only if $(\g,\vartriangleleft_{\bf t}, \g_{\bf t})$ is a right Lie module. In that case,
$$
    ({\rm id}_{\g},{\bf t}) : (\g,\vartriangleleft_{\bf t},\g_{\bf t}) \to (\g,\vartriangleleft,\a)
$$
is a morphism of right Lie modules. 
The converse statement of Proposition \ref{prop:post-Lie} is also true; any post-Lie algebra is a right Lie module of the form $(\g, \vartriangleleft_{\bf t}, \g_{\mathfrak{t}})$ where ${\bf t} = \rm{id}_{\mathfrak{g}}$. It will be convenient to call such a triple $(\g, \vartriangleleft_{\bf t}, \g_{\bf t})$ \emph{ a right post-Lie module}. 
Now we describe structures adjacent to a Hopf $\mathcal{O}$-operator $\bT$ on a right Hopf module $(H,\preliediff,K)$. See \cite{li2022post} for the proofs of the following propositions.

\begin{proposition}
\label{prop:ohopfm}
Consider the right Hopf module $(H,\preliediff,K)$ with Hopf $\mathcal{O}$-operator $\bT$.
Define the bilinear operation $\star_{\bT}$ between elements of $H$ by, for any $e,f \in H$:
\begin{equation}
\label{eqn:productohopfoperator}
    e \star_{\bf T} f 
    := (e \preliediff {\bf T} (f_{(1)}))f_{(2)}.
\end{equation}
Then $\star_{\bf T}$ is associative and the coalgebra $H_{\bf T} := (H, \star_{\bT}, \eta_{H}, \Delta_{H}, \varepsilon_{H})$ is a Hopf algebra. We call $\mathcal{S}_{\bf T}$ its antipode. Additionally, the binary operation $\vartriangleleft_{\bT}\colon H \otimes H \to H$ defined by
\begin{equation*}
    e \vartriangleleft_{\bf T} f = e \preliediff {\bf T}(f),\quad e,f \in H
\end{equation*}
yields a right Hopf module $(H,\preliediff_{\bT},H_{\bT})$ and
$$
    ({\rm id_H,\bT}) : (H,\preliediff_{\bT},H_{\bT}) \to (H, \preliediff, K)
$$
is a morphism between right Hopf modules.
\end{proposition}
\begin{proof}
The only difficult part is the existence of the antipode $\bS_{\bT}$. To that end, we will show that the right translation $R\colon H\to {\rm End}(H,H)$ is invertible as an elements of ${\rm End}(H, \End(H,H))$ endowed with the convolution product. In fact, set for any $E,F \in H$, $\beta_{\preliediff E}(F) =  F\preliediff \antipode_{K} \circ \bT(E)$. Then 
\begin{align*}
\beta_{\preliediff E_{(1)}} \circ \alpha_{\preliediff E_{(2)}} (F) &= F \preliediff \bT(E_{(1)}) \antipode_{K} \bT(E_{(2)}) \\
&= F\preliediff \bT(E)_{(1)}\antipode_{K}\bT(E)_{(2)} \\
&=F\preliediff \varepsilon_{K}\circ \bT(E) \\
&=F\preliediff \varepsilon_{H} \\
&=\varepsilon_{H}
\end{align*}
The relation $\alpha_{\preliediff E_{(1)}} \circ \beta_{\preliediff E_{(2)}}$ follows fron the same computations. We use the Theorem 2.5 in \cite{li2022post}.
\end{proof}
\begin{remark} 
A remark about the terminology: the Hopf algebra $H_{\bT}=(H, \star_{\bT},\eta_{H}, \Delta, \varepsilon, S_{\bT})$ is called the sub-adjacent Hopf algebra \cite{jiang2021lie, li2024sub} to the Post-Hopf algebra $(H, \cdot, \Delta, \varepsilon, \antipode_{H}, \preliediff_{\bT})$ as defined in \cite{li2024sub}.  
\end{remark}

It goes also by the name of $D$-algebra; see, for instance, \cite{al2022algebraic}. 

\begin{proposition}Let ${\bf T} : G \to \Gamma$ be an group $\mathcal{O}$-operator over $(G,\cdot,\Gamma)$. Define for any $f,e \in G$ the product
$$
    f \star_{\bT} e 
    := (f\cdot {\bf T}(e))e.
$$
Then $\star_{\bT}$ is associative and gives a group $(G,\cdot_{\bT})$.
Additionally, the operation $\cdot_{\bf T}\colon G \times G \to G$ defined for any $f,e \in G$ by:
\begin{equation*}
    f \cdot_{\bT} e 
    = f \cdot \bT(e)
\end{equation*}
yields a right group module $(G,\cdot_{\bT},G_{\bT})$ and
$$
    ({\rm id_G}, {\bf T}) : (G,\cdot_{\bT},G_{\bT}) \to (G, \cdot, \Gamma).
$$
is a morphism of right group-modules. We call the tuple $(G, \Gamma, \cdot_{\bT})$ a \emph{post-Group}.
\end{proposition}
\begin{remark}
    The inverse of an element $g\in G_{\bT}$ is $g^{-1}\cdot \bT(g)^{-1}$. In the remaining part of the paper, to lift ambiguities, we will refer to the product in the notations for inverses, for example
    $$
    g^{-1^{\star_{\bT}}} = g^{-1^{\centerdot}}\cdot \bT(g)^{-1^{\circ}}
    $$
    with $\centerdot$ the internal product on $G$ and $\circ$ the internal product on $\Gamma$. We will also allow ourselves to forget replace $\star_{\bT}$ by $\bT$ to keep the notations contained.
\end{remark}

We explain how a Lie $\mathcal{O}$-operator ${\bf t}$ over a right Lie module $(\g,\vartriangleleft,\mathfrak{a})$ yields an Hopf $\mathcal{O}$-operator ${\bf T}$ on the universal envelope $(\mathcal{U}(\g),\preliediff, \mathcal{U}(\mathfrak{a}))$ of the Lie module $(\g,\vartriangleleft,\a)$. Given $E \in \mathcal{U}(\g)$ and $h \in \g$,
we define the map $[{\bf t}] : \mathcal{U}(\g) \to \mathcal{U}(\a)$ recursively by the system:
\begin{align}
\tag{GO}
\begin{aligned}
\label{eqn:RBop}
    &[{\bf t}](Eh) := [{\bf t}](E) \mathbf{t}(h) - [{\bf t}](E\preliediff \mathbf{t}(h)), \\
    &[{\bf t}](h)  := \mathbf{t}(h) \\
    &\bf T(\mathbf{1})=\mathbf{1}.
\end{aligned}
\end{align}
Inspired  \cite{oudom2008lie}, we refer to \eqref{eqn:RBop} as the \emph{ Guin--Oudom recursion} with initial values ${\bf t}$.

\begin{remark} Note that the above recursion is well-defined only for Lie $\mathcal{O}$-operators of weight $1$. In particular, it is not possible to extend to $\mathcal{U}(\g)$ a Lie $\mathcal{O}$-operator of weight $\lambda \in \mathbb{R}$ by using this recursion.
\end{remark}

\begin{proposition}
\label{prop:coalgebrarelative}
For any Lie $\mathcal{O}$-operator $\mathbf{t}$ on the right Lie module $(\mathfrak{g},\vartriangleleft, \mathfrak{a})$ the operator $\bf T$ obtained by the Guin--Oudom recursion \eqref{eqn:RBop} with initial condition ${\bf t}$ is a coalgebra morphism.
\end{proposition}

\begin{proof}
We proceed by induction. For an integer $N \geq 1$, we denote by $\mathcal{U}(\g)^{(\leq N)}$ the image of $\bigoplus_{k \leq N} V^{\otimes k} \subset T(\g)$ by the natural projection. Since $\bf T$ equals $\mathbf{t}$ on primitive elements, we see that for any $h \in \g$, $[{\bf t}] \otimes [{\bf t}] \circ \Delta(h) = \Delta [{\bf t}](h)$. Let $n \geq 1$. We assume that $\bf T$ satisfies for any $E \in \mathcal{U}(\g)^{(\leq N)}:$
\begin{equation*}
    [{\bf t}] \otimes [{\bf t}] \circ \Delta (E) 
    = {\Delta} [{\bf t}](E).
\end{equation*}
Now pick a -- primitive -- element $h \in \g$, then
\begin{align*}
    \Delta [{\bf t}](Eh)
    &=\Delta\big([{\bf t}](E) [{\bf t}](h)-[{\bf t}](E\preliediff [{\bf t}](h)) \big) \\
    &=[{\bf t}](E_{(1)}) [{\bf t}](h)_{(1)} \otimes ([{\bf t}](E_{(2)}) [{\bf t}](h)_{(2)}) - [{\bf t}](E_{(1)}\preliediff [{\bf t}](h)_{(1)}) \otimes [{\bf t}](E_{(2)}\preliediff [{\bf t}](h)_{(2)}) \\
    &=[{\bf t}](E_{(1)}) [{\bf t}](h)\otimes [{\bf t}](E_{(2)}) + [{\bf t}](E_{(1)})\otimes [{\bf t}](E_{(2)}) [{\bf t}](h) \\
    &\hspace{2cm}-\big([{\bf t}](E_{(1)} \preliediff [{\bf t}](h)) \otimes [{\bf t}](E_{(2)}) + [{\bf t}](E_{(1)})\otimes [{\bf t}](E_{(2)}\preliediff [{\bf t}](h))\big)\\
    &=\big([{\bf t}](E_{(1)}) {\bf t}(h)-[{\bf t}](E_{(1)} \preliediff {\bf t}(h))\big)\otimes [{\bf t}](E_{(2)})\\
    &\hspace{2cm} + [{\bf t}](E_{(1)})\otimes \big([{\bf t}](E_{(2)}) {\bf t}(h)-[{\bf t}](E_{(2)}\preliediff {\bf t}(h))\big) \\
    &= {[{\bf t}]} \otimes {[{\bf t}]} \circ \Delta (Eh).
    \end{align*}
\end{proof}

\begin{proposition}
\label{prop:hopffromgo}
Let $\mathbf{t}$ be an Lie $\mathcal{O}$-operator on the Lie module $(\g,\vartriangleleft,\mathfrak{a})$. Then the operator $[{\bf t}]$ following from the Guin--Oudom recursion \eqref{eqn:RBop} is a Hopf $\mathcal{O}$-operator on $(\mathcal{U}(\g), \preliediff, \mathcal{U}(\mathfrak{a}))$:
\begin{equation}
\label{eqn:secondcompat}
    [{\bf t}](E)[{\bf t}](F) = [{\bf t}](E \star_{[{\bf t}]} F),
    \quad E,F \in \univalgebra.
    \end{equation}
\end{proposition}

\begin{proof}
We set for $E\otimes F \in \mathcal{U}(\g) \otimes \mathcal{U}(\g)$
\begin{align*}
    \alpha(E\otimes F) := [{\bf t}](E) [{\bf t}](F),
    \quad
    \beta(E\otimes F) := [{\bf t}](E\star_{[{\bf t}]}F)
\end{align*}
We proceed by induction. The formula holds whenever $F$ is primitive. We assume $\alpha(E\otimes F) = \beta(E\otimes F)$ whenever $F \in \mathcal{U}(\g)^{\leq (N)}$ for a certain $N\geq 1$. Pick $h \in \g$ and $E,F \in Ker(\varepsilon)$. Then
\begin{align*}
    [{\bf t}](E\star_{[{\bf t}]}(Fh))
    & = [{\bf t}](E\star_{[{\bf t}]}(F\star_{[{\bf t}]}h)) - [{\bf t}](E\star_{[{\bf t}]}(F\preliediff \mathbf{t}(h))) \\
    &=[{\bf t}](E\star_[{\bf t}] F) {\bf t}(h) - [{\bf t}](E)[{\bf t}](F\preliediff \mathbf{t}(h))\\
    &=[{\bf t}](E) \big([{\bf t}](F)[{\bf t}](h) - [{\bf t}](F\preliediff {\bf t}(h)) \big)\\
    &=[{\bf t}](E)\preliediff [{\bf t}](F\star_{[{\bf t}]}h) - [{\bf t}](F\preliediff \mathbf{t}(h))\\
    &=[{\bf t}](E)[{\bf t}](Fh).
\end{align*}
\end{proof}

\begin{corollaire} 
Any Hopf $\mathcal{O}$-operator on $(\mathcal{U}(\g), \preliediff, \mathcal{U}(\mathfrak{a}))$ satisfies the recursion \eqref{eqn:RBop} with $\bf t$ the restriction of ${\bf T}$ to $\g$; $[\bf t] = \bT$. Therefore, two Hopf $\mathcal{O}$-operators coincide if their restrictions to $\g$  coincide.
\end{corollaire}
Following the above corollary, we can forget about the brackets $[~]$ for the extension of a Lie $\mathcal{O}$-operator ${\bf t}$.
\begin{corollaire}
\label{cor:envelope}
    The envelope of the Lie algebra $\mathfrak{g}_{\bf t}$ is isomorphic to $H_{\bT}$
\end{corollaire}
\begin{proof} The proof follows from \cite{EBRAHIMIFARD201719} and the observation that, with $\bT$ the Hopf $\mathcal{O}$-operator obtained by unfolding \ref{eqn:RBop} and $\star_{\bT}$ and $\vartriangleleft_{\bT}$ defined as above, for any $E,F \in \mathcal{U}(\g)$ and $f\in \g$
\begin{align*}
    &E \star_{\bT}(F) = E \vartriangleleft_{\bT}(F_{(1)}) F_{(2)} \\ 
    & E \vartriangleleft_{\bT} (Ff) = (E \vartriangleleft_{\bT} F) \vartriangleleft_{\bT}f)-E \vartriangleleft_{\bT} (F \vartriangleleft_{\bT}f)
\end{align*}
\end{proof}
\begin{remark}
Note that the recursion \eqref{eqn:RBop} is equivalent to the closed formula proved in Theorem 5.2 of \cite{li2022post} (the operator $\bar{T}$ in the notations of \cite{li2022post} is equal to $[\bf t]$) (Propositions \ref{cor:envelope},\ref{prop:hopffromgo} follows from this Theorem also).
\end{remark}
\subsection{On the antipode of a Post-Hopf algebra}
\label{sec:antipodePost-Hopf}
We fix a Lie algebra module $(\g,\vartriangleleft,\a)$ for the entire section. Uppercase letters will be used for co-algebra morphisms, e.g $T\colon\mathcal{U}(\g)\to\mathcal{U}(\a)$. For the remaining part of the paper, bold letters ($\bm{t}, \bm{T}$ \ldots) will be used to denote $\mathcal{O}$-operators. In this Section, we define a linear mapping between coalgebra morphisms from $\mathcal{U}(\g)$ to $\mathcal{U}(\a)$ and coalgebra diffeomorphisms of $\mathcal{U}(\g)$. We pinpoint a fixed-point equation satisfied by the antipode of ${\mathcal{U}(\mathfrak{g})}_{\bf T}$ when $\bm{T}$ is a Hopf $\mathcal{O}$-operator.

 \begin{definition}[Transform $e$]
Let $T\coloni \mathcal{U}(\g)\to \mathcal{U}(\a)$ be a coalgebra morphism. We define $e_{T} \colon \mathcal{U}(\g) \to \mathcal{U}(\g)$ as the unique linear map satisfying:
\begin{equation*}
    \varepsilon_{\mathcal{U}(\g)} \circ e_{T} 
    = \varepsilon_{\mathcal{U}(\g)},
\end{equation*}
and for any element $f \in \mathcal{U}(\g)$:
	\begin{align*}
	&e_{T}(f) = f_{(1)} \preliediff T \circ e_{T}(f_{(2)}),
	\end{align*}
using Sweedler's notation $\Delta_{ \mathcal{U}(\g)}(f)=f_{(1)} \otimes f_{(2)}$. 
\end{definition}
\begin{example}
Let $x,y,z \in\g$ primitive elements of $\univalgebra$ then
\begin{align*}
    e_{T}(x) 
    &= x\\
    e_{T}(xy)
    &=xy + x \preliediff T e_T (y) + y \preliediff T\circ e_T(x) \\ 
    &= xy  
        + x \preliediff {T}(y) 
        + y \preliediff { T}(x).\\
	e_{T}(xyz) 
    &= xyz 
        + xy \preliediff {T}(z)    
        + xz \preliediff {T}(y) 
        + yz \preliediff {T}(x)\\ 
    & \quad
        + x \preliediff  Te_T(yz) 
        + y \preliediff {T}e_T(xz) 
        + z \preliediff {T}e_T(xy).
\end{align*}
When $T$(=$\bT$) is a Hopf $\mathcal{O}$-operator, implying that 
$$
    \bT(xy)=\bT(x)\bT(y)-\bT(x\preliediff \bT(y)),
$$
one gets:
\begin{align*}
    e_\bT(xyz)
    &= xyz
        +(x \preliediff {\bf T}(z))y 
        +x(y \preliediff {\bf T}(z)) \\
    &\quad + (x \preliediff {\bf T}(y))z 
        + x (z \preliediff {\bf T}(y)) \\
    &\quad + (y \preliediff {\bf T}(x))z 
        + y(z \preliediff {\bf T}(x)) \\
    &\quad + (x \preliediff {\bf T}(y)) \preliediff {\bf T}(z) 
        + x \preliediff {\bf T}(z \preliediff {\bf T}(y)) \\
    &\quad + (y \preliediff {\bf T}(x)) \preliediff {\bf T}(z) 
    + y \preliediff {\bf T}(z \preliediff {\bf T}(x)) \\
    &\quad + (z \preliediff {\bf T}(x)) \preliediff {\bf T}(y) 
     + z \preliediff {\bf T}(y \preliediff {\bf T}(x)).
\end{align*}
If $T$ is only assumed to be a coalgebra morphism, one gets instead:
\begin{align*}
    e_T(xyz) 
    &= xyz 
        +(x \preliediff { T}(z))y 
        +x(Y \preliediff {T}(z)) \\
	&\quad + (x \preliediff {T}(y))z 
        + x (z \preliediff {T}(y)) \\
	&\quad + (y \preliediff {T}(x))z 
        + y(z \preliediff {T}(x)) \\
        &\quad +(x \preliediff {T}(yz)+x\preliediff T(y \preliediff T(z)) 
        + x \preliediff T(z \preliediff T(y))  \\
        &\quad + (y \preliediff {T}(xz)+ y \preliediff T(x\preliediff T(z)) 
        + x \preliediff  T(z \preliediff  T(y)) \\
        &\quad + (z \preliediff {T}(xy)+z\preliediff  T(x\preliediff T(y)) 
        + x \preliediff  T(z \preliediff T(y)) .
\end{align*}
\end{example}
In the following propositions, we collect the basic properties of the transform $e_T$ of a coalgebra morphism.

\begin{proposition} The map $e_{T}$ is an isomorphism of coalgebras.
\end{proposition}

\begin{proof}
For any $f \in \mathcal{U}(\g)$, one has: 
$$ 
    e_{T}(f) = f_{(1)} \preliediff T\circ e_{T}(f_{(2)})
    =(\text{id} \preliediff T\circ e_{T}) \Delta(f).
$$
The strategy of the proof consists in showing that
$e_{T}(f)_{(1)} \otimes e_{T}(f)_{(2)}$ and $e_{T}(f_{(1)})\otimes e_{T}(f_{(2)})$ satisfy the same fixed point equation (this will be sufficient since both equal $1\otimes 1$ if $f=1$).
To this end, we should extend $\preliediff $ to $\mathcal{U}(\g)\otimes \mathcal{U}(\g)$ by setting
$$
    f \otimes g \preliediff  c\otimes d = f \preliediff  c \otimes g \preliediff d, \quad  f,g,c,d \in \univalgebra.
$$
From $[f \preliediff T(g)]_{(1)} \otimes [f \preliediff T(g)]_{(2)}$ = $(f_{(1)} \preliediff T\circ g_{(1)}) \otimes (f_{(2)} \preliediff T(g_{(2)}))$,  for $f,g \in \univalgebra$, it is not difficult to see that
$$
    \alpha \coloni f \mapsto e_{T}(f)_{(1)} \otimes e_{T}(f)_{(2)}
$$
solves the fixed point equation:
$$
    \alpha (f) 
    = (f_{(1)} \otimes f_{(2)}) \preliediff T\circ 	\alpha (f_{(3)}).
$$
In fact, for any $f \in \univalgebra$:
\begin{align*}
    \alpha (f)
    =e_{T}(f)_{(1)} \otimes e_{ T}(f)_{(2)} 
    &= [ f_{(1)} \preliediff T\circ e_{ T}(f_{(2)})]_{(1)} \otimes  [f_{(1)} \preliediff T\circ e_{T}(f_{(2)})]_{(2)}\\
    &=f_{(1)} \otimes f_{(2)} \preliediff T\circ e_{T}(f_{(3)})_{(1)} \otimes e_{T}(f_{(3)})_{(2)} \\
    &= (\text{id} \preliediff T \circ \alpha)(\Delta \otimes \text{id})\Delta(f).
\end{align*}
On the other hand, from the cocommutativity of $\univalgebra$, we get:
\begin{align*}
    e_{T}(f_{(1)}) \otimes e_{T}(f_{(2)}) 
   &= [ f_{(1)} \preliediff T\circ e_{T}(f_{(2)})] \otimes  [f_{(3)} \preliediff T\circ e_{T}(f_{(4)})]\\
    &=f_{(1)} \otimes f_{(3)} \preliediff T\circ e_{T}(f_{(2)}) \otimes e_{T}(f_{(4)}) \\
    &=f_{(1)} \otimes f_{(2)} \preliediff T\circ e_{T}(f_{(3)}) \otimes e_{T}(f_{(4)}).
\end{align*}
This ends the proof.
\end{proof}
In the following proposition, we denote by $T^{-1^*}$ the inverse the inverse of the coalgebra morphism $T$ with respect to the convolution product $*$ .
\begin{proposition}
\label{prop:inversett}
 Let ${T} : \mathcal{U}(\g) \to \mathcal{U}(\a)$ be a coalgebra morphism. The inverse of $e_T$ for the compositional product $\circ$ on ${\rm Hom}_{{\rm Vect}}(\mathcal{U}(\g),\mathcal{U}(\g))$ admits the following closed formula:
\begin{align*}
    e_{{T}}^{-1} 
    = \mathrm{id} \preliediff {T}^{-1^{*}}.
\end{align*}
\end{proposition}

\begin{proof}
One has
\begin{align*}
    (\mathrm{id}\preliediff T^{-1^{*}})\circ e_{{T}} 
    &= e_{{T}} \preliediff {T}^{-1^*} e_{{ T}} \\
    &=\mathrm{id} \preliediff {T} e_{{T}} * {T}^{-1^{*}}e_{{T}} \\
    &= \mathrm{id}\preliediff ({T} * {T}^{-1^*} )\circ e_{T} \\
    &=\mathrm{id} \preliediff \eta_{\mathcal{U}(\a)}\circ\varepsilon_{\mathcal{U}(\g)} \\
    &=\mathrm{id}.
    \end{align*}
\end{proof}

\begin{theorem}
\label{thm:inversion} Let ${\bf T}: \mathcal{U}(\g) \to \mathcal{U}(\a)$ be a Hopf $\mathcal{O}$-operator, then $$ \mathrm{id}\preliediff {\bf T}^{-1^{*}} 
    = e_{{\bf T}}^{-1} = \antipode_{\mathcal{U}(\a)} \circ e_{\bf T} \circ \antipode_{\mathcal{U}(\g)} = e_{{\bf T} \circ \mathcal{S}_{\mathcal{U}(\g)}}.
$$
\end{theorem}

\begin{proof}
It is enough to prove $e_{\bT \circ \antipode_{\mathcal{U}(\g)}} = e_{\bT}^{-1}$. We compute $e_{\bT} \circ e_{\bT \circ \antipode_{\mathcal{U}(\g)}}$.
\begin{align*}
    e_{\bT} \circ e_{\bT \circ \antipode_{\mathcal{U}(\g)}} 
    &= e_{\bT \circ \antipode_{\mathcal{U}(\g)}} \preliediff (\bT \circ e_{\bT}) * e_{\bT \circ \antipode_{\mathcal{U}(\g)}} \\
    &= \mathrm{id} \preliediff(\bT \circ \antipode_{\mathcal{U}(\g)} \circ e_{\bT\circ \antipode_{\mathcal{U}(\g)}}) * (\bT \circ e_{\bT} \circ e_{\bT \circ \antipode_{\mathcal{U}(\g)}}) \\
    &= \mathrm{id} \preliediff \bT \circ (\antipode_{\mathcal{U}(\g)} *_{\bT} e_{\bT}) \circ e_{\bT \circ \antipode_{\mathcal{U}(\g)}} \\
    &= \mathrm{id} \preliediff \bT \circ ((\antipode_{\mathcal{U}(\g)} \preliediff e_{\bT})\cdot e_{\bT}) \circ e_{\bT \circ \antipode_{\mathcal{U}(\g)}} \\
    &= \mathrm{id} \preliediff \bT \circ ((\antipode_{\mathcal{U}(\g)} (\mathrm{id} \preliediff e_{\bT}))\cdot e_{\bT}) \circ e_{\bT \circ \antipode_{\mathcal{U}(\g)}}, \\
    &= \mathrm{id} \preliediff \bT \circ ((\antipode_{\mathcal{U}(\g)}\circ e_{\bT}) \cdot e_{\bT}) \circ e_{\bT \circ \antipode_{\mathcal{U}(\g)}  } \\
    &=\mathrm{id} \preliediff \bT \circ \eta_{\mathcal{U}(\a))}\circ\varepsilon_{\mathcal{U}(\g)} = \mathrm{id} \preliediff \eta_{\mathcal{U}(\a)}\circ\varepsilon_{\mathcal{U}(\g)} = \mathrm{id}.
\end{align*}
\end{proof}

\begin{theorem}
\label{thm:tttransformtwoantipodes}
Let ${\bf T}: \mathcal{U}(\g) \to \mathcal{U}(\a)$ be an Hopf $\mathcal{O}$-operator. The antipode $\mathcal{S}_{H_{{\bT}}}$ of the Hopf algebra $H_{{\bT}}$ satisfies
\begin{equation}
\label{eqn:antipodebt}
    \mathcal{S}_{H_{{\bf T}}} = \mathcal{S}_{\mathcal{U}(\g)}\circ e_{{\bf T} \circ \antipode_{\mathcal{U}(\g)}} = e_{\bf T}\circ \antipode_{\mathcal{U}(\g)}   
\end{equation}
\end{theorem}
\begin{proof} It is straightforward computations.
\begin{align*}
{\rm id} \star_{\bf T} (e_{\bT} \circ \antipode_{\mathcal{U}(\g)}) &= ({\rm id} \preliediff \bT \circ e_{\bT} \circ \antipode_{\mathcal{U}(\g)})(e_{\bT} \circ \antipode_{\mathcal{U}(\g)}) \\
&=(\antipode_{\mathcal{U}(\g)}\circ e_{\bT} \circ \antipode_{\mathcal{U}(\g)})( e_{\bT} \circ \antipode_{\mathcal{U}(\g)}) \\
&= \eta_{\mathcal{U}(\a)} \circ \varepsilon_{\mathcal{U}(\g)}.
\end{align*}

\begin{align*} 
(e_{\bT} \circ \antipode_{\mathcal{U}(\g)}) \star_{\bT} { \rm id } &= (\mathcal{S}_{\mathcal{U}(\g)}\circ e_{{\bf T}\circ \mathcal{S}_{\mathcal{U}(\g)}} \preliediff \bT)\cdot {\rm id} \\
&= (\mathcal{S}_{\mathcal{U}(\g)}\circ (e_{{\bf T}\circ \mathcal{S}_{\mathcal{U}(\g)}} \preliediff \bT))\cdot {\rm id} \\
&= (\mathcal{S}_{\mathcal{U}(\g)}\circ (e_{{\bf T}\circ \mathcal{S}_{\mathcal{U}(\g)}} \preliediff \bT\circ e_{\bT} \circ e_{{\bf T}\circ \mathcal{S}_{\mathcal{U}(\g)}}))\cdot {\rm id} \\
&= (\mathcal{S}_{\mathcal{U}(\g)}\circ (e_ {\bf T}\circ e_{{\bf T}\circ \mathcal{S}_{\mathcal{U}(\g)}}))\cdot {\rm id} \\
&= \antipode_{\mathcal{U}(\g)} \cdot {\rm id} \\
&= \eta_{\mathcal{U}(\a)} \circ \varepsilon_{\mathcal{U}(\g)}
\end{align*}

\end{proof}
\begin{remark}
By inserting the defining fixed point equation for the transform $e$, we obtain from equation \eqref{eqn:antipodebt}:
\begin{equation*}
    S_{H_{\bf T}} 
    = e_{{\bf T}} \circ \mathcal{S}_H 
    = (\mathrm{id}\preliediff e_{{\bf T}}) \circ \mathcal{S}_H  
    = \mathcal{S}_H \preliediff (e_{{\bf T}} \circ \mathcal{S}_H) 
    = \mathcal{S}_H \preliediff \mathcal{S}_{H_{\bf T}}.
\end{equation*}
This formula for the antipode of the Hopf algebra $H_{{\bf T}}$ already appeared in \cite{li2022post}, Remark 2.5 and exploited in \cite{li2024sub} to provide a formula for the antipode of the Hopf algebras of ordered trees. We reproduce the formul implied on $e_{bT}$ by Proposition 2.12 in \cite{li2024sub}here in our notations and conventions:

\begin{equation}
e_{\bT}(x_1\cdots x_m) = \antipode_{\bT} \circ \antipode(x_1\cdots x_m) = \sum_{\pi} (x_{b_1} \blacktriangleleft x_{B_1} )\cdots (x_{b_{|\pi|-1}} \blacktriangleleft  x_{B_{|\pi|-1}})
\end{equation}
where $pi$ ranges the set of partitions of $[m]$ into \emph{tuples} $(B_1,\ldots, B_{|\pi|})$ such that $B_{\pi} = \{b_1 < \cdots < b_{|\pi|-1} \}$. The product $\blacktriangleright$ on $\mathcal{U}(\mathfrak{g}) $ is defined by:
\begin{equation*}
X \blacktriangleleft Y = X \preliediff_{\bT} e_{\bT} (Y)
\end{equation*}
It has been observed in \cite{li2024sub} that $\blacktriangleleft$ defined so yields an action of bi-algebras of  $\mathcal{U}(\mathfrak{g})$ equipped with the product $\star_{\bT \circ \antipode}$  over $\mathcal{U}(\mathfrak{g})$. Recall that, for any $X,Y \in \mathfrak{g}$:
$$
X \star_{\bT \circ \antipode} Y = Y_{(1)}(X \preliediff_{\bT \circ \antipode} Y_{(2)} )
$$
and, anticipating, $ T \circ \antipode = {\bT}^{-1^{\#}}$ (see Definition \ref{def:Hopfsmash}). The author in \cite{li2024sub} provides the following formula, for any $x_1,\ldots,x_m\in \mathfrak{g}$ and $Y \in \mathcal{U}(\mathfrak{g})$:
\begin{align*}
Y \blacktriangleleft x_1 \cdots x_m = \sum_{\sigma \in S_m} [x_1\cdots x_m; Y]_{\sigma}
\end{align*}
where the term $[x_1\cdots x_m; Y]_{\sigma}$ are computed according to the rules in Definition 2.14 of \cite{li2024sub} (replacing "right" with "left" to meet our conventions).
\end{remark}

\begin{remark}
\label{rk:topologicalgp}
Let $(G,\cdot,\Gamma)$ be a $\mathcal{O}$-group. Assume that both $G$ and $\Gamma$ are complete groups endowed with left invariant distances and that the action $\cdot$ is Lipschitz continuous. We denote by $d_G$ and $d_\Gamma$ the distances on $G$ and $\Gamma$, respectively. Suppose that we are given a strong contraction map $T\colon G \to \Gamma$   with $T(1_G)=T(1_\Gamma)$, for any $g,g' \in G$:
$$
d_{\Gamma}(T(g),T(g^{\prime})) \leq \sigma_T d_{G}(g,g^{\prime})~(\sigma_T \in ]0,1[).
$$
Under these hypotheses, the map $G \ni x \mapsto g \cdot T(x)\in G$ is a strong contraction and has an unique fixed point that we denote by $e_T(x)$. Furthermore, for any $x,y \in G$
$$
d_G(e_T(x),e_T(y)) \leq (1-\sigma_T)^{-1}d_G(x,y).
$$
All the above calculations can be readily adapted to prove that $e_T:G\to G$ has a continuous inverse given by:
$$
e_T^{-1}(g)=g\cdot T(g)^{-1}
$$
Notice that the right-hand side of the above equation is well defined for any continuous map from $G$ to $\Gamma$. Under the assumption that $T$ is a group $\mathcal{O}$-operator, we also have:
$$
e_{\bT}(g)= (g^{-1})^{-1^{\cdot_\bT}}.
$$
Here again, we can take the right-hand side of the above equation as a definition for $e_{\bT}(g)$ in the case where $\bT$ is a group $\mathcal{O}$-operator but not necessarily a strong contraction.
\end{remark}

\subsection{Matching $\mathcal{O}$-operator}
\label{sec:matchingoperatur}
We introduce \emph{matching $\mathcal{O}$-operator of various types}, see \cite{das2022cohomology}.
\subsubsection{Matching Lie $\mathcal{O}$-operator}
\label{ssec:matchOLie}
\begin{definition}[matching Lie $\mathcal{O}$-operator]
Let $(\mathfrak{g}, \cdot, \mathfrak{a})$ be a Lie right module. A \emph{matching Lie $\mathcal{O}$-operator} is the data of a family of maps $({\bf t}_w)_{w\in \Omega}:G\to \Gamma$ such that for any $x,y \in \g$,
\begin{enumerate}
\item \noindent and $w\in \Omega$, one has
$$
    [{\bf t}_w(x),{\bf t}_w(y)]_{\mathfrak{a}} 
    = {\bf t}_w(x \vartriangleleft {\bf t}_w(y)) 
    - {\bf t}_{w}(y \vartriangleleft {\bf t}_w(x)) 
    + {\bf t}_w([x,y]_\mathfrak{g}) 
$$

\item and for any $w\neq w^{\prime} \in \Omega$:
\begin{equation}
\label{eqn:MLie}
\tag{MLie}
[{\bf t}_w(x),{\bf t}_{w^{\prime}}(y)]_{\mathfrak{a}} = {\bf t}_w(x \vartriangleleft {\bf t}_{w^{\prime}}(y)) - {\bf t}_{w^{\prime}}(y \vartriangleleft {\bf t}_w(x))
\end{equation}
\end{enumerate}
\end{definition}

Notice that given a matching Lie $\mathcal{O}$-operator ${\bf s, \bf t}$, then ${\bf s} + {\bf t}$ is a $\mathcal{O}$ Lie operator. This \emph{does not} hold for any linear combinations of ${\bf s}$ and ${\bf t}$ except in the case $[-,-]_\g = 0$ (weight zero operators). In the next section, we provide an example of a matching Lie $\mathcal{O}$-operator.

\begin{theorem}
\label{thm:matching}
Let ${\bf s}: \mathfrak{g} \to \mathfrak{a}$ and ${\bf t}:\mathfrak{g}\to \mathfrak{a}$ be a matching Lie $\mathcal{O}$-operator on $(\mathfrak{g}, \vartriangleleft, \mathfrak{a})$ and call $\bS$ and $\bT$ their extension on $\univalgebra$ to an Hopf $\mathcal{O}$-operators on $(U(\g), \preliediff, U(\mathfrak{a}))$. Then for any $E,F \in \univalgebra$, the following relation holds
\begin{equation}
    \label{eqn:firstmatching}
    {\bT}(E \preliediff S_{\mathcal{U} (\mathfrak{a})}\bS(F_{(1)})) \star \bS(F_{(2)}) 
    = {\bS}(F\preliediff S_{\mathcal{U}(\mathfrak{a})}\bT(E_{(1)})) \star \bT(E_{(2)})
\end{equation}
\end{theorem}

\begin{proof}
We prove the statement inductively on the total length $N=m+n$ with $E=E_1\ldots E_m$ and $F=F_1\cdots F_n$. The $N=2$ initialization is equivalent to the compatibility relation between ${\bf s}$ and ${\bf t}$. So let $N > 2$ and suppose that the result holds for elements $E,F$ with $m+n < N$ and pick two elements $E,F$ such that $m+n=N-1$. Pick a third element $x \in \g$. Two cases occur. We start by proving:
\begin{equation*}
    {\bT}((Fx)\preliediff S_{\star}\bS(E_{(1)})) \star \bS(E_{(2)}) 
    = {\bS}(E\preliediff S_{\star}\bT((Fx)_{(1)})) \star \bT((Fx)_{(2)})
\end{equation*}
These are lengthy but straightforward computations. We use the tag $(Rec. H.) $ to indicate that we use the recursive hypothesis.
\begin{align*}
    \bS(&(Ex)\preliediff S_{\star}\bT(F_{(1)}))\star \bT(F_{(2)}) \\
    &= \bS((E\preliediff S_{\star}\bT(F_{(1)}))(x \preliediff S_{\star}\bT(F_{(2)})))\star \bT(F_{(3)}) \\
    &=\bS(E\preliediff (S_{\star}\bT(F_{(1)})\star S_{\star}\bS(x_{(1)}\preliediff S_{\star}\bT(F_{(2)})))\star\bS(x_{(2)} \preliediff S_{\star}\bT(F_{(3)}))\star \bT(F_{(4)}) \\
    &=\bS(E\preliediff (S_{\star}(\bS(x_{(1)}\preliediff S_{\star}\bT(F_{(1)})) \star \bT(F_{(2)}))))\star\bS(x_{(2)} \preliediff S_{\star}\bT(F_{(3)}))\star \bT(F_{(4)}) & {\rm (Rec.H.)} \\
    &=\bS(E\preliediff (S_{\star}(\bS(F_{(1)}\preliediff S_{\star}\bS(x_{(1)})) \star \bS(x_{(2)}))))\star\bT(F_{(2)} \preliediff S_{\star}\bS(x_{(3)}))\star \bS(x_{(4)}) &{\rm (Rec.H.)}\\
    &=\bS([E \preliediff {S}_{\star}\bS(x_{(1)})]\preliediff (S_{\star}(\bS(F_{(1)}\preliediff S_{\star}\bS(x_{(2)})))))\star\bT(F_{(2)} \preliediff S_{\star}\bS(x_{(3)}))\star \bS(x_{(4)}) \\
    &=\bS([E \preliediff S_{\star}\bS(x_{(1)})]\preliediff (S_{\star}(\bS(F_{(1)}\preliediff S_{\star}\bS(x_{(2)(1)})))))\star\bT(F_{(2)} \preliediff S_{\star}\bS(x_{(2)(2)}))\star \bS(x_{(4)}) \\
    &=\bT((F_{(1)}\preliediff S_{\star}\bS(x_{(2)}))\preliediff S_{\star}\bS(E_{(1)} \preliediff S_{\star}\bS(x_{(1)(1)})))\star\bS(E_{(2)} \preliediff S_{\star}\bS(x_{(1)(2)}))\star \bS(x_{(4)}) & {\rm (Rec.H.)} \\
    &=\bT((F_{(1)}\preliediff S_{\star}\bS(x_{(1)}))\preliediff S_{\star}\bS(E_{(1)} \preliediff S_{\star}\bS(x_{(2)})))\star\bS(E_{(2)} \preliediff S_{\star}\bS(x_{(3)}))\star \bS(x_{(4)}) \\
    &= \bT(F \preliediff S_{\star}\bS((E\preliediff S_{\star}\bS(x_{(1)}))\star_{\bS} x_{(2)})) \star \bS(({E}\preliediff S_{\star}\bS(x_{(3)})) \star_{S} x_{(4)})\\
    &=\bT(F \preliediff S_{\star}\bS(Ex_{(1)}))\star \bS(Ex_{(2)}).
\end{align*}
We now treat the second case
\begin{align*}
    \bS(E 
    &\preliediff S_{\star}\bT(Fx_{(1)}))\star \bT(Fx_{(2)})\\
    &= \bS(E \preliediff \big(S_{\star}\bT(x_{(1)}) \star S_{\star}\bT(F_{(1)} \preliediff S_{\star}\bT(x_{(2)}))\big))\star \bT(F_{(2)}\preliediff S_{\star}\bT(x_{(3)}))\star \bT(x_{(4)})\\
    &=\bS((E \preliediff S_{\star}\bT(x_{(1)})) \preliediff S_{\star}\bT(F_{(1)} \preliediff S_{\star}\bT(x_{(2)}))))\star \bT(F_{(2)}\preliediff S_{\star}\bT(x_{(3)}))\star \bT(x_{(4)}) \\
    &=\bT((F \preliediff S_{\star}\bT(x_{(1)})) \preliediff S_{\star}\bS(E_{(1)} \preliediff S_{\star}\bT(x_{(2)})))))\star \bS(E_{(2)} \preliediff S_{\star}\bT(x_{(2)}))\star \bT(x_{(2)}) & {\rm (Rec.H.)}\\
    &=\bT((F \preliediff S_{\star}\bT(x_{(1)})) \preliediff S_{\star}\bS(E_{(1)} \preliediff S_{\star}\bT(x_{(2)})))))\star\bT(x_{(2)}\preliediff S_{\star} \bS(E_{(2)}))\star \bS(E_{(3)}) & {\rm (Rec.H.)} \\
    &=\bT(((F \preliediff S_{\star}\bT(x_{(1)})) \preliediff S_{\star}\bS(E_{(1)} \preliediff S_{\star}\bT(x_{(2)}))))\star_{\bT}(x_{(3)}\preliediff S_{\star} \bS(E_{(2)})))\star \bS(E_{(3)}) \\
    &= \bT(((F \preliediff [S_{\star}(\bS(E_{(1)} \preliediff S_{\star}\bT(x_{(1)})) \star \bT(x_{(2)})))\star {\bT}(x_{(3)}\preliediff S_{\star} \bS(E_{(2)})))]\star \bS(E_{(3)})\\
    &=\bT(((F \preliediff [S_{\star}(\bT({x_{(1)} \preliediff S_{\star}\bS(E_{(1)})}) \star \bS(E_{(2)})))\star {\bT}(x_{(2)}\preliediff S_{\star} \bS(E_{(3)})))]\star \bS(E_{(4)}) &{\rm (Rec.H.)}\\
    &=\bT(Fx \preliediff S_{\star}\bS(E_{(1)}))\star \bS(E_{(2)}).
\end{align*}
\end{proof}

\subsubsection{Matching Hopf $\mathcal{O}$-operator}
\label{sec:MatchOope}
To state the following definition, we do not need to assume that the Hopf algebras $H$ or $K$ to be conilpotent or cocommutative. However, Proposition \ref{prop:matchingooperatorcomposition} only holds when the $\mathcal{O}$-operators are considered over a cocommutative Hopf algebra $H$ and we prefer to make this assumption from the beginning. It remains an open question to provide an equivalent result if $H$ is not co-commutative. Note that the Cartier-Quillen-Milnor-Moore theorem holds for Post-Hopf algebras \cite{catoire2024cartier}, in particular any \emph{connected} post-Hopf algebra is isomorphic to the envelope of its primitive elements \emph{via an isomorphism depending on the postHopf structure of H}.
\begin{definition}[Matching Hopf $\mathcal{O}$-operator]
\label{def:MatchOoper}
Let $H$ and $K$ be two cocommutative Hopf algebras.
Let $(H,\preliediff, K)$ be a right Hopf module. A \emph{matching Hopf $\mathcal{O}$-operator} is the data of a family of co-algebra morphisms $(\bT_w)_{w\in \Omega}:H\to K$ over such that for any $E,F \in H$,
\begin{enumerate}
\item \noindent and $w\in \Omega$, one has
$$
    \bT_w(E\star_{\bT_w} F) = \bT_w(E) \star \bT_w(F), \quad \varepsilon_K\circ\bT_w  = \varepsilon_H
$$
\item and for any $w,w^{\prime} \in \Omega$:
\begin{align}
\label{eqn:matchingHopf}
\tag{MHopf}
    {\bT}_w(E\preliediff \cS_{K}\bT_{w^{\prime}}(F_{(1)})) \star \bT_{w^{\prime}}(F_{(2)}) = {\bT}_{w^{\prime}}(F\preliediff \cS_{K}\bT_{w}(E_{(1)})) \star \bT_{w}(E_{(2)})
\end{align}
\end{enumerate}
\end{definition}
The following proposition is the Hopf counterpart to the property of a matching Lie operator that sums of those operators yield Lie $\mathcal{O}$-operators.
\begin{proposition}
\label{prop:matchingooperatorcomposition}
Let $\{\bS,\bT \}$ be a matching Hopf $\mathcal{O}$-operator. Then the coalgebra map
\begin{equation}
\label{eqn:compositionhopfooperator}
U\colon H\ni E \mapsto  {\bT}(E_{(1)}\preliediff \cS_{K}\bS(E_{(2)})) \star \bS_{}(E_{(3)}) \in K
\end{equation}
is a Hopf $\mathcal{O}$-operator (on the same right Hopf module $(H,\preliediff, K)$). In particular,
\begin{equation}
H \otimes H \ni (E, F) \mapsto \Big(\big(E \preliediff {\bT}(F_{(1)}\preliediff \cS_{K}\bS(F_{(2)}))\big)\preliediff \bT(F_{(3)})\Big)F_{(4)}
\end{equation}
is an associative product on $H$.
\end{proposition} 
\begin{proof} The proof follows from lengthy and tedious computations. Under the assumption that $H$ and $K$ are cocommutative, these computations are formally the same as for the group case (upon replacing the inverse by the post-composition with the antipode) and we prefer to detail them in the proof of Proposition \ref{prop:groupcase}.
%
%
\end{proof}
\begin{remark} We shall give an interpretation to formula \eqref{eqn:compositionhopfooperator} as a product between the operators $\bS$ and $\bT$, see Definition \ref{def:Hopfsmash}.
\end{remark}
\begin{remark}
We compare our notion of matching Hopf $\mathcal{O}$-operators to distributivity (of one operation over the other). With the notations introduced so far, $\preliediff_{\bf T}$ distributes over $\preliediff_{\bf S}$ if, for $X,Y,Z$
$$
    (X \preliediff_{\bf S} Y) \preliediff_{\bf T} Z = (X \preliediff_{\bf T} Z) \preliediff_{\bf S} ( Y \preliediff_{\bf T} Z)
$$
Inserting definitions, this will be implied by the following:
$$
    {\bf S}(Y)\star {\bf T}(Z) 
    = {\bf T}(Z)\star{\bf S}(Y \preliediff {\bf T}(Z))
$$
or else:
$$
    {\bf T}(Z)\star{\bf S}(Y)
    = {\bf S}(Y \preliediff \mathcal{S}_K{\bf T}(Z)) \star {\bf T}(Z)
$$
and we see that $\{\bS, \bT\}$ is a matching $\mathcal{O}$ Hopf operator if the ranges of $\bS$ and $\bT$ commute.
\end{remark}

\subsubsection{Matching group $\mathcal{O}$-operator}
\begin{definition}[Matching group $\mathcal{O}$-operator]
\label{def:MatchOGroup}
Let $(G, \cdot, \Gamma)$ be a right group module. A \emph{matching group $\mathcal{O}$-operator} is the data of a family of maps $(\bT_w)_{w\in \Omega}:G\to \Gamma$ such that for any $g,h \in G$,
\begin{enumerate}
\item \noindent and $w\in \Omega$, one has
$$
    \bT_w(g\star_{\bT_w}h) = \bT_w(g) \star \bT_w(h),
$$
\item and for any $w,w^{\prime} \in \Omega$:
\begin{align}
\label{eqn:matchinggroup}
\tag{MGroup}
    {\bT}_w(g\cdot \bT_{w^{\prime}}(h)^{-1})\bT_{w^{\prime}}(h) 
    = {\bT}_{w^{\prime}}(h\cdot \bT_{w}(g)^{-1}) \bT_{w}(g)
\end{align}
\end{enumerate}
\end{definition}

\begin{proposition}
\label{prop:groupcase} 
Let $\{\bS,\bT\}$ be a matching group $\mathcal{O}$-operator on a right group module $(G,\cdot, \Gamma)$. Then the map
\begin{equation}
\label{eqn:mapgroup}
    G\ni  g \mapsto \bS(g \cdot {\bT}(g)^{-1})\bT(g) \in \Gamma
\end{equation}
is a group $\mathcal{O}$-operator. In particular, 
\begin{equation*}
G\times G \ni (g,h) \mapsto \Big(g \cdot (\bS(h \cdot {\bT}(h)^{-1})\bT(h)) \Big)h
\end{equation*}
is a group product.
\end{proposition}
\begin{proof} Denote by $U$ the map defined by \eqref{eqn:mapgroup} and $h,g \in G$
\begin{align*}
U(g)U(h)&= \bS(g \cdot {\bT}(g)^{-1})\bT(g)\bS(h \cdot {\bT}(h)^{-1})\bT(h) \\
	     &= \bT(g \cdot {\bS}(g)^{-1})\bS(g)\bS(h \cdot {\bT}(h)^{-1})\bT(h) \\
	     &=  \bT(g \cdot {\bS}(g)^{-1})\bS(g \star_{\bS} (h \cdot {\bT}(h)^{-1} ))\bT(h) \\
	     &= \bT(g \cdot {\bS}(g)^{-1})\bT\big({h \cdot \bS\big((g \star_{\bS} (h \cdot {\bT}(h)^{-1}))\cdot \bT(h)\big)^{-1}}\big)\bS\big((\underline{g \star_{\bS} (h \cdot {\bT}(h)^{-1}))\cdot \bT(h)}\big),
\end{align*}
where we have used in the last equation the matching relation \eqref{eqn:matchinggroup} for $\bS, \bT$. We infer 
$$
\underline{(g \star_{\bS} (h \cdot {\bT}(h)^{-1}))\cdot \bT(h)} = (g \cdot(\bS(h\cdot\bT(h)^{-1}))(h\cdot\bT(h)^{-1}))\cdot \bT(h) = (g \cdot U(h))h,
$$
and 
\begin{align}
\label{eqn:tocompute}
U(g)U(h)&=\bT(g\cdot \bS(g)^{-1}) \bT(h\cdot \bS((g\cdot U(h))h)^{-1}) \bS((g\cdot U(h))h) \\
	     &= \bT((g\cdot \bS(g)^{-1})\star_{\bT}(h\cdot \bS((g\cdot U(h))h)^{-1}))\bS((g\cdot U(h))h)\nonumber\\
&=\bT(g\cdot \bS(g)^{-1})\star_{\bT}(h\cdot \bS((g\cdot U(h))h)^{-1}) \bS((g\cdot U(h))h) \nonumber \\ 
& =\bT( g\cdot \big(\underline{\bS(g)^{-1} \bT(h\cdot \bS((g\cdot U(h))h)^{-1}))} \big)(h\cdot \bS((g\cdot U(h))h)^{-1}))\bS((g\cdot U(h))h) \nonumber,
\end{align}
where we have used that $\bT$ is a group $\mathcal{O}$-operator. We multiply the underlined term by  $\bS((g\cdot U(h))h)^{-1})$ and we compute finally 
\begin{align*}
\underline{\bS(g)^{-1} \bT(h\cdot \bS((g\cdot U(h))h)^{-1}) \bS((g\cdot U(h))h)} &= \bS(g)^{-1} \bS(((g\cdot U(h))h)\cdot \bT(h)^{-1})\bT(h) \\
&=\bS(g^{-1^{\star_{\bS}}} \star_{\bS}\big(    ( ( g\cdot U(h) ) h )\cdot \bT(h)^{-1}   \big))\bT(h)\\
 &= g^{-1^{\star_{\bS}}} \star_{\bS}\big(    ( ( g\cdot \bS(h \cdot \bT(h)^{-1}) )) (h \cdot \bT(h)^{-1})   \big) \bT(h) \\ 
& = \bS(g^{-1^{\bS}} \star_{\bS} g \star_{\bS} (h \cdot \bT(h)^{-1}))\bT(h) \\
&= \bS(h \cdot \bT(h)^{-1}) \bT(h).
\end{align*}
From the equation above we get that \eqref{eqn:tocompute} is equal to 
$$
(g \cdot ( \bS( h \cdot \bT(h)^{-1} ) \bT(h) ) h) \cdot \bS((g\cdot U(h))h)^{-1}) =(( g \cdot U(h) )h) \bS( (g \cdot U(h))h ),
$$
which yields the desired relation
$
U(g)U(h) = U( (g\cdot U(h))h ).
$
\end{proof}
The following Proposition is a simple consequence of the definitions.

\begin{proposition}
Let $({\bf T}_{w})_{w\in\Omega}$ be a matching Hopf $\mathcal{O}$-operator. Then it restricts to a matching Lie $\mathcal{O}$-operator on the primitive element of $H$ and to a matching group $\mathcal{O}$-operator on the group-like elements of $H$.
\end{proposition}


\subsection{Examples}
\label{sec:examples}

\subsubsection{Translations by the unit in a non-symmetric operad with multiplication}
\label{sec:transunit}

\newcommand{\gdiff}{\g_{\rm diff}}
In this section, we exhibit a matching $\mathcal{O}$ Lie operator for any given operad with multiplication $(\mathcal{P},\gamma,m,\unit)$:
\begin{enumerate}
\item $\mathcal{P}=(\mathcal{P}(n))_{n\geq 1})$ is a collection of complex vector spaces with $\mathcal{P}(1)=\mathbb{C}$,
\item $\circ$ denotes composition in the category of collections, given two collections $\mathcal{P}$ and $\mathcal{Q}$ of vector spaces:
$$
    (\mathcal{P}\circ \mathcal{Q})(n)
    =\bigoplus_{\substack{p\geq 1,\\n_1+\cdots+n_p=n}}
    \mathcal{P}(p)\otimes \mathcal{Q}(n_1)\otimes\cdots\otimes\mathcal{Q}(n_p),
$$
\item the letter $\gamma$ denotes an operadic composition, in particular:
$$
    \gamma : \mathcal{P} \circ \mathcal{P} \to \mathcal{P},
$$
and $I\in\mathcal{P}(1)$ is the \emph{unit} for $\gamma$, which means $\gamma(I \otimes x)=x$ and $\gamma(x,I^{\otimes |x|})=x$,
\item $m\in \mathcal{P}(2)$ is a multiplication: it is an operator with arity $2$ in $\mathcal{P}$ satisfying:
$$
    m\circ ({\rm id}\otimes m)=m \circ m\otimes {\rm id}.
$$
\end{enumerate}
In the sequel, we denote by $\mathbb{C}[\mathcal{P}]_+$ the linear span of $\mathcal{P}(n),~n\geq 2$ and $\mathbb{C}[\mathcal{P}]$ the linear span of the $\mathcal{P}(n)$, $n\geq 1$. Multiplication $m$ induces an associative algebra product on $\mathbb{C}[\mathcal{P}]$:
$$
    p \centerdot q = \gamma(m\circ (p\otimes q)),~p,q \in \mathcal{P}.
$$
We use the symbol $\lieinv{-}{-}$ for the Lie bracket on $\mathbb{C}[\mathcal{P}]$ associated with the multiplication $\centerdot$. We introduce the left and right translations with respect to $\centerdot$ by the unit $\unit$:
\begin{equation}
\begin{array}{cccc}
    \lambda :   & \algop      & \to     & \algop \\
		      &    x 	    & \mapsto & I\centerdot x
\end{array}, 
\quad	
\begin{array}{cccc}
    \rho :  & \algop    & \to     & \algop \\
	      &    x 	  & \mapsto & x\centerdot I
\end{array}
\end{equation}
and set for brevity $\g = (\algop,\lieinv{-}{-})$. The operadic composition $\gamma$ produces a preLie product $\vartriangleleft$ on $\mathbb{C}[\mathcal{P}]_+$, the Gerstenhaber preLie product:
$$
    p \vartriangleleft q 
    = \sum_{i=1}^{|p|}\gamma(p\circ(\unit\otimes\ldots\otimes\unit\otimes q\otimes \unit \otimes \ldots \otimes \unit)),~p,q \in \mathbb{C}[\mathcal{P}]_+.
$$
We denote by $\liediff{-}{-}$ the adjacent Lie bracket on $\mathbb{C}[\mathcal{P}]_{+}$:
$$
    \liediff{p}{q} 
    := p \vartriangleleft q - q \vartriangleleft p,~p,q \in \mathbb{C}[\mathcal{P}]_+.
$$
We denote by $\a$ the Lie algebra $(\mathbb{C}[\mathcal{P}]_+, \liediff{-}{-})$. Proofs of the following propositions follow from induction.

\begin{proposition}
$(-\lambda)$ and $\rho$ are $\mathcal{O}$ Lie operators over the Lie module $(\g,\vartriangleleft,\a)$ satisfying:
\begin{align*}
    &\lambda(p) \vartriangleleft \lambda(q) = \lambda(p \vartriangleleft \lambda(q) + q \centerdot p),~
    &\rho(p) \vartriangleleft \rho(q) = \rho(p \vartriangleleft \rho(q) + p\centerdot q),
\end{align*}
Besides $(-\lambda,\rho)$ is a matching Lie $\mathcal{O}$-operator;
$$
    \liediff{\lambda(p)}{\rho(q)}
    =\lambda(p\vartriangleleft q) - \rho(q \vartriangleleft p),~p,q \in \mathbb{C}[\mathcal{P}].
$$
\end{proposition}

We next provide formulae for the extensions of $\lambda$ and $\rho$ and $\lambda + \rho$ as Hopf $\mathcal{O}$-operators on the universal envelope of the Lie module $(\g,\vartriangleleft,\a)$. Pick a sequence $P$ of $n$ elements $p_1,\ldots,p_n$ in the Lie algebra $(\mathbb{C}[\mathcal{P}],[-,-]_{\centerdot})$ and a partition $\pi = \{\pi_1,\ldots,\pi_j\} \in {\rm Part}(n)$ with blocks $\pi_i$ of the interval $[n]$ and set for any $1 \leq j \leq k$:
	$$
	P_{\pi_j} = (p_{j_1}\centerdot \cdots \centerdot p_{j_{p_k}}),~ \pi_j = {j_1 < \cdots < j_{p_k}}.
	$$
In the following, we denote by $\bar{\pi_j}$ the block $\pi_j$ endowed with the opposite of the natural order.
\begin{proposition} 
\label{prop:translationidentity}
Let $p_1,\ldots,p_n \in (\mathbb{C}[\mathcal{P}],[-,-]_{\centerdot})$, then:
	\begin{equation}
	\label{eqn:extensionrho}
	[\boldsymbol{\rho}](p_1\cdots p_n) = \sum_{\substack{\{\pi_1,\ldots,\pi_k\} \in {\rm Part}(n)\\ \sigma \in \mathcal{S}_k}} \boldsymbol{\rho}(P_{\pi_{\sigma^{-1}(1)}})\cdots \boldsymbol{\rho}(P_{\pi_{\sigma^{-1}(k)}})
	\end{equation}
 and:
        \begin{equation*}
                    [\boldsymbol{\lambda}](p_1\cdots p_n) = \sum_{\substack{\{\pi_1,\ldots,\pi_k\} \in {\rm Part}(n) \\ \sigma \in \mathcal{S}_k}}(-1)^n\boldsymbol{\lambda}(P_{\bar{\pi}_{\sigma^{-1}(1)}}) \cdots \boldsymbol{\lambda}(P_{\bar{\pi}_{\sigma^{-1}(k)}})
        \end{equation*}
\end{proposition}
\begin{proof} Recall that $\mathrm{sol}_1:\mathcal{U}(\mathbb{C}[\mathcal{P}]_+,\llbracket-,-\rrbracket) \to \mathbb{C}[\mathcal{P}]_+$ is the projector onto $\mathbb{C}[\mathcal{P}]_+$ parallel to 
$$
\pi_{nat}(\sum_{n\geq 2} \mathcal{S}^{(n)}(\mathbb{C}[\mathcal{P}]_+)) \oplus \mathbb{C}\cdot 1.
$$ 
where $\mathcal{S}(\mathbb{C}[\mathcal{P}]_{+})$ is the coalgebra of symmetric polynomials on $\mathbb{C}[\mathcal{P}]_{+}$, $\mathcal{S}^{(n)}(\mathbb{C}[\mathcal{P}]_{+})$ is the vector space of homogeneous polynomials of degree $n\geq 0$. The coalgebra map $\pi_{\rm nat} \colon \mathcal{S}(\mathbb{C}[\mathcal{P}]_{+}) \to \mathcal{U}(\mathbb{C}[\mathcal{P}]_{+}, \llbracket -,- \rrbracket)$ is an isomorphism of coalgebras defined by 
$$
\pi_{{\rm nat}}(E_1\cdots E_n) = \sum_{\sigma\in\mathcal{S}_n} E_{\sigma^{-1}(1)} \cdots E_{\sigma^{-1}(n)}
$$
Since $\mathcal{S}(\mathbb{C}[\mathcal{P}]_{+})$ ist co-free  co-nilpotent cocommutative  co-generated by $\mathbb{C}[\mathcal{P}]_{+}$ , we get that two coalgebraic maps $\alpha,\beta \colon \mathcal{U}(\mathbb{C}[\mathcal{P}],[-,-]_{\centerdot})\to \mathcal{U}(\mathbb{C}[\mathcal{P}], \llbracket -,-\rrbracket)$ are equal if and only if ${\rm sol}_1\circ \alpha = {\rm sol}_1\circ \beta $.

Set $
\rho^{(1)} := {\rm sol}_1 \circ [\boldsymbol{\rho}].$
We define $\hat{m}\colon \mathcal{U}(\mathbb{C}\llbracket P \rrbracket, [-,-]_{\centerdot}) \rightarrow \mathbb{C}\llbracket P \rrbracket$ by taking the $\centerdot$ product of all entries of a polynomial in $\mathcal{U}(\mathbb{C}\llbracket P \rrbracket, [-,-]_{\centerdot})$.
We prove recursively that $\rho^{(1)}(E_1 \cdots E_n) = \rho (\hat{m}(E_1\cdots E_n))$. If $n=1$, the assertion is trivial.
Pick $p_1 \cdots p_{n} \in \univalgebra$ and $h \in \mathbb{C}[\mathcal{P}]$. Then
\begin{align*}
\rho^{(1)}(p_1\cdots p_n h)&=\rho^{(1)}(p_1\cdots p_n)\preliediff \rho(h) - \rho^{(1)}((p_1\cdots p_n) \preliediff \rho(h))\\
&=\rho({\hat{m}(p_1\cdots p_n)}) \preliediff \rho(h) - \rho(\hat{m}(p_1\cdots (p_i \preliediff \rho(h)) \cdots p_n)) \\
&=\rho(\hat{m}(p_1 \cdots p_n h)) + \rho(p_1\cdots (p_i \preliediff \rho(h))\cdots p_n) - \rho(\hat{m}(p_1\cdots (p_i \preliediff \rho(h)) \cdots p_n)) \\
&=\rho(\hat{m}(p_1 \cdots p_n h))
\end{align*}
The right-hand side of equation \eqref{eqn:extensionrho} is well-known to define the unique coalgebric extension of its corestriction to $\mathbb{C}[\mathcal{P}]_{+}$, since the postcomposition of the right hand side of \eqref{eqn:extensionrho} by the inverse of $\pi_{\rm nat}$ is equal to
$$
\mathcal{U}(\g)\ni p \mapsto \varepsilon(p) + \sum_{n\geq 0} m^{(n)} ({\varphi}^{\otimes n} (\overline{\Delta}_n(p)))
$$
where $\varphi = \rho\circ \hat{m}$ and $m^{(n)}\colon \mathcal{S}(\mathbb{C}[\mathcal{P}]_+)^{n}$ is the $n^{th}$ iterate of the produt on $\mathcal{S}(\mathbb{C}[\mathcal{P}]_+)$, and $\Delta_n$ is the $n^{th} $ iteration of the reduced (unshuffle) coproduct of $\mathcal{U}(\g)$, see e.g  Remark 2.13.1 in \cite{cartier2021classical}.
\end{proof}

\subsubsection{Translations in the dendriform operad}
\label{ssec:translations}

In this section, we show that in the dendriform operad Dend, which is an operad with multiplication, we can define a second matching $\mathcal{O}$ Lie operator, besides the one described in the previous section, also given by translations by the unit of Dend, but with respect to the two half-dendriform products. The dendriform operad is generated by two operators $\prec$ and $\succ$ with arity $2$ subject to the relations:
\begin{align*}
    \prec \circ (\prec \otimes {\rm I})
    &=\prec \circ (I \otimes (\prec + \succ)),\\
    \prec \circ(\succ \otimes I)
    &=\succ \circ (I \otimes \prec),\\
    \succ \circ (\succ \otimes I)
    &= ((\succ + \prec) \otimes I) \circ \succ .
\end{align*}
In the following we will use the notation:
$$
    m = \prec + \succ.
$$
The above relations imply that $m$ is a product in Dend. We will use $\centerdot$ for the induced product on $\mathbb{C}[\mathcal{P}]$. Note that, for any $p,q \in \mathcal{C}[\mathcal{P}]$:
$$
    [p,q]_{\centerdot}=p \blacktriangleleft q - q \blacktriangleleft p
$$
with $\blacktriangleleft$ the preLie product over $\mathbb{C}[\mathcal{P}]$ defined by: 
$$
    p \blacktriangleleft q = p \prec q - q \succ p.
$$
Next, we define:
\begin{equation}
\begin{array}{cccc}
	\lambda_{\prec} : & \mathbb{C}[{\rm Dend}] & \to     & \mathbb{C}[{\rm Dend}] \\
					   &    x 	                & \mapsto & I\prec x
\end{array}, \quad
\begin{array}{cccc}
	\rho_{\succ} :  &\mathbb{C}[{\rm Dend}] & \to     &\mathbb{C}[{\rm Dend}] \\
				    &    x 	                & \mapsto & x\succ I
\end{array}
\end{equation}
Recall the notations of the previous Section.

\begin{proposition}
Both $\lambda_{\prec}$ and $\lambda_{\succ}$ are $\mathcal{O}$ Lie operators. Furthermore, $(\lambda_{\prec},\rho_{\succ})$ is a matching $\mathcal{O}$ Lie operator on the Lie module $(\g,\vartriangleleft,\a)$.
\end{proposition}

\begin{proof}
Let us prove that $\lambda_{\prec}$ is a $\mathcal{O}$ Lie operator. The result for $\rho_{\succ}$ will follow from the same computations. 
\begin{align*}
    \lambda_{\prec}(x) \vartriangleleft \lambda_{\prec}(y)
    &=\lambda(x) \prec \lambda(y)\\ 
    &=((I \prec y) \prec x)+I \prec (x \vartriangleleft(I\prec y))\\
    &=I \prec (y \prec x + x \succ y) + I \prec (x \vartriangleleft (I \prec y))\\
    &=\lambda_{\prec}(x \vartriangleleft \lambda_{\prec}(y))+\lambda_{\prec}(y \prec x + x \succ y)\\
    &=\lambda(x \vartriangleleft \lambda_{\prec}(y))+\lambda_{\prec}(x\centerdot y).
\end{align*}
This yields
\begin{align*}
    \llbracket\lambda_{\prec}(x),\lambda_{\prec}(y)\rrbracket
    =\lambda_{\prec}(x\vartriangleleft\lambda_{\prec}(y))-\lambda_{\prec}(y \vartriangleleft \lambda_{\prec}(x))+ \lambda ([x,y]_{\centerdot})
\end{align*}
Next, we show that $\lambda_{\prec}, \rho_{\succ}$ is a matching $\mathcal{O}$ Lie operator.
\begin{align*}
    \lambda_{\prec}(x) \vartriangleleft \rho_{\succ}(y)
    &=\rho_{\succ}(y)\prec x + I \prec (x \vartriangleleft \rho_{\succ}(y)) \\
    &=(y \succ I) \prec x + I \prec (x \vartriangleleft \rho_{\succ}(y)) \\
    &= y \succ (I \prec x) + \lambda_{\prec}(x \vartriangleleft\rho_{\succ}(y)) \\
\end{align*}
and, by the same computations: 
\begin{align*}
    \rho_{\succ}(y) \vartriangleleft \lambda_{\prec}(x)
    = y \succ (I \prec x) + \rho_{\succ}(y \vartriangleleft (I\prec x)) 
\end{align*}
    The above two equations imply $\llbracket \lambda_{\prec}(x),\rho_{\succ}(y) \rrbracket=\lambda_{\prec}(x \vartriangleleft \rho_{\succ}(y))-\rho_{\succ}(y \vartriangleleft \lambda_{\prec}(y))$.
\end{proof}

\subsubsection{Chen Fliess series and multiplicative feedback}
\label{sec:chenfliess}
We refer the reader to \cite{guggilam2022formal} for details.
In the following, $\mathbb{R}^q\langle\langle X \rangle\rangle$ is the linear space of $q$-tuples of formal series on words with entries in an alphabet $X$. We use the symbol $\shuffle$ for the shuffle product between series in $\mathbb{R}^q\langle\langle X\rangle\rangle$.
Pick $k\geq 1$ and $\ell \geq 1$ two integers. Let $X$ and $X^{\prime}$ be two alphabets with $|X^{\prime}|=\ell+1$. Let  $c=(c_1,\ldots,c_k) \in \mathbb{R}^{k}\langle\langle X^{\prime} \rangle\rangle$ and $d=(d_1,\ldots,d_\ell) \in \mathbb{R}^\ell\langle\langle X\rangle\rangle$ be formal series. 
Define the composition product $c \circ d$  as the formal series in $\mathbb{R}^k\langle\langle X \rangle\rangle$ defined by:
$$
c \circ d = \sum_{\eta \in (X^{\prime})^{\star}} (c,\eta) \psi_d(\eta)(\emptyset)
$$
where $\psi_d(\eta) \in {\rm End}(\mathbb{R}^k\langle\langle X \rangle\rangle)$ is defined by:
\begin{align*}
&\psi_d(x_i^{\prime} \eta) = \psi_d(x_i^{\prime}) \circ \psi_d(\eta), \\
&\psi_d(x^{\prime}_i)(e)=x_0(d_i\shuffle e)
\end{align*}
The multiplicative fixed composition product $e \curvearrowleft F$ between two formal series $e \in \mathbb{R}^{\ell}\langle\langle X \rangle\rangle$ and $F \in \mathbb{R}^m \langle\langle X \rangle\rangle$ on the same alphabet $X=\{x_0,\ldots,x_m\}$ is defined as 
$$
e \curvearrowleft F = \sum_{\eta\in X^{\star}} (e,\eta) \eta \curvearrowleft F = \sum_{\eta \in X^{\star}}(e,\eta)\bar{\varphi}_d(\eta)(\emptyset)
$$
where $\varphi_d\coloni \mathbb{R}\langle\langle X \rangle\rangle \rightarrow {\rm End}(\mathbb{R}\langle\langle X \rangle\rangle)$ is defined by:
\begin{align*}
&\bar{\varphi}_d (x_0)(t)=x_0t, \\
&\bar{\varphi}_d (x_i)(t)=x_i(e_i\shuffle t).
\end{align*}
From Lemma 5.2 in \cite{guggilam2022formal}, for any $c \in \mathbb{R}^k\langle\langle X^{\prime} \rangle\rangle$, $e \in \mathbb{R}^{\ell}\langle\langle X \rangle\rangle$ and $F \in \mathbb{R}^m\langle \langle X \rangle \rangle$:
\begin{equation}
\label{eqn:dist}
c \circ (e \curvearrowleft F)=(c\circ d) \curvearrowleft F.
\end{equation}
With the associative group $\star$ product between two series $F,G \in \mathbb{R}^m\langle\langle X \rangle\rangle$ defined by:
$$
F \star G = (F \curvearrowleft G) \shuffle G
$$
one obtains that $((\mathbb{R}^k\langle\langle X \rangle\rangle,\shuffle),\curvearrowleft, (\mathbb{R}^m\langle\langle X\rangle\rangle,\star) )$ is a right group module. To any formal series $\mathbb{R}^m \langle\langle X \rangle\rangle$, one considers the left translation:
$$
\lambda_c \colon \mathbb{R}^k\langle \langle X \rangle \rangle \to \mathbb{R}^m\langle\langle X \rangle\rangle, d \mapsto c \circ d 
$$
Then $\{\lambda_c,~ c \in \mathbb{R}^m\langle\langle X \rangle\rangle\}$ is a matching group $\mathcal{O}$-operator . As it can be easily checked from \eqref{eqn:dist}:
\begin{align*}
c \circ (d \curvearrowleft (c^{\prime} \circ d^{\prime})^{-1^{\star}})\star (c^{\prime} \circ d^{\prime}) &= ((c \circ (d \curvearrowleft (c^{\prime} \circ d^{\prime})^{-1^{\star}})) \curvearrowleft (c^{\prime}\circ d^{\prime}))\shuffle (c^{\prime}\circ d^{\prime})\\
&=(c \circ (d \curvearrowleft ((c^{\prime} \circ d^{\prime})^{-1^{\star}}\star(c^{\prime}\circ d^{\prime}))))\shuffle (c^{\prime}\circ d^{\prime})\\
&=(c \circ d) \shuffle (c^{\prime} \circ d)\\
&=c^{\prime} \circ (d^{\prime} \curvearrowleft (c \circ d)^{-1^{\star}}) \star (c\circ d) 
\end{align*}

\subsection{Filtered setting}
\label{sec:filtered}
All Hopf algebras considered in this section are cocommutative.
In this section, we propose definitions for $\mathcal{O}$-operators of the various types we have considered so far where Lie algebras, Hopf algebras and groups come equipped with a filtration (by Lie algebras, Hopf algebras and groups, respectively). We refer to the first part of \cite{goncharov2024formal} for the basic definitions of (complete) filtered Lie algebra, (complete) filtered Hopf algebra and (complete) filtered group in the context of Rota-Baxter structures. This section is a mere generalization of \cite{goncharov2024formal}

Let us introduce the filtered counterpart of the objects we have defined.

Given a right Lie module $(\mathfrak{g},\vartriangleleft,\mathfrak{a})$, we call \emph{standard filtration} of a right Lie module $(\mathfrak{g},\vartriangleleft,\mathfrak{a})$ the filtrations by Lie algebras recursively defined by:
\begin{align*}
&\mathcal{F}_1 \mathfrak{g} = \mathfrak{g} \\
& \mathcal{F}_{n+1}\mathfrak{g} = {\rm ad}_{\mathfrak{g}} \mathcal{F}_{n}\mathfrak{g} + \mathcal{F}_n\mathfrak{g} \vartriangleleft \mathfrak{a},\quad n\geq 1
\end{align*}
and we consider on the Lie algebra $\mathfrak{a}$ its derived filtration:
\begin{align*}
&    \mathcal{F}_{1} = \mathfrak{a} \\
&    \mathcal{F}_{n+1} \mathfrak{a} = {\rm ad}_{\mathfrak{a}} \mathcal{F}_{n}\mathfrak{a}, \quad n\geq 1.
\end{align*}
Given a right filtered Lie module $(\mathfrak{g}, \mathcal{F}_{\bullet})$, we denote by $\hat{\mathfrak{g}}$ its completion:
$$
\hat{\mathfrak{g}} = \underset{\leftarrow}{\lim} g/\mathcal{F}_{n+1}\mathfrak{\g} 
$$
We let $({\mathcal{U}(\mathfrak{g}), \mathcal{F}_{\bullet}})$ be the filtered universal envelope of $(\mathfrak{g}, \mathcal{F}_{\bullet})$ \cite{goncharov2024formal, quillen1969rational, fresse2017homotopy} and denote by $\widehat{\mathcal{U}(\mathfrak{g})}$ its completion.
Recall that filtration on $\mathcal{U}{\mathfrak{g}}$ is defined by, for any $n\geq 1$
\begin{align}
\mathcal{F}_n(\mathcal{U}(\g)) = \{E_1\cdots E_n \colon E_i \in \mathcal{F}_{n_i}\mathfrak{g},~ \sum_{i}n_i \geq n\}.
\end{align}
We denote by $\mathbb{G}$ the set of group like elements of $\widehat{\mathcal{U}(\mathfrak{g})}$:
\begin{align*}
{\rm Grp}(\mathcal{U}(\mathfrak{g})) = \{ x \in \widehat{\mathcal{U}(\mathfrak{g})}\colon \Delta(x) = x \widehat{\otimes} x\}.
\end{align*}
Given a map between filtered Lie algebras, filtered Hopf algebras or filtered group respecting the filtrations, we add an $\widehat{\quad}$ to its symbol to denote its extension/coextension to the completions. For exemple, if $\bT \colon \mathcal{U}(\mathfrak{g}) \to \mathcal{U}(\mathfrak{a})$ is filtered Hopf $\mathcal{O}$-operator, its completion is denoted $\widehat{\bT}\colon \widehat{\mathcal{U}(\mathfrak{g})}\to\widehat{\mathcal{U}(\mathfrak{a})}$.

Given $x=(0,x_1,\ldots,x_m,\ldots) \in \mathcal{F}_1\widehat{\mathcal{U}(\mathfrak{g})}(={\rm ker}(\varepsilon))$, we define the exponential map $ \exp\colon \mathcal{F}_1\widehat{\mathcal{U}(\mathfrak{g})}\to \widehat{\mathcal{U}(\mathfrak{g})}$ by th equation
\begin{equation}
\label{expo}
\exp(x) = (1,1+x_1, 1+x_2 + \frac{x_2^2}{2}, \ldots, \sum_{i=1}^{n} \frac{x_n^i}{i!},\ldots) : = 1 + x + \frac{1}{2}x^2 + \cdots + \frac{1}{n!}x^n + \cdots 
\end{equation}
We recall the following result. 
\begin{proposition}[see \cite{quillen1969rational}]
\label{prop:isomorphismexpo}
The exponential map is a bijection from $\hat{\mathfrak{g}}$ to $\mathbb{G}$.
\end{proposition}

\begin{definition}[]
$\bullet$ A \emph{right filtered Lie module} is a triple $(\mathfrak{g},\vartriangleleft, \mathfrak{a})$ where $\mathfrak{g},\mathfrak{a}$ are two filtered Lie algebras and $\vartriangleleft$ is compatible with the filtration, which means 
$$
\mathcal{F}_{m}(\mathfrak{g})\vartriangleleft \mathcal{F}_{n}(\mathfrak{a}) \subset \mathcal{F}_{m+n} \mathfrak{g}.
$$
$\bullet$ A \emph{right filtered group module} is a triple $(G,\cdot,\Gamma)$ where $G,\Gamma$ are two filtered group  and $\cdot$ is compatible with the filtrations, which means 
$$
\mathcal{F}_{m}(G)\cdot \mathcal{F}_{n}(\Gamma) \subset \mathcal{F}_{m+n} G.
$$
$\bullet$ A \emph{right filtered Hopf module} is a triple $(H,\preliediff,K)$ where $H,K$ are two filtered Hopf algebras and $\preliediff$ is compatible with the filtrations, which means 
$$
\mathcal{F}_{m}(H)\preliediff \mathcal{F}_{n}(K) \subset \mathcal{F}_{m+n} H.
$$
\end{definition}
On the completion $\hat{\mathfrak{g}}$ of a filtered Lie algebra $(\mathfrak{g}, \mathcal{F}_{\bullet})$, we denote by $\cdot_{\rm BCH}$ the BCH (Baker-Campbell-Hausdorff) group law on $\mathfrak{g}$:
$$
x\cdot_{\rm BCH} y = \log(\exp(x)\exp(y)),\quad x,y \in \hat{\mathfrak{g}}.
$$
\begin{proposition}[{Formal integration of a filtered right Lie module}]
\label{prop:formaintegrationliemodule}
Let $(\g,\vartriangleleft, \a)$ be a right filtered Lie module and let $(\mathcal{U}(\g), \preliediff, \mathcal{U}(\a))$ be the (ordinary) envelope of $(\g,\vartriangleleft, \a)$. Then, if considering the filtrations on $\mathcal{U}(\g)$ and $\mathcal{U}(\a)$ induced by the filtrations of $\g$ and $\a$ respectively, $(\mathcal{U}(\g), \preliediff, \mathcal{U}(\a))$ becomes a filtered right Hopf module. 
Moreover, $\preliediff$ yields in a canonical way a complete right Hopf module $(\widehat{\mathcal{U}(\g)}, \widehat{\preliediff},\widehat{\mathcal{U}(\a))}$, meaning that 
Define the action $\cdot \colon \hat{\g} \times \hat{\a} \to \g$ by 
$$
{x} \cdot {y}  := x \preliediff \exp(y)
$$
Then $( (\g, \cdot_{\rm BCH}), \cdot, (\a, \cdot_{\rm BCH}) )$ is a right filtered group module.
\end{proposition}
\begin{proof} The first part of the proposition is obvious. We will continue to use the symbol $\preliediff$ for completion $\widehat{\preliediff}$. Let us check the second part. Let $x_1,x_2 \in \g$ and $y\in \a$. We prove first that $\cdot$ distributes over the multiplication of two elements in $(\g,\cdot_{\rm BCH}) $. Since $ \exp\colon \hat{g} \to \mathbb{G}$ is a bijection, we check that 
$$
\exp(x_1x_2 \preliediff y) = \exp((x_1 \preliediff \exp(y))\cdot_{\rm BCH}\exp((x_2 \preliediff \exp(y))))
$$
The result follows from straightforward computations based on the relation $\exp(x\preliediff g) = \exp(x)\preliediff g$, $x\in \g, g\in {\rm Grp}({\mathcal{U}(\mathfrak{a})})$. 
\begin{align*}
\exp(x_1x_2 \preliediff y) &= \exp(x_1\cdot_{\rm BCH}x_2) \preliediff \exp(y) \\
&= (\exp(x_1) \exp(x_2))\preliediff \exp(y) \\
&= \exp(x_1\cdot_{\rm BCH}x_2) \preliediff \exp(y) \\ 
&= (\exp(x_1)\preliediff \exp(y))(\exp(x_2)\exp(y)) \\
&=(\exp(x_1)\preliediff \exp(y) )( \exp(x_2)\preliediff \exp(y) ) \\
&=(\exp(x_1\preliediff \exp(y)) )( \exp(x_2\preliediff \exp(y)) ) \\
&=\exp( (x_1\preliediff \exp(y))\cdot_{\rm BCH}  (x_2\preliediff \exp(y))  ). 
\end{align*}
We check next that $(x \preliediff \exp(y_{1}))\preliediff \exp(y_{2})=(x \preliediff \exp(y_{1}\cdot_{\rm BCH} y_{2}))$ for any $x\in \mathfrak{g}, y_1,y_2 \in \mathfrak{a}$. This relations follows from the fact that $\preliediff$ is a action and the definition of the $\cdot_{\rm BCH}$ group law.
\end{proof}
\begin{definition}[{Filtered $\mathcal{O}$-operators}]
$\bullet$ A filtered Lie $\mathcal{O}$-operator on a filtered Lie module $(\mathfrak{g}, \vartriangleleft, \mathfrak{a})$ is a $\mathcal{O}$-operator $\bT\colon\mathfrak{g}\to\mathfrak{a}$ compatible with the filtrations of $\mathfrak{g}$ and $\mathfrak{a}$.

$\bullet$ A filtered Hopf $\mathcal{O}$-operator on a filtered Hopf module $(H, \preliediff, K)$ is a $\mathcal{O}$-operator $\bT\colon\mathfrak{g}\to\mathfrak{a}$ compatible with the filtrations of $H$ and $K$.

$\bullet$ A filtered group $\mathcal{O}$-operator on a filtered group module $(G, \cdot, \Gamma)$ is a $\mathcal{O}$-operator $\bT\colon G \to K$ compatible with the filtrations of $G$ and $\Gamma$.
\end{definition}

\begin{definition}[Filtered matching $\mathcal{O}$-operators]
$\bullet$ A filtered matching Lie $\mathcal{O}$-operator on a filtered Lie module $(\mathfrak{g}, \vartriangleleft, \mathfrak{a}, \mathcal{F}_n)$ is a matching $\mathcal{O}$-operator $\{{\bm t}_w,\,w\in\Omega
\}$ on $(\mathfrak{g}, \vartriangleleft, \mathfrak{a})$ compatible with the filtrations of $\mathfrak{g}$ and $\mathfrak{a}$.

$\bullet$ A filtered matching Hopf $\mathcal{O}$-operator on a filtered Hopf module $(H, \preliediff, K, \mathcal{F}_{\bullet})$ is a matching $\mathcal{O}$-operator $\{\bT_w,\,w\in \Omega\}$ on $(H,\preliediff, K)$ compatible with the filtrations on $H$ and $K$.

$\bullet$ A filtered matching group $\mathcal{O}$-operator on a filtered group module $(G, \cdot, \Gamma, \mathcal{F}_{\bullet})$ is a matching $\mathcal{O}$-operator $\{\bT_w,\,w \in \Omega\}$ compatible with the filtrations of $G$ and $\Gamma$.
\end{definition}

The same argument as in Proposition 4.13 \cite{goncharov2024formal} implies that $\widehat{\bT}$ is a Hopf $\mathcal{O}$-operator in the completion if $\bT\colon H \to K$ is a filtered Hopf $\mathcal{O}$-operator. 
\begin{proposition}[see Proposition 4.13 in \cite{goncharov2024formal}]
Let $\bS$ and $\bT$ be a matching filtered Hopf $\mathcal{O}$-operator over a right filtered Hopf module $(H,\preliediff, K)$. Then their completion $\widehat{\bS}$ and $\widehat{\bT}$ yield a complete matching Hopf $\mathcal{O}$-operator $\{\widehat{\bS},\widehat{\bT}\}$.
\end{proposition}
\begin{proposition} Let $(\mathfrak{g}, \triangleleft, \mathfrak{a})$ be a right filtered Lie module. Consider ${\bf t}\colon \mathfrak{g}\to\mathfrak{a}$ a filtered Lie $\mathcal{O}$-operator. Then, the extension $[\bf t]$ defined by the equation \eqref{eqn:RBop} is a filtered Hopf $\mathcal{O}$-operator.
\end{proposition}

The following theorem is a generalisation of Theorem 4.12 of \cite{goncharov2024formal} about formal integration of Rota-Baxter operators on filtered Lie algebras, but in the case of matching Lie $\mathcal{O}$-operators.
\begin{theorem}{(Formal Integration of matching filtered Lie $\mathcal{O}$-operators)}
\label{thm:integrationmatchingolie}
Let $(\mathfrak{g}, \preliediff, \mathfrak{a})$ be a right filtered Lie module and consider a matching filtered Lie $\mathcal{O}$-operator ${\bm s}, {\bm t} \colon \mathfrak{g}\to\mathfrak{a}$. Then  
\begin{equation}
{\mathcal{T}}(x) = \log(\widehat{\bT}(\exp(x))), \quad 
{\mathcal{S}}(x) = \log(\widehat{\bS}(\exp(x))),\quad x \in \hat{\mathfrak{g}}
\end{equation}
is a matching group $\mathcal{O}$-operator on the formal integration of the right filtered Lie module $(\mathfrak{g}, \vartriangleleft, \mathfrak{a})$.
\end{theorem}
\begin{proof}
We have to check first that the two operators $\mathcal{T}$ and $\mathcal{S}$, separately, are group $\mathcal{O}$-operators. We do not introduce a symbol for the BCH group laws on $\hat{\mathfrak{g}}$ and $\mathfrak{a}$ to compute. We recall that a filtered right Lie module $(\mathfrak{g}, \vartriangleleft, \mathfrak{a})$ yields a filtered right group module $(\mathfrak{g}, \cdot, \mathfrak{a})$ if $\mathfrak{g}$ and $\mathfrak{a}$ are equipped with their BCH group laws. Let $x,y \in \hat{\mathfrak{g}}$,
\begin{align*}
\mathcal{T}(x)\mathcal{T}(y) & =  \log(\widehat{\bT}(\exp(x)))\log(\widehat{\bT}(\exp(y))) \\
					    &=   \log(\widehat{\bT}(\exp(x))\widehat{\bT}(\exp(y))) \\
					    & =  \log(\widehat{\bT}(\exp(x)\star_{\bT} \exp(y))) \\
					    &= \log(\widehat{\bT}(\exp(x \vartriangleleft_{\bm t} \exp(y))\exp(y))) \\
					    &= \log(\widehat{\bT}(\exp(\big(x \vartriangleleft_{\bm t}\exp(y)\big)y))) \\
					    & = \mathcal{T}((x \cdot_{\mathcal{T}} y)y)
\end{align*}
\begin{align*}
\mathcal{T}(x \cdot \mathcal{S}(y^{-1})))\mathcal{S}(y)&=\log(\widehat{\bT}(\exp(x \preliediff \widehat{\bS}(\exp(-y))))) \\
							    &=\log(\widehat{\bT}(\exp(x) \preliediff \widehat{\bS}(\exp(-y)))) \\
							    &= \log(\widehat{\bS}(\exp(y) \preliediff \widehat{\bT}(\exp(-x)))) \\
							    &= \log(\widehat{\bS}(\exp(y \preliediff \widehat{\bT}(\exp(-x))))) \\
							    &= \mathcal{S}(x \cdot \mathcal{T}(y^{-1})))\mathcal{T}(y)
\end{align*}
This concludes the proof.
\end{proof}

\section{Semenov-Tian-Shanskii factorisation and matching $\mathcal{O}$-operators}
\label{sec:SemenovTianShanskii}
This section develops the deformation theoretic perspective on the STS factorisation.
Hereafter, the subset of all $\mathcal{O}$ -Hopf operators on a Hopf module $(H,\preliediff,K)$ will be denoted $G^{\mathcal{O}}_{\Delta}$ (or $G^{\mathcal{O}}_{\Delta}(H,K)$ to eliminate ambiguities). 
\subsection{Guin-Oudom group}
\label{sec:gogroup}

The objective of this section is to give a Lie algebraic interpretation to the recursion \eqref{eqn:RBop} by defining the structure of a post-Lie algebra on $\varepsilon$-co-cocycles from $H$ to $K$. This adds a novel perspective on the Guin--Oudom construction and will be useful in the next section regarding computing the quantization of linear combinations of matching $\mathcal{O}$-operators (a notion introduced subsequently).
The association of a Lie algebra and corresponding group to each right Hopf module is functorial, but we need a slight variation of the notion of morphism between two right Hopf modules. In this section, given two right Hopf modules $(H,\preliediff,K)$ and $(H^{\prime}, \preliediff^{\prime},K^{\prime})$, a morphism with source $(H,\preliediff, K)$ and target $(H^{\prime}, \preliediff^{\prime}, K^{\prime})$ will be a pair of $(\Phi^{\prime}, \Theta)$, $\Phi^{\prime}:H^{\prime}\to H$ and $\Theta:K\to K^{\prime}$ such that for $h^{\prime} \in H^{\prime}$ and $k \in K$
$$
    \Phi^{\prime}(h^{\prime} \preliediff^{\prime} \Theta(k)) 
    = \Phi^{\prime}(h^{\prime}) \preliediff k.
$$
 For the entire section, we fix a right Hopf module $(H,\preliediff,K)$ (recall that both $H$ and $K$ are supposed to be graded connected Hopf algebras). If $H$ is cocommutative, we denote by $G_{\Delta}$ the group of coalgebra morphisms from $H$ to $K$.
We denote by $\delta$ the vector space of \emph{$\varepsilon$-co-cocycles} from $H$ to $K$:
\begin{equation*}
    \delta = \{ \alpha : H \to K, ~\Delta_{K} \circ \alpha = (\eta_K \circ \varepsilon_{H} \otimes \alpha + \alpha \otimes \eta_K\circ\varepsilon_{H}) \circ \Delta_{H },\quad \varepsilon_K \circ \beta = 0\}.
\end{equation*}
It is well known that the space $\delta$ is a Lie algebra for the convolution Lie bracket (see \cite{cartier2021classical}, part I, for various fundamental constructions about Hopf algebras and bialgebras, specifically Chapter 2):
\begin{equation}
    [\alpha, \beta]_{*} = \alpha * \beta - \beta * \alpha, \qquad \alpha,\beta \in \delta.
\end{equation}
Furthermore, under the assumption that $H$ and $K$ are graded co-nilpotent, the exponential convolution $\exp_{*}$ is a bijection from $\delta$ to the convolution group $G_{\Delta}$ of all coalgebra morphisms from $H$ to $K$.

\begin{definitionproposition}
\label{definitionproposition:post-Lie} Let $\alpha, \beta \in \delta$ be two $\varepsilon$-co-cocyles and define for any $h\in H$:
\begin{equation}
    (\alpha \preliediff^{*} \beta)(h) 
    := -\alpha ((\mathrm{id}\preliediff \beta)(h)) 
    = -\alpha(h_{(1)} \preliediff \beta(h_{(2)})).
\end{equation}
Then $\alpha \preliediff^{*} \beta$ is an $\varepsilon$-co-cocycle and $(\delta, \preliediff^{*}, [-,-]_*)$ is a post-Lie algebra.
\end{definitionproposition}
\begin{remark}
We add a minus to the definition of $\preliediff*$ to obtain a post-Lie product with respect to the bracket $[~,~]_{\star}$; otherwise we would have obtained a post-Lie operation with respect to $-[~,~]_{\star}$.
\end{remark}
\begin{proof} Let $\alpha,\beta,\gamma \in \delta$. First, one has to prove that $\alpha \preliediff^{*} \beta$ is an $\varepsilon$-co-cocycle. We pick $h \in H$ (discarding the $\eta_H$ from the computations to lighten them up), 
\begin{align*}
    \Delta ((\alpha \preliediff^{*} \beta) (h))&= -\Delta(\alpha(h_{(1)} \preliediff \beta(h_{(2)}))) \\
    &= -\alpha( (h_{(1)} \preliediff \beta(h_{(2)}))_{(1)}) \otimes \varepsilon_{H}((h_{(1)} \preliediff \beta(h_{(2)}))_{(1)}) \\ 
    &\hspace{3cm}+  \varepsilon_{H}((h_{(1)} \preliediff \beta(h_{(2)}))_{(1)}) \otimes -\alpha((h_{(1)} \preliediff \beta(h_{(2)}))_{(2)})
\end{align*}
where the last inequality follows from the fact that $\alpha$ is a cococycle.
Since $\varepsilon_{H}(h \preliediff k) = \varepsilon_{H}(h)\varepsilon_{K}(k)$, we get for the first term in the last equality of the previous equation:
\begin{align*}
\alpha( (h_{(1)} \preliediff \beta(h_{(2)}))_{(1)}) \otimes \varepsilon_{H}((h_{(1)} \preliediff \beta(h_{(2)}))_{(2)})
&=\alpha( (h_{(1)} \preliediff \beta(h_{(3)})_{(1)})) \otimes \varepsilon_{H}((h_{(2)} \preliediff \beta(h_{(3)})_{(2)})),
\end{align*}
from the fact that $\beta$ is a $\varepsilon$ co-cocycle, the previous expression is equal to
$$
=\alpha( (h_{(1)} \preliediff \beta(h_{(3)}))) \otimes \varepsilon_{H}((h_{(2)} \preliediff \varepsilon_H(h_{(3)})) + \alpha( (h_{(1)} \preliediff \varepsilon(h_{(3)}))) \otimes \varepsilon_{H}((h_{(2)} \preliediff \beta(h_{(3)}))
$$
Since $ \varepsilon(h \preliediff k) = \varepsilon(h)\varepsilon(k)$ and $\varepsilon_K \circ \beta = 0$, we get again that the last expression is equal to
$$
= \alpha( (h_{(1)} \preliediff \beta(h_{(2)}))) \otimes \varepsilon(h_{(3)}).
$$
The computations of the second term is done in the same way. We infer that $ \alpha \preliediff^{*} \beta $ is a $\varepsilon$ co-cocycle.
We now prove the post-Lie axioms. First, one sees that $\preliediff^{*} \beta $ is a derivation with respect to the Lie bracket  $[~,~]_{\star}$. Second,
\begin{align*}
    (\alpha \preliediff^{*} \beta) \preliediff^{*} \gamma 
    &= -(\alpha \preliediff^{*} \beta)(\mathrm{id}\preliediff \gamma) = \alpha((\mathrm{id} \preliediff \gamma) \preliediff \beta) + \alpha (\mathrm{id}\preliediff \beta(\mathrm{id}\preliediff \gamma)) \\
    &=\alpha(\mathrm{id} \preliediff \gamma * \beta) + \alpha \preliediff^{*}(\beta\preliediff^{*} \gamma) \\
    &=-\alpha\preliediff^{*} (\gamma * \beta) + \alpha \preliediff^{*} (\beta \preliediff^{*} \gamma)
    \end{align*}
   which gives 
   $$
 (\alpha \preliediff^{*} \beta) \preliediff^{*} \gamma - \alpha \preliediff^{*}(\beta \preliediff^{*} \gamma) = -\alpha \preliediff^{*}(\gamma * \beta).
   $$
Anti-symmetrizing in $\beta$ and $\gamma$ yields
$$
    (\alpha \preliediff^{*} \beta) \preliediff^{*} \gamma - \alpha \preliediff^{*}(\beta \preliediff^{*} \gamma) -
 \big((\alpha \preliediff^{*} \gamma) \preliediff^{*} \beta - \alpha \preliediff^{*}(\gamma \preliediff^{*} \beta)\big) = \alpha \preliediff^{*}[\beta,\gamma]_{*}.
$$
\end{proof}
The following proposition is a simple consequence of the definitions above.
\begin{proposition} 
Let $(\Phi^{\prime},\Theta)\coloni (H,\preliediff,K) \to (H^{\prime}, \preliediff^{\prime},K^{\prime}) $ be a morphism between two right Hopf modules $(H,\preliediff,K)$ and $(H^{\prime}, \preliediff^{\prime},K^{\prime})$ and denote by $\delta$ and $\delta^{\prime}$ their respective post-Lie algebras introduced in Definition \ref{definitionproposition:post-Lie}. Then
$$    
    \Theta^*\Phi^{\prime}_*:
    \delta \to \delta^{\prime},~\alpha\mapsto \Theta\circ\alpha\circ\Phi^{\prime}
$$
is a morphism between post-Lie modules:
$$
    \Theta^*\Phi^{\prime}_*(\alpha\preliediff^{*} \beta) = \Theta^*\Phi^{\prime}_*(\alpha)\preliediff^{\prime,*}\Theta^*\Phi^{\prime}_*(\alpha)(\beta).
$$
\end{proposition}

For the rest of the paper, we set for two $\varepsilon$-co-cocycles $\alpha,\beta \in \delta$:
\begin{equation}
    \label{eqn:post-Liebrck}
    [\alpha,\beta]_{\#} = [\alpha,\beta]_{*} + \alpha \preliediff^{*} \beta - \beta \preliediff^{*} \alpha
\end{equation}
Our objective is now to define an algebra structure on linear homomorphisms from $H$ to $K$ whose commutator Lie algebra contains $(\delta,[\alpha,\beta]_{\#})$. 

\begin{definitionproposition}[$\#$-product]
\label{def:Hopfsmash}
Assume that $H$ is co-commutative. We define $\#$ to be the bilinear associative product defined over linear maps from $H$ to $K$
\begin{equation}
    (A \# B)(f)
    = A \Big(f_{(1)} \preliediff \antipode_{K}\big(B(f_{(2)})_{(1)}\big)\Big) B(f_{(2)})_{(2)},
\end{equation}
where $A,B\coloni H \to K$ and $f \in H$. The unit of the product $\#$ is the convolutional unit $\varepsilon_H \circ \eta_K$.
\end{definitionproposition}

\begin{proof} 
We prove that $\#$ is an associative product. This follows from the following computation. Pick $A,B,C$ linear maps from $H$ to $K$ and $f \in H$, then
\begin{align*}
    ((A\# B)\# C) (f) 
    &= (A\# B)(f_{(1)} \preliediff \antipode_{K}(C(f_{(2)})_{(1)})) C(f_{(2)})_{(2)} \\
    &=A\Big[[ f_{(1)} \preliediff \antipode_{K}(C(f_{(2)})_{(1)})]_{(1)} \preliediff \antipode_{K}\big(B\big([f_{(1)} \preliediff \antipode_{K}(C(f_{(2)})_{(1)})]_{(2)}\big)_{(1)}\big)\Big]\\
    &\hspace{1cm}  B\big([f_{(1)} \preliediff \antipode_{K}(C(f_{(2)})_{(1)})]_{(2)}\big)_{(2)} C(f_{(2)})_{(2)} \\
    &=A\Big[[ f_{(1)} \preliediff \antipode_{K}(C(f_{(3)})_{(1)(1)})] \preliediff \antipode_{K}\big(B\big([f_{(2)} \preliediff \antipode_{K}(C(f_{(3)})_{(1)(2)})]\big)_{(1)}\big)\Big]\\
    &\hspace{1cm}  B\big([f_{(2)} \preliediff \antipode_{K}(C(f_{(3)})_{(1)(2)})]\big)_{(2)}   C(f_{(3)})_{(2)} \\
    &=A\Big[[ f_{(1)} \preliediff \big[\antipode_{K}(C(f_{(3)})_{(1)(1)})  \antipode_{K}\big(B\big(f_{(2)} \preliediff \antipode_{K}(C(f_{(3)})_{(1)(2)})\big)_{(1)}\big)\big]\Big]\\
    &\hspace{1cm}  B\big([f_{(2)} \preliediff \antipode_{K}(C(f_{(3)})_{(1)(2)})]\big)_{(2)}  C(f_{(3)})_{(2)}\\
    &=A\Big[[ f_{(1)} \preliediff \big[\antipode_{K}\big(B\big(f_{(2)} \preliediff \antipode_{K}(C(f_{(3)})_{(1)(2)})\big)_{(1)}  C(f_{(3)})_{(1)(1)}\big)\big]\Big]\\
    &\hspace{1cm} B\big([f_{(2)} \preliediff \antipode_{K}(C(f_{(3)})_{(1)(2)})]\big)_{(2)} C(f_{(3)})_{(2)} \\
    &=A\Big[[ f_{(1)} \preliediff \antipode_{K}\big(B\big([f_{(2)} \preliediff \antipode_{K}(C(f_{(3)})_{(1)})]\big)_{(1)}  C(f_{(3)})_{(2)}\big)\big)\Big]\\
    &\hspace{1cm}  B\big([f_{(2)} \preliediff \antipode_{K}(C(f_{(3)})_{(1)})]\big)_{(2)}   C(f_{(3)})_{(3)} \\
    &= A \# (B \# C)(f),
\end{align*}
where in the ante-penultimate equality we have used the co-commutativity of $H$.
\end{proof}
\begin{remark} When $H$ and $K$ are cocommutative filtred Hopf algebras, the product $\#$ yields an associative product on the set of linear maps from $H$ to $K$ preserving the filtrations, in fact since 
$$
A \# B = m_{K}((A \otimes B) \circ ({\rm id} \preliediff B \antipode_K) ) \circ \Delta 
$$
where $m_{K}\colon K\otimes K \to K$ is the product on $K$, $A\#B$ will preserve filtrations if the operators $A,B$ preserve the filtrations. In addition, $\widehat{A \# B}=\widehat{A} \# \widehat{B}$.
\end{remark}
\begin{example}
We continue the example \ref{sec:chenfliess}. With the notations we have introduced, the computations we did in \ref{sec:chenfliess} above also prove that :
$
\lambda_c \# \lambda_{c^{\prime}} = \lambda_{c \shuffle c^{\prime}}.
$
We obtain the Theorem 6.1 in \cite{guggilam2022formal} as an immediate corollary of the computations above, with the notations of  \cite{guggilam2022formal}: $(c$ {\verb+@+}$ d) $ {\verb+@+} $d^{\prime} = e_{\lambda_{d^{\prime}}}((e_{\lambda_{d}})(c))=e_{\lambda_d \# \lambda_{d^{\prime}}}(c) = e_{\lambda_{d\shuffle d^{\prime}}}(c) = c$ {\verb+@+}$ (d \shuffle d^{\prime})$.
\end{example}
We should denote by $\mathbb{C}\llbracket \mathcal{P}\rrbracket$ the vector space of formal series on operators in $\mathcal{P}$. The above extends verbatim to $\mathbb{C}\llbracket \mathcal{P}\rrbracket$.
The following proposition is obtained by writing any group-like element of $\mathbb{C}\llbracket \mathcal{P}\rrbracket$ as the exponential of a primitive element and is a Corollary of Proposition \ref{prop:translationidentity}
 \begin{corollaire}
 \label{cor:formulae}
     Whenever $P$ is a group like elements in $\mathbb{C}\llbracket \mathcal{P}\rrbracket$,
     \begin{align}
     [-\bm{\lambda}](P)= I \centerdot P^{-1},\quad [\bm{\rho}](P)= P \centerdot I,\quad  [\rho-\lambda] = [\rho] \# [\lambda]^{-1^{\#}} = P \centerdot I \centerdot P^{-1}
     \end{align}
 \end{corollaire}

\begin{remark}
We might take here the opportunity to highlight connections with free probability theory. A certain transform, called the $T$-transform, plays a major role in the computations of the distribution of the product of two freely independent non-commutative random variables. Given a random variable $a$, $T_a$ is a series in $\mathbb{C}\llbracket{\rm Hom} (B)\rrbracket$ where:
\begin{enumerate}
\item $B$ is a complex unital algebra,
\item ${\rm Hom}(B)$ stands for the endomorphism operad of $B$; ${\rm Hom}(B)(n)={\rm Hom}_{Vect_{\mathbb{C}}}(B^{\otimes n},B),~n\geq 1$.
\end{enumerate}
The $T$-transform $T_a \in \mathbb{C}1\oplus \mathbb{C}\llbracket{\rm Hom}(B)\rrbracket$ of a random variable $a$ is a series on multilinear maps on $B$ with values in $B$ which starts by $1$:
$$
T_a = 1+\sum_{n\geq 1}T_a^{(n)},\quad T_a^{(n)}:B^{\otimes n} \rightarrow B
$$
Given two freely independent random variables, the $T$-transform $T_{ab}$ of the product $ab$ factorizes as follows:
$$
T_{ab}=T_a\circ(T_b\centerdot I \centerdot T_B^{-1})\centerdot T_b,
$$
where $\centerdot$ is the pointwise product of series and $\circ$ is the composition product $\times$:
$$
(A \circ B)^{(n)} (x_1,\ldots,x_n) = \sum_{\substack{k\geq 1,\\ q_1+\cdots+q_k=n}}A_k(B_{q_1}(x_1,\ldots,x_{q_1}),\ldots,B_{q_k}(x_{q_1+\cdots+q_{k-1}+1},\ldots,x_{n})). 
$$
We elaborate on this in the second part of this work.
\end{remark}
It is straightforward that for the above-defined product $\#$, one has $\delta \subset {\rm Com}({\rm End}(H,K), \#)$ (the commutator Lie algebra of ${\rm End}(H,K)$  for the product $\#$). 

We denote by $G^{inv}_{\#}$ the set of all invertible elements in $({\rm End}(H,K), \#)$.
For the following proposition, we specialize the general setting to the one where $H$ is a reduced graded conilpotent coalgebra (and thus the envelope of its primitive elements by the Milnore--Moore--Cartier theorem).
\begin{definitionproposition} 
\label{def:guinoudomgroup}
Let $(\g,\a,\vartriangleleft)$ be a Lie module.
The set $G_{\Delta}$ of all coalgebra morphisms from $\mathcal{U}(\g)$ to $\mathcal{U}(\a)$ is a group for the product $\#$ with the unit $\varepsilon_{\mathcal{U}(\a)} \circ \eta_{\mathcal{U}(\g)}$; $G_{\Delta} \subset G^{inv}_{\#}$ and
$$
A^{-1^{\star}} = \antipode_{\univalgebra} \circ A, \quad A \in G_{\Delta}.
$$
\end{definitionproposition}

\begin{proof} 
The fact that the set $G_{\Delta}$ of all coalgebra morphisms is stable by $\#$ is a simple consequence of the fact that $\preliediff$ and $\mathcal{S}_{K}$ are compatible with the coproducts. It remains to show that any coalgebra morphism $A: \mathcal{U}(\g) \to \mathcal{U}(\a)$ is invertible. This follows from the following: $$ A = \varepsilon_{\mathcal{U}(\a)} \circ \eta_{\mathcal{U}(\mathfrak{g})} + \bar{A}
$$
where $\bar{A}$ is pointwise nilpotent for $\#$: for any $E \in \mathcal{U}(\g)$ there exists $n(E)$ such that $A^{\# n(E)}(E)=0.$
\end{proof}
We use $G^{GO}$ for the group $({G_{\Delta},\#})$.
\begin{proposition}
\label{prop:inverseOhopfoperator}
Let ${\bT}:\mathcal{U}(\g) \to \mathcal{U}(\a)$ be a Hopf $\mathcal{O}$-operator. Then $\bf T$ is invertible for the product $\#$ and ${\bf T}^{\# -1} = {\bf T} \circ \mathcal{S}_{\mathcal{U}(\g)}$.
\end{proposition}
The previous proposition has the following simple corollary.
\begin{corollaire}
Let ${\bf T}\colon \univalgebra \to \mathcal{U}(\mathfrak{a})\in G_{\#}$ be a Hopf $\mathcal{O}$-operator, then
$$
    ({\bf T}^{-1^*})^{-1^{\#}} = ({\bf T}^{-1^{\#}})^{-1^{*}}.
$$
\end{corollaire}
\begin{proof}
We have to prove that ${\bf T} \circ \mathcal{S}_{\univalgebra}$ is a right- and left-inverse to ${\bf T}$ for the product $\# $. First, we compute:
\begin{align*}
    {\bf T} \# ({\bf T} \circ \antipode_{\mathcal{U}(\g)}) 
    &= {\bf T}({\rm id} \preliediff \antipode_{\mathcal{U}(\a)} \circ {\bf T } \circ \antipode_{\mathcal{U}(\g)})  *  ({\bf T} \circ \antipode_{\mathcal{U}(\g)}) \\
    &={\bf T}({\rm id} \preliediff ({\bf T} \circ \antipode_{H_{\bf T}} \circ \antipode_{\mathcal{U}(\g)}))  *  ({\bf T} \circ \antipode_{\mathcal{U}(\g)}) \\
    &={\bf T}({\rm id} \preliediff {\bf T} \circ e_{\bf T})  *  ({\bf T} \circ \antipode_{\mathcal{U}(\g)}) \\
    &=({\bf T} \circ e_{\bf T})  *  ({\bf T} \circ \antipode_{\mathcal{U}(\g)}) \\
    &={\bf T} \circ ((e_{\bf T} \preliediff {\bf T} \circ \antipode_{\mathcal{U}(\g)})\cdot \antipode_{\mathcal{U}(\g)}) \\
    &= {\bf T}({\rm id} \cdot \antipode_{\mathcal{U}(\g)}) \\
    &= {\bf T} \circ (\eta_{\mathcal{U}(\a)}\circ\varepsilon_{\mathcal{U}(\g)}) = \eta_{\mathcal{U}(\a)} \circ \varepsilon_{\mathcal{U}(\g)}\end{align*}
And second:
\begin{align*}
    ({\bf T}\circ \antipode_{\mathcal{U}(\g)}) \# {\bf T} &= {\bf T}(\antipode_{\mathcal{U}(\g)} \preliediff \antipode_{\mathcal{U}(\a)} \circ {\bf T})  *  {\bf T} \\
    &={\bf T}((\antipode_{\mathcal{U}(\g)} \preliediff {\bf T} \circ \antipode_{H_{\bf T}}) *_{\bf T} {\rm id}) \\
    &={\bf T}((\antipode_{\mathcal{U}(\g)} \preliediff ({\bf T} \circ \antipode_{\mathcal{U}(\g)} \circ e_{{\bf T} \circ \antipode_{\mathcal{U}(\g)}}))  *_{\bf T} {\rm id} ) \\
    &= {\bf T}((\antipode_{\mathcal{U}(\g)} \circ e_{\bf T \circ \antipode_{\mathcal{U}(\g)}})  *_{\bf T} {\rm id}) \\
    &={\bf T}(\antipode_{H_{\bf T}}  *_{\bf T}{\rm id}) = {\bf T} \circ (\eta_{\mathcal{U}(\a)} \circ \varepsilon_{\mathcal{U}(\g)}) = {\eta_{\mathcal{U}(\a)}} \circ \varepsilon_{\mathcal{U}(\g)}.
\end{align*}
\end{proof}
\begin{remark} Let $\bT\colon \univalgebra \to \mathcal{U}(\mathfrak{a})$ ba a Hopf $\mathcal{O}$-operator. Then, the following relations hold
\begin{equation*}
    \bT \circ \antipode_{\univalgebra} = \bT^{-1^{\#}},\quad S_{\mathcal{U}(\mathfrak{a})} \circ \bT =  \bT\circ \antipode_{\bT} = \bT^{-1^{\star}}.
\end{equation*}
\end{remark}
\begin{remark}
This is a continuation of Remark \ref{rk:topologicalgp}. It is obvious that the product \# makes sense in the setting of Remark \ref{rk:topologicalgp} as a product between the continuous map from $G$ to $\Gamma$. However, when it comes to evaluating the Lipschitz modulus of the product 
\begin{align*}
d_\Gamma(S\#T(g), S\#T(g')) &= d_{\Gamma}(S(g\cdot T(g)^{-1})T(g),S(g'\cdot T(g')^{-1})T(g')) \\ 
&\leq d_{\Gamma}(S(g\cdot T(g)^{-1}), S(g'\cdot T(g')^{-1}))+d_{\Gamma}(T(g),T(g^{\prime})) \\
&= \sigma_{S}(d_{G}(g\cdot T(g)^{-1}, g'\cdot T(g')) + \sigma_T d_G(g,g') \\
&=\sigma_S (d_G (g,g')+d_G(T(g)^{-1},T(g')^{-1}))+\sigma_T d_G(g,g')
\end{align*}
 When $G$ and $\Gamma$ are matrix subgroups of a group of unitary matrices (equipped with the Hilbert Schmidt distance), $d_G(T(g)^{-1},T(g')^{-1})) = d_G(T(g),T(g'))$ and 
 $$
 d_{\Gamma}(S \# T(g),S\# T(g'))\leq (\sigma_S + \sigma_T + \sigma_S\sigma_T)d_{G}(g,g').
 $$
 Thus, $S\#T$ is a strong contraction map in the case where:
$$
\sigma_S \leq \frac{1-\sigma_T}{1+\sigma_T}\quad(\Leftrightarrow \sigma_T \leq \frac{1-\sigma_S}{1+\sigma_S})
$$
 $$e_{\bT}= g \cdot (\bT^{-1^{\#}}(g))^{-1}$$

\end{remark}
\subsection{Exponential mapping}
We restrict ourselves to co-nilpotent Hopf algebras $H$ and $K$, which are isomorphic to the universal envelopes of their primitive elements by the Cartier-Milnor-Moore theorem.
More concretely, $H = \mathcal{U}(\mathfrak{g})$ and $K = \mathcal{U}(\mathfrak{a})$ with $\mathfrak{g}$ and $\mathfrak{a}$ the Lie algebras of their primitive elements. 
\label{sec:exponentialmapping}
Define the exponential mapping
$\exp_{\#} : \delta \to {\rm Hom}_{\rm Vect}(H,K)$ by:
\begin{equation}
\label{hushexp}
    \exp_{\#}(\alpha) 
    = \varepsilon \circ \eta_{H} + \sum_{n\geq 1} \frac{1}{n!} \alpha^{\# n}.
\end{equation}
Observe that the sum of the right-hand side of the above equality is locally finite since we assumed that $H$ is conilpotent.
Recall the definition of the standard filtration of a right Lie module, see Section \ref{sec:filtered}.
\begin{lemma}
We assume the postLie algebra $(\delta, \vartriangleleft^{*}, [-,-]_{\star})$ endowed with its standard filtration. Then $\delta$ is complete.
\end{lemma}
\begin{proof}
We call $\mathcal{F}_{\bullet}\delta$ the standard filtration of ${\delta}$.
An element of $x\in \hat{\delta}$ (the completion of $\delta$ with respect to $\mathcal{F}_{\bullet}\delta$) is an infinite formal sum $ \alpha = \sum_{i=1}^{\infty} \alpha_i$, with $x_i \in \mathcal{F}_i \delta$. 
We prove recursively that elements of $\mathcal{F}_{i}\delta$ vanish on $\{E_1\cdots E_k : E_j \in \g, j \leq k,\,k\leq i-1\}$. This actually follows from the co-niplotency hypothesis and the definition of $\Delta$. We thus get a well-defined map 
$$
\alpha = \sum_{i=1}^{\infty} \alpha_i \mapsto (h \mapsto \sum_{i=1}^{\infty} \alpha_i(h)) \in \End(H,K),
$$
inverse to the canonical map $\delta \mapsto \hat{\delta}$.
\end{proof}
\begin{proposition} 
The exponential mapping $\exp_{\#}$ defined in \eqref{hushexp} is a bijection from $\delta$ to the group $G_{\#}$.
\end{proposition}

\begin{proof}
We prove that $\exp_{\#}$ is surjective. 
Let $\mathcal{U}(\delta, [-,-]_{*})$ be the envelope of the convolution Lie algebra $(\delta, [-,-]_{*})$. The envelope of $(\delta, [-,-]_{\#})$ is isomorphic to $\mathcal{U}(\delta, [-,-]_{*})$ equipped with the Grossman--Larsson product
$$
    A \#_{\rm fo.} B 
    := (A \preliediff^{*}_{\rm fo.} B_{(1)})B_{(2)},
$$
where we use the subscript $\sharp_fo$ to distinguish betweem the product on $\End(H,K)$ and the product on $\mathcal{U}(\delta, [-,-]_{*})$.
From the universal property of $\mathcal{U}(\delta, [-,-]_{\#})$, we get the existence of an algebra morphism 
$$I:\mathcal{U}(\delta, [-,-]_{\#}) \to ({\rm End}(H,K),\#).$$ 

Recall that we denote by $\widehat{\mathcal{U}}(\delta, [-,-]_{*})$ and $\widehat{\mathcal{U}}(\delta, [-,-]_{\#})$ the completions of the respective universal envelopes (the universal envelopes are considered equipped with their natural filtrations span by monomial on Lie elements of bounded degrees). Also, the two formal exponentials, $\exp_{*_{\rm fo.}}$ and $\exp_{\#_{\rm fo.}}$, are bijections from the completion $\hat{\delta}$ to the set of group-like elements of $\widehat{\mathcal{U}}(\delta, [-,-]_{*})$ and $\widehat{\mathcal{U}}(\delta, [-,-]_{\#})$ (see Propostion \ref{prop:isomorphismexpo}), which are isomorphic as sets.

Pick ${\bf T} \in G_{\#}$. The convolution exponential $\exp_{*}$ is a bijection from ${\delta}$ to the convolution group $G_{\Delta}$ of coalgebra morphisms in ${\rm End}(H,K)$ from $H$ to $K$; this comes from the fact that the coproduct is co-nilpotent; there exists $t \in {\delta}$ such that $\exp_{*}(t)={\bf T}$. 

We can find an element $\Omega(t)\in \hat{\delta}$ such that $\exp_{*_{\rm fo}.}(t)=\exp_{{\#_{\rm fo.}}}(\Omega(t))$ ($\Omega(t)$ is called the post Lie Magnus expansion of $t$ \cite{ebrahimi2023magnus}).

Now, since $\delta$ is complete, $\Omega(t)$ is in $\delta$ and we apply $I$ to get the result.

\end{proof}

We compare the exponential mapping $\exp_\#$ and the Guin--Oudom recursion \eqref{eqn:RBop}. 
Recall the definition of the Solomon elements, see \cite{cartier2021classical} 
$$
{\rm sol}_1 = \log_{\cdot}({\rm id}_{\mathcal{U}(\mathfrak{g})}) = \sum_{n\geq 1} ({\rm id}_{\mathcal{U}(\mathfrak{g})}-\varepsilon)^{\cdot n},\quad {\rm sol}_n = \frac{1}{n!} e_n,\quad e_n = {\rm sol}_1^{\cdot n}
$$
where we recall that $\cdot$ denotes the convoltion product $\cdot$ on ${\rm End}(H,H)$
\begin{theorem}
\label{prop:solande} Let $(\g,\vartriangleleft,\a)$ be a Lie module.
Let ${\bf T}\coloni \univalgebra\to \mathcal{U}(\a)$ be a Hopf $\mathcal{O}$-operator then, with $\Gamma^{\bT} = (\Gamma^{\bT}_n:\mathcal{U}(\g)\to\mathcal{U}(\a))_{n \geq 1}$ the {sequence} defined by
$$
    \Gamma^{\bT}_n := {\bf T} \circ {\rm sol}_n : \univalgebra \to \mathcal{U}(\a):
$$
the following formula holds, for any $f \in H$:
\begin{equation}
    \Gamma^{\bT}_{n+1}(f) = \Gamma^{\bT}_n * \Gamma^{\bT}_1 (f)- \Gamma^{\bT}_n(f_{(1)} \preliediff \Gamma^{\bT}_{1}(f_{(2)})) = \frac{1}{n!}(\Gamma^{\bT}_1)^{\# n} (f)
\end{equation}
Hence,
$$
    \bT = \exp_{\#}(\bT \circ {\rm sol}_1).
$$
\end{theorem}

\begin{proof} 
Let $n\geq 1$ be an integer greater than $1$. First, recall that $e_{n+1}=e_{n}\cdot {\rm sol}_1$ where $\cdot$ denotes the convolution product with respect to the concatenation product over $\univalgebra$. Second, recall that
$$
    \Delta \circ {\rm sol}_1 
    = ({\rm sol}_1 \otimes \varepsilon + \varepsilon \otimes {\rm sol}_1) \circ \Delta
$$
Then, for any $E \in \univalgebra$, one gets
\begin{align*}
    e_n *_{\bT} {\rm sol}_1 (E) 
    &= e_n(E_{(1)}) \preliediff \bT((sol_{1}(E_{(2)}))_{(1)}) \cdot {\rm sol}_1(E_{(2)})_{(2)} \\
    &=e_{n}(E_{(1)})\preliediff \bT({\rm sol}_1(E_{(2)})) + e_{n}(E_{(1)}){\rm sol}_1(E_{(2)})
\end{align*}
Hence,
$$
    e_{n} *_{\bf T} {\rm sol}_1 
    = e_{n+1} + e_{n} \preliediff \bT \circ {\rm sol}_1.
$$
We apply the above formula to get
\begin{align*}
    {\bf T}(e_{n+1}) 
    &= {\bf T}(e_n *_{{\bf T}} {\rm sol}_1 - e_n \preliediff {\bf T}\circ {\rm sol}_1) \\
    &=\Gamma^{\bT}_n * \Gamma^{\bT}_1 - {\bf T}(e_n \preliediff {\bf T} \circ {\rm sol}_1) \\
    &=\Gamma^{\bT}_n * \Gamma^{\bT}_1 - \Gamma^{\bT}_n (\mathrm{id} \preliediff
     \Gamma^{\bT}_1)
\end{align*}
where the last follows from the fact that for any elements $x \in \mathfrak{g}$, $E\in \mathcal{U}(\mathfrak{g})$, 
$
e_n(E \preliediff x) = e_n(E) \preliediff x.
$
In fact, if $e_k \preliediff {\rm sol}_1 = e_k({\rm id} \preliediff {\rm sol}_1)$ for any $k\geq n$, then
$$
e_{n+1} \preliediff x = (e_n\cdot e_1)\preliediff x = (e_n\preliediff x)\cdot e_1 + e_n\cdot (e_1\preliediff x) =e_n({\rm id}\preliediff x)\cdot e_1 + e_n\cdot e_1({\rm id}\preliediff x).
$$
Since $\Delta {\rm id} \preliediff x = {\rm id}\otimes ({\rm id} \preliediff x) \circ  \Delta+  ({\rm id} \preliediff x) \otimes {\rm id} \circ \Delta$, we actually get $e_{n+1}({\rm id \preliediff x}) = e_{n+1}\preliediff x$. The relation for $n=1$ is readily checked using the definition of $e_1 = {\rm sol}_1$. Hence, for any $n\geq 1$,
$$
{\bf T}(e_{n+1}) = \bT(e_n)\star_{\bT}\Gamma^{\bT}_1
$$
This ends the proof.
\end{proof}

We can give a more general result than Theorem \ref{prop:solande}. As an application, we show that the subset of elements of $G_{\#}$ whose inverse with respect to $\#$ can be computed by precomposition with the antipode of $H$ is much larger than the subset of $\mathcal{O}$ - Hopf operators.

\begin{theorem}
Let ${t} \in \delta$ and set, for any integer $n\geq 1$,
$$
    {t}_n = {t} \circ {\rm sol}_{n}~ \textrm{ and }
    {T} = \exp_{\#}(t).
$$
Then, for any integer $p\geq 1$,
\begin{equation*}
T \circ {\rm sol}_p=\sum_{\lambda \vdash p} \frac{1}{|\lambda|!}{\bf t}_{\lambda_1} \#\cdots\#\, {\bf t}_{\lambda_{|\lambda|}},
\end{equation*}
where $\lambda \vdash p$ denotes a partition into an order tuple $(\lambda_1,\ldots,\lambda_{|\lambda|})$ of $p$.
\end{theorem}

\begin{corollaire}
Let $t:\g \to \a$ be a linear map and set ${\bf T}=\exp_{\#}(t\circ sol_1)$. Then
$$
    {\bf T}^{\# -1} = {\bf T} \circ \mathcal{S}_H.
$$
In particular,
$$
    ({\bf T}^{\# -1})^{* -1} = ({\bf T}^{* -1})^{\# -1}.
$$
\end{corollaire}

\begin{proof}
We only treat the case where ${t}_n = 0$ if $n > 1$. The general case follows by the same lines of arguments. First, we set
\begin{align*}
    \gamma_{n,p} 
    = ({t} \circ {\rm sol}_1)^{\# n}\circ {\rm sol}_p
\end{align*}
The objective is to find a recursion on $(\gamma_{n,p})_{n,p}$. Let $n\geq 1$ and $p\geq 1$. By using the definition of $\#$, we get
\begin{align*}
    \gamma_{n+1,p} =  (({t} \circ {\rm sol}_1)^{\# n} * ({t}\circ {\rm sol}_1))\circ {\rm sol}_p-({t} \circ {\rm sol}_1)^{\# n}({\rm id} \preliediff ({t} \circ {\rm sol}_1))\circ {\rm sol}_p
\end{align*}
Second, since the PBW isomorphism is a bijection of \emph{coalgebras}, one has, in fact:
\begin{align*}
    \Delta {\rm sol}_p 
    = \sum_{\substack{a,b \geq 0\\ a+b=p}} {\rm sol}_a \otimes {\rm sol}_b
\end{align*}
Because $sol_1 \circ sol_p=0$ if $p > 1$, one gets:
\begin{align}
\label{eqn:rec}
    \gamma_{n+1,p} 
    = \gamma_{n,p-1} * ({t} \circ {\rm sol}_1)- \gamma_{n,p-1}({\rm id} \preliediff ({t} \circ {\rm sol}_1)) = \gamma_{n,p-1} \# (t \circ sol_1).
\end{align}
For $n > 0$, one has
\begin{align*}
    \gamma_{n+1,1}
    & = (({t} \circ {\rm sol}_1)^{\# n} * ({t}\circ{\rm sol}_1)) \circ {\rm sol}_1 - ({t} \circ {\rm sol}_1)^{\# n}(\mathrm{id}\preliediff ({t} \circ {\rm sol}_1))\circ {\rm sol}_1 \\
    &= 0
\end{align*}
The result follows by unfolding the recursion \eqref{eqn:rec}.
\end{proof}

\subsection{Abstract STS Factorisation}
\label{sec:factorisation}
\begin{theorem}[Semenov-Tian-Shanskii factorisation]
\label{thm:semenovtianshanskii}
Let $\{\bm{s},\bm{t}\}$ be a matching $\mathcal{O}$-operator on a Lie module $(\g,\a,\vartriangleleft)$. With the notation introduced so far:
$$
[\bm{s} + \bm{t}] = [\bm{s}] \# [\bm{t}] = [\bm{t}] \# [\bm{s}].
$$
\end{theorem}

\begin{proof}
Let $\{\bf{s},\bf{t}\}$ be a matching $\mathcal{O}$ Lie operator. It is easy to see that the matching relations \eqref{eqn:MLie} implies the vanishing of the Lie bracket of ${\bf s}\circ {\rm sol}_1$ with ${\bf t}\circ {\rm sol}_1$:
$$
[{\bf s}\circ {\rm sol}_1,{\bf t}\circ {\rm sol}_1]_{\#}=0.
$$
In fact, 
\begin{align}
({\bf s}\circ {\rm sol}_1) \# ({\bf t}\circ {\rm sol}_1) &= -({\bf s}\circ {\rm sol}_1)( {\rm id} \preliediff {{\bf t}\circ {\rm sol}_1}) + ({\bf s}\circ {\rm sol}_1) \cdot ({\bf t}\circ {\rm sol}_1) \\
&= -{\bf s}({\rm sol}_1 \preliediff {{\bf t}\circ {\rm sol}_1}) + ({\bf s}\circ {\rm sol}_1) \cdot ({\bf t}\circ {\rm sol}_1)
\end{align}
Thus for any, $E \in \mathcal{U}(\g)$,
\begin{align}
({\bf s}\circ {\rm sol}_1) \# ({\bf t}\circ {\rm sol}_1)(E) =  -{\bf s}({\rm sol}_1(E_{(1)}) \preliediff {{\bf t}\circ {\rm sol}_1}(E_{(2)})) + ({\bf s}\circ {\rm sol}_1(E_{(1)})) \cdot ({\bf t}\circ {\rm sol}_1{E_{(2)}})
\end{align}

Setting $x = E_{(1)}, y=E_{(2)}$, we get 
$$
({\bf s}\circ {\rm sol}_1 \#{\bf t}\circ {\rm sol}_1)(E) = \sum {\bf s}(x){\bf t}(y) -{\bf s}(x \preliediff {\bf t}(y))
$$

Since $\sum x \otimes y = \sum y \otimes x$, we also get 
$$
({\bf t}\circ {\rm sol}_1 \#{\bf s}\circ {\rm sol}_1)(E) = \sum {\bf t}(y){\bf t}(x) -{\bf t}(y \preliediff {\bf t}(x))
$$
and the result follows.
\end{proof}

\begin{remark}{(Recovering the classical Semenov-Tian-Shanskii factorisation.)} 
Recall that a Rota-Baxter operator ${\bf r}\colon \mathfrak{g}\to \mathfrak{g}$ satisfies: 
\begin{align}
[{\bm{r}}(x),{\bm{r}}(y)]
    &={\bm{r}}([\hat{{r}}(x),y]_{\mathfrak{g}}+[x,{\bm{r}}(y)]_{\mathfrak{g}}+[x,y]), \quad x,y \in \mathfrak{g}
\end{align}  
If defining $\hat{r} = -{\rm id}_{\mathfrak{g}} - {\bf r} $, one has also:
\begin{align*}
    [\hat{\bm{r}}(x),\hat{\bm{r}}(y)]
    &=\hat{\bm{r}}([\hat{\bm{r}}(x),y]_{\mathfrak{g}}+[x,\hat{\bm{r}}(y)]_{\mathfrak{g}}+[x,y])\\
    [\hat{\bm{r}}(x),{\bm{r}}(y)]
    &={\bm{r}}([\hat{\bm{r}}(x),y]_{\mathfrak{g}}+\hat{\bm{r}}[x,{\bm{r}}(y)]_{\mathfrak{g}})
\end{align*}
for any $x,y \in \mathfrak{g}$.
Recall that ${\bm{R}}$ and ${\hat{\bm{R}}}$ are the Rota--Baxter--Hopf extending and coextending ${\bm{r}}$ and $\hat{\bm{r}}$ to the Hopf algebra $\univalgebra$.
Since ${\bm{r}}+\hat{{\bm{r}}}=-{\rm id}_{\mathfrak{g}}$ and ${\hat{\bm{R}}}\#{\bm{R}}$ is a Rota--Baxter--Hopf operator, we have $\hat{\bm{R}}\# {\bm{R}}=\mathcal{S}_{\mathcal{U}(\mathfrak{g})}$. We obtain the following factorisation for any $X \in \univalgebra$:
\begin{equation}
    X = (\mathcal{S}_{\univalgebra}\circ \hat{\bm{R}})(e_{\bm{R}}^{-1}(X_{(1)})) \mathcal{S}_{\univalgebra}({\bm{R}}(X_{(2)})).
\end{equation}
This last formula must be compared with Corollary 23 of \cite{EBRAHIMIFARD201719} wherein $R_+=\mathcal{S}_{\univalgebra}\circ \hat{\bm{R}}$ and $R_-={\bm{R}}$.
Let $x \in \mathfrak{g}$, then we get $$ \mathcal{S}_{\univalgebra}(\exp_{\star_{\bm{R}}}(x))
    ={\hat{\bm{R}}}\# {\bm{R}}(\exp_{\star_{\bm{R}}}(x))
    =\hat{\bm{R}}(e_{\bm{R}}^{-1}\exp_{\star_{\bm{R}}}(x))\exp_{\star_{\bm{R}}}{\bm{r}}(x).
$$
Recall that we have the following formula for the right action $\preliediff$ (the Adjoint action) of $\univalgebra$ over itself, for any element $X, Y\in\univalgebra$: 
$$
    X \preliediff Y 
    = \mathcal{S}_{\univalgebra}(Y_{(1)})XY_{(2)}.
$$
It is not difficult to see that: 
\begin{equation}
\label{eqn:expreflection}
    \exp_{\bf \hat{R}}(x)
    =\exp_{\star{\bm{R}}}(x)\preliediff_{\star{\bm{R}}}\exp_{\star_{\bm{R}}}(-x).
\end{equation}
The formula $e_{\bm{R}}^{-1}=\mathrm{id}\preliediff {\bm{R}}^{ -1^{\star}}=\mathrm{id}\preliediff \mathcal{S}_{\univalgebra}\circ {\bm{R}}$ yields, for any $x \in \mathfrak{g}$: 
$$
    e_{\bm{R}}^{-1}(\exp_{\star_{\bm{R}}}(x))
    =\exp_{\star_{\bm{R}}}({\bm{R}}(x))\preliediff \mathcal{S}_{\univalgebra}\circ{\bm{R}}(\exp_{\star_{\bm{R}}}({\bm{R}}(x)))
    =\exp_{\star_{\bm{R}}}({\bm{R}}(x))\preliediff_{\bm{R}}\exp_{\star_{\bm{R}}}(-{\bm{R}}(x))
$$
The last formula together with \eqref{eqn:expreflection} yield:
\begin{equation}
\label{eqn:ttransformreflection}
    e_{\bm{R}}^{-1}(\exp_{\star_{\bm{R}}}(x))
    =\exp_{\star_{\hat{\bm{R}}}}(x).
\end{equation}
Finally, we obtain:
$$
    \mathcal{S}_{\univalgebra}(\exp_{\star_{\bm{R}}}(x))
    =\exp_{\star_{\bm{R}}}({\hat{\bm{r}}}((x)))\exp_{\star_{\bm{r}}}({\bm{r}}(x)).
$$
This last equation implies the equality
$$
    \exp_{\bm{R}}(x)
    =\exp(-{\bm{r}}(x))\exp(-{\hat{\bm{r}}}(x)).
$$

\end{remark}

\begin{theorem}
\label{thm:compinverse}
Assume that $H$ is cocommutative. The compositional inverse of the $T$-transform yields a group morphism from the group $G_{\#}$ to the group $G_{\circ}$ of coalgebra morphisms on $H$ equipped with the compositional product;
\begin{equation}
    e^{-1^{\circ}} : G_{\#}\to \big(\mathrm{End}_{\Delta}(H,H),\circ\big).
\end{equation}
\end{theorem}
Theorem \ref{thm:compinverse} admits the following reformulation. The group $G_{\circ}$ acts via pre-composition on the convolution group $G_{*}$ via group morphisms,
$$
    (A * B) \cdot X 
    = (A * B) \circ X = (A \circ X) * (B \circ X).
$$
Then Theorem \ref{thm:compinverse} together with Proposition \ref{prop:inversett} implies that $e^{-1}$ is a $\mathcal{O}$-operator on the group module $(G_{*},\cdot,G_{\circ})$. And, with the notations put in place in the previous Section, for any $A,B \in G_{*}$
$$
    A \# B 
    = A *_{e^{-1^{\circ}}} B,
$$
since $S_K \circ A = A^{-1^{*}}$.

\begin{remark}
From \cite{li2022post}, a set-solution to the Quantum Yang Baxter equation is given by
$$
    R_{\#}:G_{\#}\times G_{\#}\to G_{\#} \times G_{\#},~R_{\#}(\bS,\bT) = (\bS \# \bT \# (\bS \circ e_{\bT}^{-1^{\circ}})^{-1^{\#}}, \bS \circ e_{\bT}^{-1^{\circ}} )
$$
\end{remark}
\section{Abstract STS Factorisation in the operator-valued free probability theory}
\label{sec:operatorvaluedfreeproba}
The objectives of this Section are, first, to re-cast the various of a non-commutative disitribution in non-commutative probability (such as the $T$-transform, the cumulants transform $K$, the $H$-transform) and their properties with respect to the multiplication of independent (free, monotone but also conditionally free or conditionally independent) random variables by using the notions $\mathcal{O}$-groups and mathching $\mathcal{O}$ we have introduced. Second, we explain how the notion of cross product of Hopf algebras yields a neat analogy between the definitions of  transforms in free and monotone probability and their conditional counterparts (which so far only were defined in the setting of scalar non-commutative probability theory). Finally, we continue exploring the geometry of $\mathcal{O}$-operators by providing some abstract constructions motivated by non-commutative probability theory, so we have chosen to include them here. 
We refer to \cite{speicher1998combinatorial} for basic definitions about operator-valued free probability theory and to \cite{dykema2006stransform,dykema2007multilinear} for definition of the various transforms used in this section (we recall relevant definitions below though).
Through the section, we consider an operator-valued algebraic probability space $(\mathcal{A}, E,\mathcal{B})$.

Let $a,b \in \mathcal{A}$ two random variables and suppose that $\mathbb{E}(a)=\mathbb{E}(b)=1$. 

Consider their cumulants series: 
$$
K_a = 1+R_a = 1+\sum_{n\geq 1} R^{(n)}_a\in \mathbb{C}\llbracket{\rm End}\,B\rrbracket,~K_b=1+\sum_{n\geq 1} R^{(n)}_b \in \mathbb{C}\llbracket{\rm End}\,B\rrbracket,
$$
with for $x\in \{a,b\}$, $b_1,\ldots,b_n \in B$
$$
R^{(n)}_x \in {\rm Hom}(B^{\otimes n},B),\quad R^{(n)}_x(b_1,\ldots,b_n)=\kappa_{n+1}(ab_1,ab_2,\ldots,b_na)
$$
and $\kappa_{n}$ denotes the $n^{\rm th}$ free cumulant. 
In the following, we use the notation $I={\rm id}_{B}$.

Recall that we denote by $\mathbb{C}\llbracket{\rm End}\,B\rrbracket$ the formal series with entries in the endomorphisms operad of $B$:
\begin{align}
\alpha \in \mathbb{C}\llbracket{\rm End} \,B\rrbracket, \quad \alpha = \sum_{n\geq 1} \alpha^{(n)},\quad \alpha^{(n)}\colon B^{\otimes n} \to B.
\end{align}
We denote by $\mathbb{C}_0\llbracket {\rm End} B\rrbracket$ the group supported by formal series in $\mathbb{C}\oplus \mathbb{C}\llbracket{\rm End} \,B\rrbracket$ starting with $1$ and equipped with the Cauchy product $\centerdot$. We let $\Gamma$ the group supported by formal series in $\mathbb{C}\llbracket{\rm End} \,B\rrbracket$ starting with the unit $I$ and equipped with the composition product $\circ$:
$$
(\alpha \circ \beta)^{(n)}=\sum_{\substack{q_1+\cdots q_k=n, \\ k,q\geq 1}} \alpha^{(k)}\circ (\beta^{(q_1)}\otimes\cdots\otimes\beta^{(q_k)}). 
$$
The $T$-transform $T_x \in \mathbb{C}_{0}\llbracket {\rm End} B\rrbracket$of a random variable $x\in\mathcal{A}$ is defined recursively via the fixed point equation:
$$
K_x = T_x \circ (I\centerdot K_x)
$$
It has been proved that \cite{dykema2006stransform}
$$
T_{ab} = T_a \circ ({T_b\centerdot I\centerdot T^{-1^{\centerdot}}_b}) \centerdot T_b
$$
Recall that $(\mathbb{C}_0[{\rm End} B],\cdot,\Gamma)$ is a right group-module with the right action $\cdot$ defined by, for any $\alpha,\beta \in \CC\llbracket{\rm End}\,B\rrbracket$ :
$$
(1 + \alpha) \cdot \beta = 1 + \alpha\circ\beta.
$$
Generic elements of $\mathbb{C}_0[{\rm End} B]$ will be denoted $A,B\ldots$ while elements of $\Gamma$ will be denoted $\alpha,\beta$...
The right translation $\bm{\rho}$ by the unit is a group $\mathcal{O}$-operator, as well as $\bm{\lambda}\circ \mathcal{S}$ $(=\bm{\lambda}^{-1^{\#}})$ where $\mathcal{S}$ denotes the inversion map with respect to the group law of $G$ and $\lambda$ is the left translation by the unit $I = {\rm id}_B$ (the operator $\bm{\lambda}$ alone is a group  $\mathcal{O}$-operator but with respect to opposite group $G^{op}$). The pair $(\bm{\lambda}^{-1^{\#}}, \bm{\rho})$ is a matching group $\mathcal{O}$-operator , as proved in Example \ref{sec:transunit}:
$$
\bm{\rho} \colon A \mapsto A \centerdot I,\quad \bm{\lambda}^{-1^{\#}} \colon A \mapsto I\centerdot A^{-1^{\centerdot}}.
$$
This implies, see Section \ref{sec:Kupershmidt}, that $\mathbb{C}_0\llbracket {\rm End B} \rrbracket$ can be endowed with three different group products, $\star_{\bm{\lambda}^{-1^{\#}}}$, $\star_{\bm{\rho}},$  and $ \star_{\bm{\lambda}^{-1^\#}\# \bm{\rho}}$, for any $A,B \in \mathbb{C}_0\llbracket {\rm End B} \rrbracket$:
$$
A \star_{\bm{\lambda}^{-1^{\#}}} B = (A\cdot (I\centerdot B^{-1^{\centerdot}} ))\centerdot B,\quad 
A \star_{\bm{\rho}} B = (A\cdot (B\centerdot I) ))\centerdot B
$$
and 
$$
A \star_{\bm{\lambda}^{-1^{\#}}\#\bm{\rho}} B = (A \cdot (B\centerdot I\centerdot B^{-1^{\centerdot}})) \centerdot B.
$$
The fact that $(\bm{\lambda}^{-1^{\#}}, \bm{\rho})$ is a matching $\mathcal{O}$-operator implies the following important relations, see Theorem \ref{thm:semenovtianshanskii}:
\begin{align}
\label{eqn:subordination}
A \star_{\bm{\lambda}^{-1^{\#}}\#\bm{\rho}} B 
&= (A \cdot_{\bm{\lambda}}( B^{-1^{{\star_{\rho}}}}))\star_{\bm{\rho}} B = (A \cdot_{\bm{\rho}^{-1^{\#}}} (B^{-1^{\bm{\star_{\lambda^{-1^{\#}}}}}})) \star_{{\bm{\lambda^{-1^{\#}}}}} B 
\end{align}
We shall now indicate how those three products are related to non-commutative convolutions (both additive and multiplicative).
To begin with, in the context operator-valued monotone multiplicative convolution, it is customary to define the $H$-transform, which we will denote $H_a$ (and $H_b$) defined by (again in our notations):
\begin{align}
\label{eqn:htransform}
H_a = (T^{-1^{\centerdot}}_a \cdot (\frac{I}{1-I}))^{-1^{\star_{\bm{\lambda}}}}= (\Sigma_a)^{-1^{\star_{\bm{\lambda}}}},\quad H_{ba} = H_a \star_{\bm{\lambda}}H_b .
\end{align}
whenever $a-1$ and $b-1$ are monotone independent and where we have denoted 
$$
\Sigma_a = T^{-1^{\centerdot}}_a \cdot \frac{I}{1-I}.
$$
Let us see how $H_a$ can be expressed in term of the free cumulants series $K_a$ of $x$. First, by definition, the formal series $H_a$ satisfies the fixed point equation:
\begin{align}
H_a = (T_a\cdot \frac{I}{1-I}) \cdot_{\lambda} H_a
\end{align}
This fixed point equation implies that $H_a$ is also the unique solution $A_a$ to
$$
A_a = T_a \cdot ((1-I\centerdot H_a)^{-1}\centerdot(I\centerdot A_a)).
$$
On the other hand,
\begin{align*}
K_a\cdot ((1-I\centerdot H_a)^{-1} \centerdot I) &= T_a \cdot ((I\centerdot K_a) \circ ((1-I\centerdot H_a)^{-1}\centerdot I)) \\
&=  T_a \cdot ((1-I\centerdot H_a)^{-1}\centerdot I \centerdot K_a\circ ((1-I\centerdot H_a)^{-1} \centerdot I))
\end{align*}
And thus we derive the following formula for $H_a$
\begin{align}
\label{eqn:formulahx}
H_a = A_a = K_a\cdot ((1-I\centerdot H_a)^{-1} \centerdot I) = K_a \cdot_{\bm{\rho}}(1-I\centerdot H_a)^{-1}.
\end{align}
Let us stop a moment on the relations \eqref{eqn:htransform}. We observe that the relations between $\Sigma_a$ and $H_a$ on one hand is the same as the relation between $T_a$ and $K_a$ on the other hand, a double inversion, namely. 
Additionally pre-composition by the map $\frac{I}{1-I}$ (that is the action of $\frac{I}{1-I}$ through $\cdot$) is equivariant with respect the action $\cdot_{\bm{\lambda^{-1^{\#}}}\#\bm{\rho}}$:
$$
(A \cdot_{\bm{\lambda^{-1^{\#}}}\#\bm{\rho}}B)\cdot (\frac{I}{1-I} )= (A \cdot (\frac{I}{1-I}))\cdot_{\bm{\lambda^{-1^{\#}}}\#\bm{\rho}}(B\cdot (\frac{I}{1-I})).
$$
The above equivariance relation can be proved by using the explicit formulae for the product $\star_{\bm{\lambda}\#\bm{\rho}}$, see Proposition \ref{thm:tttransformtwoantipodes}.
Finally, the relations \eqref{eqn:subordination} can be described using the concept of subordination in non-commutative probability.
In fact, whenever $a,b$ are free independent then:
\begin{align*}
    H_{ab} = \Sigma_{ab}^{-1^{\star_{\bm{\lambda}}}}= ((\Sigma_a \cdot_{\bm{\rho}} (\Sigma_y^{-1^{\bm{\star_{\lambda}}}}))\star_{\bm{\lambda}} \Sigma_b )^{-1^{\star_{\bm{\lambda}}}} = H_b \star_{\lambda} (\Sigma_a \cdot_{\bm{\rho}}(\Sigma_y^{-1^{\star_{\bm{\lambda}}}}))^{-1^{\star_{\bm{\lambda}}}}).
\end{align*}
Setting $H_{a \dashv b} :=  (\Sigma_a \cdot_{{\bm{\rho}}} (\Sigma_b^{-1^{\star_{\bm{\lambda}}}}))^{-1^{\star_{\bm{\lambda}}}} $, we see that $H_{a \dashv b}$ satisfies the fixed point equation :
\begin{align*}
H_{a \dashv b} = (\Sigma^{-1}_x \cdot_{\bm{\rho}}(\Sigma_b^{-1^{\star_{\bm{\lambda}}}})) \cdot_{\bm{\lambda}} H_{a \dashv b} &= ((\Sigma_a^{-1}\cdot_{\bm{\lambda}} \Sigma^{-1^{\star_{\bm{\lambda}}}}_a)\cdot ({\bm{\lambda}}(\Sigma_a)\circ{\bm{\rho}}(\Sigma_b^{-1^{\star_{\bm{\lambda}}}})) \cdot_{{\bm{\lambda}}} H_{a \vdash b} \\
&=\Sigma_a^{-1^{\star_{\bm{\lambda}}}}\cdot (\bm{\lambda}(\Sigma_a) \circ {\bm{\rho}}(\Sigma^{-1^{\star_{\bm{\lambda}}}}_b)\circ \bm{\lambda}(H_{a\vdash b})).
\end{align*}
We use the relations \eqref{eqn:matchinggroup} to write :
\begin{align*}
(\bm{\lambda}(\Sigma_a) \circ {\bm{\rho}}(\Sigma^{-1^{\star_{\bm{\lambda}}}}_y)\circ \bm{\lambda}(H_{a\vdash b})) &= {\bm{\rho}}(\Sigma^{-1^{\star_{\bm{\lambda}}}}_b \cdot_{\bm{\lambda}} ( (\Sigma_a \cdot_{\bm{\rho}} \Sigma^{-1^{\star_{\bm{\lambda}}}}_b)^{-1^{\star_{\bm{\lambda}}}} ))\circ \bm{\lambda}(\Sigma_a \cdot_{\bm{\rho}} \Sigma^{-1^{\star_{\bm{\lambda}}}}_b) \circ \bm{\lambda}(H_{a\vdash b}) \\
&={\bm{\rho}}(\Sigma^{-1^{\star_{\bm{\lambda}}}}_b \cdot_{\bm{\lambda}} H_{a\vdash b})
\end{align*}
We obtain finally the following fixed point equation :
\begin{align}
\label{eqn:subordinationfixedpointeqn}
    H_{a\vdash b} = H_a \cdot_{\bm{\rho}} (H_b \cdot_{\bm{\lambda}} H_{a\vdash b})
\end{align}
We develop further the use of our formalism toward conditional independence in the next sections. 
\subsection{Conditional independences}
\label{sec:condindependence}
We first give the definitions of conditional free and monotone independence. We refer to \cite{belinschi2012infinite,popa2015multiplication,popa2013non} and \cite{popa2012realization} for the definitions in the context of operator-valued probability spaces and to \cite{hasebe2011conditionally,bozejko1996convolution} for their definitions in the scalar case.

Let $(\mathcal{A}, \varphi, \psi, B)$ a conditional operator-valued probability space; $\mathcal{A}$ is an unital complex algebras, $\mathcal{B}$ is acting on $\mathcal{A}$ from the right and from the left, $\varphi$ and $\psi$ are two $B$-bimodule morphisms. 

We say that two variables $a_1$ and $a_2$ are conditionally free if they are \emph{free with respect to $\psi$} and for any $p_1,\ldots p_n \in\langle a_1\rangle \cup \langle a_2 \rangle$:
\begin{equation*}
    \varphi(p_1\cdots p_n) = \varphi(p_1)\cdots \varphi(p_n)
\end{equation*}
whenever $p_j\in\mathcal{A}_{i_j},\psi(p_j)=0$ and $i_j\neq i_{j+1}$.

We say that the two variables $a_1$ and $a_2$ are \emph{conditionally monotone independent} if they are \emph{monotone independent with respect to $\psi$} and for any $p_1,\ldots,p_n \in \langle a_1 \rangle_0 \cup \langle a_2\rangle_0$, $p_i \in \mathcal{A}_{i_j}$, $i_j \neq i_{j+1}$
$$
\varphi(p_1\cdots p_n)= \varphi(p_1)\varphi(p_2\cdots p_n),\quad \varphi(p_1\cdots p_n)= \varphi(p_1p_2\cdots )\varphi(p_n)
$$
and for any $2 \leq i \leq n-1$:
$$
\varphi(p_1\cdots p_n) = \varphi(p_1\cdots p_{i-1})(\varphi(p_i)-\psi(p_i))\varphi(p_{i+1}\cdots p_n) + \varphi(p_1\cdots p_{i-1}\psi(p_i)p_{i+1}\cdots p_n)
$$
Given $\x = \{x_1,\ldots,x_n\} \in \mathcal{A}$, we denote by $\varphi_x$ and $\psi_x$ the joint distribution of $\x$ with respect to $\varphi$ and $\psi$, respectively:
\begin{align}
    \varphi_{\x}(p(X_1,\ldots,X_n)) = \varphi(p(\x)),\quad \psi_{\x}(p(X_1,\ldots,X_n)) = \psi(p(\x))
\end{align}
for any non-commutative polynomial $p\in \mathbb{C}[X_1,\ldots,X_n]$. Then given $a_1$ conditionally monotone independent from $a_2$, there exists two conditionally free random variables $\tilde{a_1},\tilde{a_2}$ in some operator-valued conditional probability space $(\tilde{\mathcal{A}}, \tilde{\varphi}, \tilde{\psi}, \mathcal{B})$ whose distributions under $\varphi$ are equal to the distribution of $a_1$ and $a_2$ under $\varphi$ and 
$$
\psi_{\tilde{a}_1}(p) = p(1),\quad \psi_{a_2}(q)=\psi_{\tilde{a}_2}(q),\quad p \in \mathbb{C}[X_2],~q \in \mathbb{C}[X_1]
$$
such that the following equality between linear functions on $\mathbb{C}[X_1,X_2]$ holds:
\begin{align}
    \varphi_{a_1,a_2} = \varphi_{\tilde{a}_1, \tilde{a}_2}.
\end{align}
\subsection{Cross products and conditional independences}
\label{sec:crossproduct}
We leverage the interpretation that we gave for the $T$-transform in Voiculescu's free probability theory, namely as the composition of two antipodes, see Proposition \ref{thm:tttransformtwoantipodes}, to define the relevant $T$-transform in the setting of operator-valued \emph{conditional freeness}. In particular, from the Hopf-algebraic perspectiv, going to this conditional setting among to taking the \emph{crossed product} of the Hopf algebra $H_{\bm{\lambda}^{-1^{\#}}\#\bm{\rho}}$ with the Hopf algebra $H_{\varepsilon}$, 
We recall the \emph{cross product construction}\footnote{\url{https://ncatlab.org/nlab/show/crossed+product+algebra}} in the particular setting of cocommutative Hopf algebras and of groups. Let $(H,\preliediff,K)$ be a right Hopf module. 
The \emph{crossed product} of $K$ and $H$ over the action $\preliediff$ is the Hopf algebra denoted $K \ltimes_{\preliediff} H$, which is equal as a coalgebra to $(K \otimes H, \tau_{23}\circ\Delta_K \otimes \Delta_H)$ and is equipped with the following unital associative product:
\begin{equation}
(k^a \ltimes h^a) (k^b \ltimes h^b) =   k^ak^b_{(1)} \ltimes (h^a \preliediff k^b_{(2)}) h^b. 
\end{equation}
On easily checks that $1_K \otimes 1_H$ is the unit for the above defined product.
The antipode of the crossed product $K\ltimes_{\theta} H$ Hopf algebra $S_{K\ltimes_{\preliediff} H}:K\ltimes_{\preliediff} H \to K\ltimes_{\preliediff} H $ satisfies the following formula:
\begin{align*}
\mathcal{S}_{K\ltimes_{\preliediff}H}(k\ltimes h) = \mathcal{S}_K(k_{(1)}) \otimes (\mathcal{S}_{H}(h) \preliediff \mathcal{S}_K(k_{(2)}) ).
\end{align*}
If $(G,\cdot, \Gamma)$ is a right group module, the crossed product of $G$ and $\Gamma$ over $\cdot$ is the group supported by the set $\Gamma \times H$ with the product defined by 
$$
(k^a \ltimes h^a)
(k^b \ltimes h^b) =   k^ak^b_{(1)} \ltimes (h^a \cdot k^b_{(2)}) h^b. 
$$
\subsubsection{Conditional $T$-transform}
Recall the definition of the Hopf algebra $H_{\bf T}$, given a $\mathcal{O}$ Hopf operator $\bT$ over a right Hopf module $(K,H,\preliediff)$ and the definition of $\preliediff_{\bf T}$, a right action of the Hopf algebra $H_{\bf T}$ on the Hopf algebra $H$. 
Inspired by applications to non-commutative probability developed below, we define the relative $e$-transform of $\bf T$ conditionally to the convolution unit $\eta_K\circ \varepsilon_{H}$ of $H$.
Notice that in our notations, $H_{\eta_K\circ \varepsilon_H}=H$ ($\eta_K\circ \varepsilon_H$ is a Hopf $\mathcal{O}$-operator on the right Hopf-module $(K,H,\preliediff)$ ) as Hopf algebras.

\begin{definition}[Relative $e$-transform] Let $\bT: H\to K$ be a Hopf $\mathcal{O}$-operator on the right Hopf module $(H,K,\preliediff)$.
\label{def:relativettransform}
Define the morphism of coalgebra 
$$e_{\bT\,|\,\varepsilon}\colon H\otimes H \to H \otimes H$$ 
by, for any $k\otimes h \in H\otimes H$
\begin{align*}
e_{\bT\,|\,\varepsilon}(k\otimes h) &:= \mathcal{S}_{H_{\bT} {\ltimes}_{\preliediff_{\bT}} H} \circ (\mathcal{S}_{H} \otimes \mathcal{S}_{H})(k \otimes h) \\
& = (S_{H_{\bT}} \circ \mathcal{S}_{H})(k_{(1)}) \,\otimes (h \preliediff_{\bT}  ((\mathcal{S}_{H_{\bT}} \circ \mathcal{S}_{H})(k_{(2)}))) \\
& =e_{\bT}(k_{(1)}) \,\otimes (h \preliediff_{\bT}  e_{\bT}(k_{(2)}))
\end{align*}
\end{definition}
Remark that whenever $h$ and $k$ are group-like elements (in a adequate completion of $H$), the following formula is in hold:
\begin{align*}
e_{\bT \, | \, \varepsilon}(k\otimes h) = (k^{-1})^{-1^{\star_{\bT}}}\otimes h \cdot_{\bT} (k^{-1})^{-1^{\star_{\bT}}}.
\end{align*}
The same formula can be used to define the relative $T$-transform of a group $\mathcal{O}$-operator  in the setting of $\mathcal{O}$-group and crossed products of groups. More precisely, with $(G,\cdot,\Gamma)$ a right group-module, and $\bT\colon G \to \Gamma$ a $\mathcal{O}$- group operator, with the notations introduced in Definition, for any $(k,h)\in G \times G$ :
\begin{align}
    e_{\bT\,|\,\varepsilon}(k,h):=\big((k^{-1})^{-1^{\star_{\bT}}}, h \cdot_{\bT} (k^{-1})^{-1^{\star_{\bT}}}\big) \in G_\bT \times G.
\end{align}
Recall the definition of the product in $G_\bT \ltimes_{\cdot_\bT} G$ : 
\begin{align}
\label{eqn:producttwistedgp}
(k_1 \ltimes h_1)(k_2 \ltimes h_2) = k_1\star_\bT k_2 \ltimes ((h_1\cdot_\bT k_2)h_2 ),\quad k_1,k_2 \in G_{\bT},\quad h_1,h_2 \in G.
\end{align}
Notice that $\mathcal{S}_H \otimes \mathcal{S}_H = \mathcal{S}_{H_\varepsilon \ltimes_{\preliediff_{\varepsilon}}} H$. If $\bS$ and $\bT$ are two Hopf $\mathcal{O}$-operaors, observe that we can consider as well the composition between the antipodes of the Hopf algebras $H_{\bT} {\ltimes}_{\preliediff_{\bT}} H$ and $H_{\bS} {\ltimes}_{\preliediff_{\bS}} H$ since both are supported on the tensor product $H\otimes H$ and define 
\begin{align}
e_{\bT | \bS}(k\otimes h) &:= \mathcal{S}_{H_{\bT} {\ltimes}_{\preliediff_{\bT}} H} \circ \mathcal{S}_{H_{\bS} {\ltimes}_{\preliediff_{\bS}} H^{op}} (k\otimes h) \nonumber \\ 
&= \mathcal{S}_{H_\bT} \circ \mathcal{S}_{H^{op}_\bS}(k_{(1)})\otimes h \preliediff ((\bS\circ \mathcal{S}_{H^{op}_\bS})\star (\bT \circ \mathcal{S}_{H_\bT} \circ \mathcal{S}_{H^{op}_\bS}))(k_{(2)})
\end{align}
where we recall that $\star$ is the convolution product between linear functional from $H$ to $K$.
The following proposition states that in cases of interest, we can express this more general relative, two arguments, $T$-transform by using the product $\#$ and the relative $T$-transform of Definition \ref{def:relativettransform}.
\begin{proposition} 
\label{prop:relativettransform}
Assume that $\bT$ and $\bS^{-1^{\#}}$ are compatible Hopf $\mathcal{O}$-operators then:
\begin{align*}
e_{\bT | \bS} = e_{\bS^{-1^{\#}} \# \bT \,|\,\varepsilon}.
\end{align*}
\end{proposition}
\begin{proof}
Recall that $\mathcal{S}_{H_\bT} = e_\bT \circ \mathcal{S}_H$ (see Theorem \ref{thm:tttransformtwoantipodes}). Then 
\begin{align*}
\mathcal{S}_{H_{\bT}} \circ \mathcal{S}_{H^{op}_{\bS}} = e_\bT \circ e_{\bS \circ \mathcal{S}}= e_{\bS\circ\mathcal{S} \# \bT}= e_{\bS^{-1^{\#}} \# \bT}  = \mathcal{S}_{\bS^{-1^\#} \# \bT} \circ \mathcal{S}_H = \mathcal{S}_{\bT \# \bS^{-1^\#}} \circ \mathcal{S}_H
\end{align*}
It follows that
\begin{align*}
(\bS\circ \mathcal{S}_{H^{op}_{\bS}})\star (\bT \circ \mathcal{S}_{H_\bT} \circ \mathcal{S}_{H_\bS}) 
&=(\mathcal{S}_H \circ \bS) \star (\mathcal{S}_H \circ \bT \circ \mathcal{S}_{H_\bS}) && {\rm eqn.\,\,}\eqref{def:Ohopfoperator}, \eqref{eqn:productohopfoperator}\\
&= \mathcal{S}_H \circ ((\bT \circ \mathcal{S}_{H_\bS}) \star \bS )\\
&=(\bT^{-1^\#} \# \bS )^{-1^\star}\\
&=(\bT^{-1^\#} \# \bS)  \circ \mathcal{S}_{\bT^{-1^\#} \# \bS} && {\rm Theorem\,\,} \ref{thm:semenovtianshanskii}\\
&=(\bT \# \bS^{-1^{\#}})  \circ \mathcal{S}_H \circ \mathcal{S}_{\bT^{-1^\#} \# \bS} && {\rm Proposition \,\,} \ref{prop:inverseOhopfoperator}\\
&=(\bT \# \bS^{-1^{\#}}) \circ e_{\bT \# \bS^{-1^{\#}}} && {\rm Theorem \,\, \ref{thm:tttransformtwoantipodes}}
\end{align*}
\end{proof}
We specialize this construction to the setting exposed at the beginning of the section. Pick $T,T^c \in G$ two formal series of multilinear functionals on the algebra $B$ starting with $1$, then (with $I={\rm id}_B$)
\begin{equation}
\label{eqn:ttransformlambda}
e_{\bm{\lambda}\,|\,I\varepsilon}(T,T^c) = (e_{\bm{\lambda}}(T), T^c \cdot_{{\bm \lambda}} e_{\bm{\lambda}}(T)) = (e_{\bm{\lambda}}(T), T^c\cdot (I\centerdot e_{\lambda}(T)).
\end{equation}
Recall that we set $K_a$, respectively $K_a^c$, the formal series of the free cumulants of $x$, respectively, the formal series of the conditionally free cumulants of $x$
\begin{align}
\label{eqn:cumulantsmoments}
(C_x^{\psi})^{-1^{\circ}} = I \centerdot(1+ K_a \centerdot I)^{-1}, \quad (C_x^{\varphi})^{-1^{\circ}} = I \centerdot (1+ K^c_x\circ (C_x^{\psi} \circ  (C_x^{\varphi})^{-1^{\circ}} \centerdot I))^{-1}
\end{align}
where the multilinear series $C^{\psi}_x$ and $C^{\varphi}_x$ are the operator-valued momen$T$-transforms of $x$ \cite{dykema2006stransform}: 
\begin{align}
C_x^{\varphi}=\sum_{n\geq 0} \varphi(b(ab)^n),\quad C_x^{\psi}=\sum_{n\geq 0} \psi(b(ab)^n).
\end{align}
(We have noted  $\varphi(b(ab)^n)$ for the multilinear function $(b_0,\ldots,b_n)\mapsto \varphi(b_0ab_1\cdots ab_n)$ ). In our notations and with the multilinear series $\tilde{C}^{\psi}_x$ (starting with $1$) defined through the relation ${C}^{\psi}_x = I\centerdot \tilde{C}^{\psi}_x$, 
\begin{align}
\label{eqn:momentdoubleinversion}
\tilde{C}^{\psi}_x = ((1+ K_a \centerdot I)^{-1})^{-1^{\star_{\bm{\lambda}}}}
\end{align}
and $\tilde{C}_x^{\psi}$ is the solution to the fixed point equation (the moments-cumulants relations)
\begin{align}
\label{eqn:dependencehc}
\tilde{C}^{\psi}_x = (1+ K_a \centerdot I) \cdot_{\bm{\lambda}}\tilde{C}^{\psi}_x
\end{align}
We note that, with $C_x^{\psi}=I+I\centerdot \bar{C}_x^{\psi} \centerdot I$:
$$
H_a = \bar{C}_x^{\psi} (1+I\centerdot \bar{C}_x^{\psi})^{-1}
$$
In fact, the defining relation: $
I = C_x^{\psi} \centerdot (1+(K_a \circ C_x^{\psi}) \centerdot C_x^{\psi})^{-1^{\centerdot}} $ implies, since the left translation by $I$ is injective :
$
\bar{C}^{\psi}_x \centerdot I=(K_a \circ C^{\psi}_x) \centerdot C_x^{\psi}, 
$
and for the same reason, but applied to the right translation by $I$, 
$ 
\bar{C}^{\psi}_x = (K_a \circ C^{\psi}_x)\centerdot (1+ I\centerdot \bar{C}^{\psi}_x)
$
It is easily seen that this last fixed point equation is equivalent to the fixed point equation \eqref{eqn:formulahx} but for $(\bar{C}_x^{\psi} (1+I\centerdot \bar{C}_x^{\psi})^{-1}$ and the formula \eqref{eqn:momentdoubleinversion} is proved.

The purpose of the following two Theorems is to demonstrate the utility of our formalism in defining, first, the \emph{operator-valued conditionally free $T -$ transform}, which may be used to compute the operator-valued conditionally free multiplicative convolution and also the conditional $H$-transform, which is used in computing the conditionally monotone multiplicative convolution of two non-commutative distributions. 
Given two formal series $(H,K) \in \CC_0\llbracket {\rm End} B \rrbracket \times \CC_0\llbracket {\rm End B} \rrbracket$, we will use $(H,K)^{-1^{\ltimes_{\bm{\lambda}}}}$ for the inverse of $(H,K)$ with respect to the product of $ \CC_0\llbracket {\rm End} B \rrbracket _{\bm{\lambda}}\ltimes_{\cdot_{\bm{\lambda}}}\CC_0\llbracket {\rm End} B \rrbracket$, see \eqref{eqn:producttwistedgp}.

\begin{theorem}  
\label{thm:factorisationconditionalttransform}
Pick $a,b$ two conditionally free random variables in an algebraic conditional probability space $(\mathcal{A},\varphi,\psi)$. Additionally, we assume that $\psi(a)=\psi(b)=1=\varphi(a)=\varphi(b)$.
For $x\in \{a,b\}$, we define the multilinear series $T_x \in \CC_0\llbracket {\rm End} B \rrbracket$ and $T_x^c\in \CC_0\llbracket {\rm End} B \rrbracket$ by (see \eqref{eqn:cumulantsmoments})
\begin{align}
\label{eqn:conditionalttransform}
(K_a, K^c_x) := e_{\bm{\lambda}\,|\,\varepsilon}(T_x,T_x^c) = ((T_x,T^c_x)^{-1^{\centerdot}})^{-1^{\ltimes_{\bm{\lambda}}}},
\end{align}
then the following factorisation holds 
\begin{equation}
(T_{ab},T_{ab}^c) = (T_{a},T_{a}^c)\ltimes_{\bm{\lambda}^{-1^\#} \# \bm{\rho} }(T_{b},T_{b}^c)=(T_a\star_{\bm{\lambda}^{-1^\#} \# \bm{\rho}}T_a, (T^c_a \cdot_{\bm{\lambda}^{-1^\#} \# \bm{\rho}} T_b)\centerdot T^c_b) .
\end{equation}
\end{theorem}
\begin{remark}
We call the pair $(T_x,T^c_x)$ the conditional operator-valued $T$-transform of $x$.  We made the restrictive assumptions $\psi(x)=1=\varphi(x)$. These can be relaxed to $\psi(x)=\alpha$ and $\varphi(x)=\beta$ invertible elements of $B$. Instead of considering multilinear functions series starting with 1, we allow these series to start (to have degree zero component) with any invertible element of $B$, we call the set of such series $\mathbb{C}_{\times}\llbracket \textrm{End} B\rrbracket$. Likewise, the set of series with no degree zero component but starting with a degree one component equal to $\alpha I$ for some invertible $\alpha$ in $B$ is called $\Gamma_\times$. The left and right translations by the identity $I$ from $\mathbb{C}_{\times}\llbracket \textrm{End} B\rrbracket$ to $\Gamma_\times$ are compatible $\mathcal{O}$-operators and the equation \eqref{eqn:conditionalttransform} define the conditional operator-valued $T$-transform under these weaker assumptions. 
\end{remark}
The proof of Theorem \ref{thm:factorisationconditionalttransform} is based on the following Lemma, the operator-valued version of the Lemma 2.2 in \cite{popa2015multiplication}
\begin{lemma} 
\label{lemma:factorisation}
Let $a,b$ be two conditionally free random variables with $\varphi(a)=\varphi(b)=1=\psi(a)=\psi(b) $and define the multilinear series $K_{a \dashv b}$, $K_{a \vdash b}$ 
\begin{align}
	\label{eqn:equationdeuxun}
	&K_b \cdot_{\bm{\lambda}} (K_a \cdot_{\bm{\rho}} K_{a \vdash b}) = K_{a \vdash b}, \quad K_a  \cdot_{\bm{\rho}} (K_b \cdot_{\bm{\lambda}} K_{a \dashv b}) = K_{a \dashv b}
\end{align}
then:
\begin{align}
\label{eqn:condcumulantsmult}
&K^c_{ab} =  (K^c_{a} \cdot_{\bm{\rho}} K_{a \vdash b}) \centerdot (K^c_{b} \cdot_{\bm{\lambda}}K_{a \dashv b}), \\ 
&K_{ab} = K_b \star_{\bm{\lambda}} K_{a \dashv b}= K_a \star_{\bm{\rho}} 
K_{a \vdash b}	         \label{eqn:freecumulantsmult}
\end{align} 
\end{lemma}
\begin{proof}
Notice that \eqref{eqn:condcumulantsmult} implies \eqref{eqn:freecumulantsmult} when $\varphi = \psi$. It is enough to prove \eqref{eqn:condcumulantsmult}.
We use the notations introduced in the proof of Lemma 2.2 in \cite{popa2015multiplication}.Additionally, we set $[2n]=\{1<\hat{1}<2<\hat{2}<\cdots n < \hat{n}\}=[n]\cup \widehat{[n]}$. First, we prove
\begin{align}
\label{eqn:krewerasun}
\kappa^c_n(a_1b_1,\ldots, a_nb_n)=\sum_{\pi \in {\rm NC}_0(2n)}\mathcal{K}^c(a_1,b_1,\ldots,a_n,b_n)
\end{align}
or equivalently
\begin{align}
\label{eqn:krewerasdeux}
\mathcal{K}^c_{\pi}(a_1b_1,\ldots,a_nb_n) = \sum_{\substack{\sigma \in {\rm NC}_S(2n)\\ \sigma \vee \hat{0}_n = \hat{\pi} }} \mathcal{K}^c_{\sigma}(a_1,b_1,\ldots,a_n,b_n)
\end{align}
with $\mathcal{K}^c_{\pi}$ the partitioned conditional free cumulants, with $\pi_1,\ldots,\pi_n$ the irreductible components of the partition $\pi$ and $a_1,\ldots,a_n$ random variables
\begin{align*}
&\mathcal{K}_{\pi}(a_1,\ldots,a_n) = K_{\pi_1}(a_{\pi_1})\cdots K_{\pi_n}(a_{\pi_1}) \\
&\mathcal{K}_{\pi_j}(a_1,\ldots,a_n) = \kappa_{|V_j|}(a_{v^j_1}\mathcal{K}_{\pi_j^1}(a_{\pi_j^1}),\ldots,\mathcal{K}_{\pi_j^{p-1}}(a_{\pi_j^{p-1}})a_{v^j_p})
\end{align*}
\begin{align*}
&\mathcal{K}^c_{\pi}(a_1,\ldots,a_n)   = K^c_{\pi_1}(a_{\pi_1})\cdots K^c_{\pi_n}(a_{\pi_1}) \\
&\mathcal{K}^c_{\pi_j}(a_1,\ldots,a_n) = \kappa^c_{|V_j|}(a_{v^j_1}\mathcal{K}_{\pi_j^1}(a_{\pi_j^1}),\ldots,\mathcal{K}_{\pi_j^{p-1}}(a_{\pi_j^{p-1}})a_{v^j_p})
\end{align*}
where $V_j$ is the outer block of $\pi_j$, $V_j=\{v^j_1,\ldots,v^j_p\}$ and the $\pi^k_j$ is the restriction of $\pi_j$ to $\{v^j_k+1, v^j_{k+1}-1\}$ and $(\kappa^c_n)_{n\geq 1}$ is the sequence of operator valued conditionally free cumulants,
\begin{align*}
K^c_a = 1+ \sum_{n\geq 1} K^c_a {}^{(n)},\quad K^c_a{}^{(n)}(x_1,\ldots,x_n)=\kappa^c_{n+1}(ax_1,\ldots, x_na)
\end{align*}
The same inductive arguments as in the proof of Lemma 2.2 \cite{popa2015multiplication} yields equation \eqref{eqn:krewerasdeux}. 
Before continuing with the proof, we remark that a partition $\alpha=(\alpha_+,Kr(\alpha_+)) \in \mathrm{NC}_0(n)$ has the following restriction property, whenever $O=\{o_1 < \cdots < o_k\} \subset [n]$ (or $\widehat{O}=\{\widehat{o}_1 < \cdots < \widehat{o}_k\} \subset \widehat{[n]}$) is an outer block of $\alpha$, the restriction of $\alpha$ to an interval $I=\rrbracket o_i,o_{i+1}\llbracket$ is a non-crossing partition, that we will call $\alpha^{\prime}$ such that $\{\{ \delta \}\} \cup \alpha^{\prime} $ is a non-crossing partition in ${\rm NC}_0(\{\delta \}\cup I)$ with the convention that $i_1 = \hat{\delta}$, $\delta < i_1 < i_2 < \hat{i}_2 < \cdots i_k < \hat{i}+k$ ( if instead $O \subset \widehat{[n]}$, we add $\delta$ with the convention that $\delta$ is greater than the elements in $I$).
Now since $\pi = (\pi_+, Kr(\pi_+))$ has two outer block, the one containing $1$ and the other containing $n$, one gets
\begin{align}
\label{eqn:unun}
&\kappa^c_n((ab) x_1,\ldots, x_{n-1}(ab)) \\
&= \sum_{p=1}^n \sum_{\substack{I, J}} \sum_{\alpha,\beta} \kappa^c_{p}(a\mathcal{K}_{\alpha_1}((1,b,a,b,\ldots,a,b)(1,x_{I_1 \cap{\widehat{[n]}}})),\ldots,a\mathcal{K}_{\alpha_{p-1}}((1,b,a,b,\ldots,a,b)(1,x_{I_{p-1}\cap{\widehat{[n]}}})),a) \nonumber\\
&\hspace{1cm}\kappa^c_{n-p}(b,\mathcal{K}_{\beta_1}((x_{J_1\cap [n]},1)(a, b,\ldots,a, b,a,1)b,\ldots,\mathcal{K}_{\beta_{n-p-1}}((x_{J_{n-p-1}\cap [n]},1)(a,b,\ldots,a,1))b)\nonumber
\end{align}
where the second and third sums ranges overs non intersecting intervals of integers $I_1,\ldots,I_p$ and $J_1 < \cdots <J_{n-q}$ with $I_1 <\cdots I_{p-1} < J_1 < \cdots < J_{n-p-1}$ of $\{1,\ldots,\hat{n}\}$ such that $I_1,\ldots,I_p \cup J_1 \cdots < J_{n-q}$ covers $\{1,\ldots,2n\}$ but a subset of $n$ points.
The partitions $\alpha$ are non-crossing partitions in ${\rm NC}_0(\{\delta \}\cup\{ I \})$ and the partitions $\beta$ are irreducible non-crossing partitions in ${\rm NC}_0(J \cup \{\delta \})$. 

Notice that $\mathcal{K}_{\alpha}((1,b,a,b,\ldots,a,b)(1,x_{I \cap \widehat{[n]}}))$ is non-zero only when $\{\delta\}\in \alpha$ ($\{ \delta\}\in\alpha_+$ which implies that $Kr(\alpha_+)$ is an irreducible partition $\{i_1,\hat{i}_2,\hat{i}_3,\ldots,\hat{i}_k\}$ with if $\delta \cup I=\{\delta < i_1 < i_2 < \hat{i}_2\cdots < i_k < \hat{i}_k \}$
). Moreover, for any $x=(x_1,\ldots,x_k)$
\begin{align} \label{eqn:undeux}
\kappa_{k}((b,ab,\ldots,ab)(1,x))&=\sum_{\alpha}\mathcal{K}_{\alpha}((1,b,a,b,\ldots,a,b)(1,x)) \\
&= \sum_{p,s}\kappa(b
\kappa_{s_1}((x_{I_1},1)(ab,\ldots,ab,a)), \ldots, \kappa_{s_p}((x_{I_p},1)(ab,\ldots,ab,a))b) \label{eqn:undeuxprime}
\end{align}
where $I_j=\llbracket s_1+\cdots s_j,s_1+\cdots s_{j+1}\rrbracket$ and $p\geq 2$.
The same reasoning yields 
\begin{align}
\kappa_{k}((x,1)(ab,\ldots,ab,a))&=\sum_{\alpha}\mathcal{K}_{\alpha}((x,1)(ab,\ldots,ab,a)) \label{eqn:untrois}\\
&= \sum_{p,s}\kappa(a
\kappa_{s_1}((x_{J_1},1)(b,\ldots,ab)), \ldots, \kappa_{s_p}((x_{J_p},1)(b,ab,\ldots,ab))a) \label{eqn:untroisprime}
\end{align}
Together with the equation inferred from \eqref{eqn:unun} and \eqref{eqn:undeux}, \eqref{eqn:untrois}:
\begin{align}
&\kappa^c_n((ab) x_1,\ldots, x_{n-1}(ab)) \nonumber\\
&= \sum_{p}\sum_{s,t} \kappa^c_{p}(a\kappa_{s_1}((1,b,a,b,\ldots,a,b)(1,x_{I_1 \cap{\widehat{[n]}}})),\ldots,a\kappa_{s_p}((1,b,a,b,\ldots,a,b)(1,x_{I_{p-1}\cap{\widehat{[n]}}})),a) \\
&\hspace{1cm}\kappa^c_{n-p}(b,\kappa_{t_1}((x_{J_1},1)(a, b,\ldots,a, b,a,1)b,\ldots,\kappa_{t_{n-p-1}}((x_{J_{n-p-1}},1)(a,b,\ldots,a,1))b) \nonumber
\end{align}
The results now follow by noticing that \eqref{eqn:undeuxprime} and \eqref{eqn:untroisprime} the generating series of the $$(x_1,\ldots,x_k) \mapsto \kappa_{k}((ab,x_1\ldots,ab x_{k-1},a))$$ and $$(x_1,\ldots,x_k) \mapsto\kappa_{k}((b, x_1ab,\ldots,x_k ab))$$ satisfy the fixed point equations in \eqref{eqn:equationdeuxun} respectively. 
 \end{proof}
\begin{remark} Notice that $K_a \cdot_{\bm{\rho}} K_{a \vdash b} = K_{a \dashv b}$ and $K_b \cdot_{\bm{\lambda}} K_{a \dashv b}	 = K_{a \vdash b}	 $. Hence \eqref{eqn:condcumulantsmult} is equivalent to 
\begin{align}
K^c_{ab} = (K^c_{a} \cdot_{\bm{\rho}} K_{a \vdash b}) \centerdot (K^c_{b} \cdot_{\bm{\lambda}}(K_a \cdot_{\bm{\rho}} K_{a \vdash b} )) =  (K^c_{a} \cdot_{\bm{\rho}} (K_b \cdot_{\bm{\lambda}} K_{a \dashv b})) \centerdot (K^c_{b} \cdot_{\bm{\lambda}}K_{a \dashv b}). 
\end{align}
It is also possible to write the equations \eqref{eqn:equationdeuxun} and \eqref{eqn:condcumulantsmult} by using the product $\ltimes_{\rho}$ and the diagonal group action $\cdot_{\bm{\lambda}}$ of $\mathbb{C}_{0}[{\rm End} B]_{\bm{\lambda}}$ on  $\mathbb{C}_{0}[{\rm End} B]\times \mathbb{C}_{0}[End B]$.
Notice first that $K_{ab} = K_a \star_{\bm{\rho}}(K_b \cdot_{\bm{\lambda}} K_{a \dashv b})$. Then, we claim that 
\begin{align}
(K_{ab}, K^c_{ab}) = (K_a,K^c_a)\ltimes_{\bm{\rho}}((K_b,K_b^c)\cdot_{\bm{\lambda}}K_{a \dashv b})
\end{align}
\end{remark}
\begin{proof}[Proof of Theorem \ref{thm:factorisationconditionalttransform}]
Let us first notice the following relations
$$
K_{a \dashv b} = K_a \cdot_{\bm{\rho}}(T_b \cdot_{\bm{\lambda}} K_{ab}) = T^{-1^{\centerdot}}_b \star_{\bm{\lambda}} K_{ab}
$$
The proof of the factorisation of the conditional $T$-transform is a straightforward computations based on Lemma \ref{lemma:factorisation}. We insert first the definitions of the various products to obtain
\begin{align*}
K^c_{ab} &= (K^c_{a} \cdot_{\bm{\rho}} K_{a \vdash b}) \centerdot (K^c_{b} \cdot_{\bm{\lambda}}(K_a \cdot_{\bm{\rho}} K_{a \vdash b} ))\\
	       &= ((T^c_a \cdot_{\bm{\lambda}} K_a)\cdot_{\bm{\rho}} K_{a \vdash b})\centerdot (T^c_b \cdot_{\bm{\lambda}} (K_b \star_{\bm{\lambda}} K_{a \dashv b})) \\
	       &= ((T^c_a \cdot_{\bm{\lambda}} K_a)\cdot_{\bm{\rho}} K_{a \vdash b})\centerdot (T^c_b \cdot_{\bm{\lambda}} K_{ab}) \\
	       &=((T^c_a \cdot_{\bm{\lambda}} K_a)\cdot_{\bm{\rho}} (K_b \cdot_{\bm{\lambda}} K_{a \dashv b}))\centerdot (T^c_b \cdot_{\bm{\lambda}} K_{ab})\\
	       &= (T^c_a \cdot (\bm{\lambda}(K_a)\circ \bm{\rho}(K_b \cdot_{\bm{\lambda}} K_{a \dashv b}) ))\centerdot  (T^c_b \cdot_{\bm{\lambda}} K_{ab}) \\
	       &= (T^c_a \cdot (\bm{\lambda}(K_a)\circ \bm{\rho}(T_b \cdot_{\bm{\lambda}} (K_b\star_{\bm{\lambda}}K_{a \dashv b}) ))\centerdot  (T^c_b\cdot_{\bm{\lambda}} K_{ab}) \\
	       &=(T^c_a \cdot (\bm{\lambda}(K_a)\circ \bm{\rho}(T_b \cdot_{\bm{\lambda}} K_{ab} ))\centerdot  (T^c_b\cdot_{\bm{\lambda}} K_{ab}).
\end{align*}
We then use the fact that $(\bm{\rho},\bm{\lambda}^{-1^{\#}})$ is a matching $\mathcal{O}$-operator to infer 
\begin{align*}
\bm{\lambda}(K_a)\circ \bm{\rho}(T_b \cdot_{\bm{\lambda}} K_{ab} )& = \bm{\lambda}((K_a \cdot_{\rho} (T_b \cdot_{\bm{\lambda}} K_{ab}) ) \cdot_{\rho} (T_b \cdot_{\bm{\lambda}} K_{ab})^{-1^{\star_{\bm{\rho}}}} )\circ \bm{\rho}(T_b \cdot_{\bm{\lambda}} K_{ab} ) \\
& = \bm{\rho}((T_b \cdot_{\bm{\lambda}} K_{ab}) \cdot_{\bm{\lambda}}(K_a \cdot_{\rho} (T_b \cdot_{\bm{\lambda}} K_{ab}))^{-1^{\star_{\bm{\lambda}}}}) \circ \bm{\lambda}(K_a \cdot_{\rho} (T_b \cdot_{\bm{\lambda}} K_{ab})) && \textrm{see eqn.\,\,}\eqref{eqn:matchinggroup} \\
& = \bm{\rho}((T_b \cdot_{\bm{\lambda}} K_{ab}) \cdot_{\bm{\lambda}}(T_b^{-1^{\centerdot}}\star_{\bm{\lambda}}K_{ab})^{-1^{\star_{\bm{\lambda}}}}) \circ \bm{\lambda}(T_b^{-1^{\centerdot}}\star_{\bm{\lambda}}K_{ab}) \\
& = \bm{\rho}(T_b\cdot_{\bm{\lambda}} (T_b^{-1^{\centerdot}})^{-1^{\star_{\bm{\lambda}}}}) \circ \bm{\lambda}(T_b^{-1^{\centerdot}}\star_{\bm{\lambda}}
K_{ab}) \\
&=  \bm{\rho}^{-1^{\#}}(T_b^{-1^{\centerdot}} \cdot_{\bm{\lambda}^{-1^{\#}}} T_b^{-1^{\star_{\bm{\lambda}^{-1^{\#}}}}}) \circ \bm{\lambda}^{-1^{\#}}(T_b) \circ \bm{\lambda}(K_{ab}).
\end{align*}
where the last equality follows the fact that $(T_b^{-1^{\centerdot}})^{-1^{\star_{\bm{\lambda}}}} = (T_b^{-1^{\star_{\bm{\lambda}^{-1^{\#}}}}})^{-1^{\centerdot}}$.
Since 
$$T_b^{-1^{\centerdot}} \cdot_{\bm{\lambda}^{-1^{\#}}} T_b^{-1^{\star_{\bm{\lambda}^{-1^{\#}}}}} =  T_b^{-1^{\star_{\bm{\lambda}^{-1^{\#}}}}}
$$ we infer 
\begin{align*}
K^c_{ab} &= (T^c_a \cdot (\bm{\rho}^{-1^{\#}}(T_b^{-1^{\star_{\bm{\lambda}^{-1^{\#}}}}})  \circ \bm{\lambda}^{-1^{\#}}(T_b)) \cdot_{\bm{\lambda}} K_{ab})\centerdot (T^c_b\cdot_{\bm{\lambda}} K_{ab})\\
	       &= ((T^c_a \cdot (\bm{\rho}\#\bm{\lambda}^{-1^{\#}})(T_b) \centerdot T^c_b ) \cdot_{\bm{\lambda}} K_{ab}
\end{align*}
and the result follows from Theorem \ref{thm:semenovtianshanskii} and the formulae recalled at the beginning of the Section.
 \end{proof}
 \subsubsection{Conditional $H$-transform}
To state the next theorem, we set $E = \frac{I}{1-I} \in \Gamma$.
\begin{theorem} 
\label{thm:conditionalhtransform}
Define the two formal series $H_x$ and $H_x^c$ in $\mathbb{C}_0[End B]$ through :
$$(H_x, H^c_x) := e_{\bm{\lambda}\,|\,\varepsilon}(T_x\cdot E,T_x^c\cdot E) = ((T_x\cdot E,T^c_x\cdot E)^{-1^{\centerdot}})^{-1^{\ltimes_{\bm{\lambda}}}} = (H_x, (T^c_x \cdot E) \cdot_{\lambda} H_x ).
$$
for $x$ a self-adjoint element of $\mathcal{A}$ with $\phi(x)=\psi(x)=1$.
Pick $a,b$ two conditionally monotone random variables with $\phi(a)=\psi(a)=1=\phi(b)=\psi(b)$, then:
\begin{align}
(H_{ba},H_{ba}^c) = (H_{a},H_{a}^c)\ltimes_{\bm{\lambda}}(H_{b},H_{b}^c) .
\end{align}
\end{theorem}
\begin{remark} Notice that
\label{rk:dependencehtransform}
\begin{align*}
K^c_x \cdot ((1-I\centerdot H_x)^{-1}\centerdot I) &= T^c_x \circ (I\centerdot K) \circ ((1-I\centerdot H_x)^{-1}\centerdot I)\\
&=T^c_x \circ ((1-I\centerdot H_x)^{-1})\centerdot I \centerdot (K_x\circ (I\centerdot (1-I\centerdot H_x)^{-1})) \\
&=T^c_x \circ ((1-I\centerdot H_x)^{-1}) \centerdot (I\centerdot H_x))\\
&=H^c_x.
\end{align*}
Similar computations as done right before Theorem \ref{thm:factorisationconditionalttransform} yield
\begin{align}
\label{eqn:momentshc}
H^c_x = \bar{C}_x^{\varphi}(1+I\centerdot \bar{C}_x^{\varphi})
\end{align}
\end{remark}
\begin{proof} We let $\tilde{a}$ and $\tilde{b}$ be two conditionally free random variables in a conditional probability space $(\mathcal{A},\psi,\varphi,B)$ as in Theorem \ref{thm:conditionalhtransform}. Since $\tilde{a}$ and $\tilde{b}$ are conditionally free,
$$
T^{c}_{\tilde{b}\tilde{a}} = (T^c_{\tilde{b}}\cdot_{\bm{\lambda}^{-1^{\#}}\#{\bm \rho}} T_{\tilde{a}})\centerdot T^{c}_{\tilde{a}}
$$
Thus we infer
\begin{align}
    (T^c_{\tilde{b}\tilde{a}} \cdot E) \cdot_{\bm{\lambda}} H_{\tilde{b}} &=((T^c_{\tilde{b}} \cdot E)\cdot_{\bm{\lambda}} H_{\tilde{b}})\centerdot ((T^c_{\tilde{a}} \cdot E) \cdot_{\bm{\lambda}}H_{\tilde{b}}))
\end{align}
We note that $H_{\tilde{b}\tilde{a}}=H_{\tilde{b}}$ since $\tilde{a}$ is distributed as $1\in\mathcal{A}$ (see equation \eqref{eqn:dependencehc} with respect to $\psi$ and thus from the previous equation we infer that 
\begin{align}
H^c_{\tilde{b}\tilde{a}} = H^c_{\tilde{b}} \centerdot ((T^c_{\tilde{a}} \cdot E) \cdot_{\bm{\lambda}}H_{\tilde{b}})) = H^c_{b} \centerdot ((T^c_{\tilde{a}} \cdot E)\cdot_{\bm{\lambda}} H_{{b}})
\end{align}
where the second equality follows from the fact that $\tilde{b}$ is distributed as $b$ with respect to $\psi$ and $\varphi$. Following Remark \ref{rk:dependencehtransform} and equation \eqref{eqn:momentshc}, $H^c_{\tilde{b}\tilde{a}}$ does only depend upon the moments of $ \tilde{b}\tilde{a}$ under $\varphi$, which are equal to the moments of $ba$ and $H^c_{\tilde{b}\tilde{a}}=H^c_{{b}{a}}$. Also, $H_{\tilde{a}}=1$ implies
$$
H^c_{{b}{a}} = H^c_{\tilde{b}\tilde{a}} = H^c_{b} \centerdot (H^c_{\tilde{a}}\cdot_{\bm{\lambda}} H_{{b}})
$$ 
Now since $\tilde{a}$ is distributed as $a$ under $\varphi$, equation \eqref{eqn:momentshc} implies $H^c_{\tilde{a}}=H^c_{{a}}$ and the result follows.
\end{proof}
We end this section with the computations of the subordination equation for the multiplicative convolution for operator-valued conditional freeness, that is the conditional counterpart of equation \eqref{eqn:subordinationfixedpointeqn}
\begin{align*}
H^c_{ab} &= T^c_{a} \cdot ((\bm{\lambda}^{-1^{\#}} \# \bm{\rho})(T_y)\circ \frac{I}{1-I}\circ \bm{\lambda}(H_{ab})) \centerdot (T^c_y\cdot (\frac{I}{1-I}\circ \bm{\lambda}(H_{ab}))) \\
	      &=T^c_{a} \cdot (E\circ (\bm{\lambda}^{-1^{\#}} \# \bm{\rho})(T_y \cdot E)\circ \bm{\lambda}(H_{ab})) \centerdot (T^c_y\cdot (\frac{I}{1-I}\circ \bm{\lambda}(H_{ab}))) \\
	      &=T^c_{a} \cdot (E\circ (\bm{\lambda}^{-1^{\#}} \# \bm{\rho})(T_y \cdot E)\circ \bm{\lambda}(H_{b}\star_{\bm{\lambda}}H_{a\vdash b})) \centerdot (T^c_y\cdot (\frac{I}{1-I}\circ \bm{\lambda}(H_{b}\star_{\bm{\lambda}}H_{a\vdash b}))) \\
	      &=T^c_{a} \cdot (E\circ (\bm{\lambda}^{-1^{\#}} \# \bm{\rho})(T_y \cdot E)\circ \bm{\lambda}(H_{b}\star_{\bm{\lambda}}H_{a\vdash b})) \centerdot (H^c_b\cdot_{\bm{\lambda}}H_{a\vdash b}) \\
	      &=T^c_{a} \cdot (E\circ \bm{\rho}((T^{-1^{\centerdot}}_y \cdot E)^{-1^{\star_{\bm{\lambda}}}})\bm{\lambda}((T^{-1^{\centerdot}}_y \cdot E))\circ \bm{\lambda}(H_{b}\star_{\bm{\lambda}}H_{a\vdash b})) \centerdot (H^c_b\cdot_{\bm{\lambda}}H_{a\vdash b}) \\
	      &=T^c_{a} \cdot (E\circ \bm{\rho}(H_y)\circ\bm{\lambda}(H_y^{-1^{\star_{\bm{\lambda}}}})\circ \bm{\lambda}(H_{b}\star_{\bm{\lambda}}H_{a\vdash b})) \centerdot (H^c_b\cdot_{\bm{\lambda}}H_{a\vdash b})\\
	      &=T^c_{a} \cdot (E\circ \bm{\rho}(H_y)\circ\bm{\lambda}(H_{a\vdash b})) \centerdot (H^c_b\cdot_{\bm{\lambda}}H_{a\vdash b}) \\
	      &=T^c_{a} \cdot (E\circ \bm{\lambda}(H_{a\vdash b}\cdot_{\bm{\rho}}(H_y\cdot_{\bm{\lambda}}H_{a\vdash b})^{-1^{\star_{\bm{\rho}}}})\circ\bm{\rho}(H_y\cdot_{\bm{\lambda}}H_{a\vdash b})) \centerdot (H^c_b\cdot_{\bm{\lambda}}H_{a\vdash b}) && {\rm see\,\,} \eqref{eqn:matchinggroup} \label{eqn:usematching}\\
	      &=T^c_{a} \cdot (E\circ \bm{\lambda}(H_a)\circ\bm{\rho}(H_y\cdot_{\bm{\lambda}}H_{a\vdash b})) \centerdot (H^c_b\cdot_{\bm{\lambda}}H_{a\vdash b}) \\
	      &=T^c_{a} \cdot (E\circ \bm{\lambda}(H_a)\circ\bm{\rho}(H_y\cdot_{\bm{\lambda}}H_{a\vdash b})) \centerdot (H^c_b\cdot_{\bm{\lambda}}H_{a\vdash b}) \\
	      &=H^c_a\circ\bm{\rho}(H_y\cdot_{\bm{\lambda}}H_{a\vdash b})) \centerdot (H^c_b\cdot_{\bm{\lambda}}H_{a\vdash b})
\end{align*}
Note that we have used once again the defining relation of matching $\mathcal{O}$-operator.
We set $H_{a\dashv b}=H_y\cdot_{\bm{\lambda}}H_{a\vdash b}$. We have proved the following Theorem
\begin{theorem}
    \label{thm:subordinatoin}
    Let $a,b$ be two conditionally free random variables with $ \varphi(a)=\varphi(b)=1=\psi(a)=\psi(b)$, then 
\begin{equation*}
H^c_{ab} = (H^c_a\cdot_{\bm{\rho}}H_{a\dashv b}) \centerdot (H^c_b\cdot_{\bm{\lambda}}H_{a\vdash b})
\end{equation*}
and 
\begin{equation*}
H_{a\dashv b} = H_y \cdot_{\bm{\lambda}} ( H_a \cdot_{\bm{\rho}}H_{a\dashv b}),\quad H_{a\vdash b} = H_a \cdot_{\bm{\rho}} (H_y \cdot_{\bm{\lambda}} H_{a\vdash b})
\end{equation*}

\end{theorem}

\bibliographystyle{plain}
\bibliography{biblio}

\begin{thebibliography}{10}

\bibitem{agrachev1980chronological}
Andrei~Alexandrovich Agrachev and Revaz~Valer'yanovich Gamkrelidze.
\newblock Chronological algebras and nonstationary vector fields.
\newblock {\em Itogi Nauki: Problemy geometrii}, 11:135–176, 1980.

\bibitem{al2022algebraic}
Mahdi J~Hasan Al-Kaabi, Kurusch Ebrahimi-Fard, Dominique Manchon, and Hans
  Munthe-Kaas.
\newblock Algebraic aspects of connections: from torsion, curvature, and
  post-lie algebras to gavrilov's double exponential and special polynomials.
\newblock {\em arXiv preprint arXiv:2205.04381}, 2022.

\bibitem{bardakov2021rota}
Valeriy~G Bardakov and Vsevolod Gubarev.
\newblock Rota--baxter operators on groups.
\newblock {\em arXiv preprint arXiv:2103.01848}, 2021.

\bibitem{bardakov2022rota}
Valeriy~G Bardakov and Vsevolod Gubarev.
\newblock Rota---baxter groups, skew left braces, and the yang---baxter
  equation.
\newblock {\em Journal of Algebra}, 596:328--351, 2022.

\bibitem{belinschi2012infinite}
S.T. Belinschi, Mihai Popa, and Victor Vinnikov.
\newblock Infinite divisibility and a non-commutative boolean-to-free
  bercovici--pata bijection.
\newblock {\em Journal of Functional Analysis}, 262(1):94--123, 2012.

\bibitem{bozejko1996convolution}
Marek Bo{\.z}ejko, Michael Leinert, and Roland Speicher.
\newblock Convolution and limit theorems for conditionally free random
  variables.
\newblock {\em Pacific Journal of Mathematics}, 175(2):357--388, 1996.

\bibitem{cartier2021classical}
Pierre Cartier and Fr{\'e}d{\'e}ric Patras.
\newblock {\em Classical hopf algebras and their applications}.
\newblock Springer, 2021.

\bibitem{catoire2024cartier}
Pierre Catoire.
\newblock The cartier-quillen-milnor-moore theorem in the post-hopf case.
\newblock {\em arXiv preprint arXiv:2401.09116}, 2024.

\bibitem{curry2019post}
Charles Curry, Kurusch Ebrahimi-Fard, and Hans Munthe-Kaas.
\newblock What is a post-lie algebra and why is it useful in geometric
  integration.
\newblock In {\em European Conference on Numerical Mathematics and Advanced
  Applications}, pages 429--437. Springer, 2019.

\bibitem{das2022cohomology}
Apurba Das.
\newblock Cohomology and deformations of weighted rota--baxter operators.
\newblock {\em Journal of Mathematical Physics}, 63(9):091703, 2022.

\bibitem{dykema2006stransform}
Kenneth~J Dykema.
\newblock On the s-transform over a banach algebra.
\newblock {\em Journal of Functional Analysis}, 231(1):90--110, 2006.

\bibitem{dykema2007multilinear}
Kenneth~J Dykema.
\newblock Multilinear function series and transforms in free probability
  theory.
\newblock {\em Advances in Mathematics}, 208(1):351--407, 2007.

\bibitem{ebrahimi2021twisted}
Kurusch Ebrahimi-Fard and Nicolas Gilliers.
\newblock On the twisted factorization of the $ t $-transform.
\newblock {\em arXiv preprint arXiv:2105.09639}, 2021.

\bibitem{EBRAHIMIFARD201719}
Kurusch Ebrahimi-Fard, Igor Mencattini, and Hans Munthe-Kaas.
\newblock Post-lie algebras and factorization theorems.
\newblock {\em Journal of Geometry and Physics}, 119:19–33, 2017.

\bibitem{ebrahimi2023magnus}
Kurusch Ebrahimi-Fard, Igor Mencattini, and Alexandre Quesney.
\newblock What is the magnus expansion?
\newblock {\em arXiv preprint arXiv:2312.16674}, 2023.

\bibitem{fresse2017homotopy}
Benoit Fresse.
\newblock Homotopy of operads and grothendieck-teichmuller groups.
\newblock 2017.

\bibitem{goncharov2021rota}
Maxim Goncharov.
\newblock Rota--baxter operators on cocommutative hopf algebras.
\newblock {\em Journal of Algebra}, 582:39--56, 2021.

\bibitem{goncharov2024formal}
Maxim Goncharov, Pavel Kolesnikov, Yunhe Sheng, and Rong Tang.
\newblock Formal integration of complete rota-baxter lie algebras.
\newblock {\em arXiv preprint arXiv:2404.09261}, 2024.

\bibitem{guggilam2022formal}
Venkatesh~Subbarao Guggilam.
\newblock A formal power series approach to multiplicative dynamic and static
  output feedback.
\newblock {\em IFAC-PapersOnLine}, 55(30), 2022.

\bibitem{guo2021integration}
Li~Guo, Honglei Lang, and Yunhe Sheng.
\newblock Integration and geometrization of rota--baxter lie algebras.
\newblock {\em Advances in Mathematics}, 387:107834, 2021.

\bibitem{hasebe2011conditionally}
Takahiro Hasebe.
\newblock Conditionally monotone independence i: Independence, additive
  convolutions and related convolutions.
\newblock {\em Infinite Dimensional Analysis, Quantum Probability and Related
  Topics}, 14(03):465--516, 2011.

\bibitem{jacques2023post}
Jean-David Jacques and Lorenzo Zambotti.
\newblock Post-lie algebras of derivations and regularity structures.
\newblock {\em arXiv preprint arXiv:2306.02484}, 2023.

\bibitem{jiang2021lie}
Jun Jiang, Yunhe Sheng, and Chenchang Zhu.
\newblock Lie theory and cohomology of relative rota-baxter operators.
\newblock {\em arXiv preprint arXiv:2108.02627}, 2021.

\bibitem{Kupershmidt1999}
Boris~A. Kupershmidt.
\newblock What a classical r-matrix really is.
\newblock {\em Journal of Nonlinear Mathematical Physics}, 6:448--488, 1999.

\bibitem{li2024sub}
Yunnan Li.
\newblock On the sub-adjacent hopf algebra of the universal enveloping algebra
  of a post-lie algebra.
\newblock {\em arXiv preprint arXiv:2408.01345}, 2024.

\bibitem{li2022post}
Yunnan Li, Yunhe Sheng, and Rong Tang.
\newblock Post-hopf algebras, relative rota-baxter operators and solutions of
  the yang-baxter equation.
\newblock {\em arXiv preprint arXiv:2203.12174}, 2022.

\bibitem{mlotkowski2002operator}
Wojciech M{\l}otkowski.
\newblock Operator-valued version of conditionally free product.
\newblock {\em Studia Math}, 153(1):13--30, 2002.

\bibitem{oudom2008lie}
J-M Oudom and Daniel Guin.
\newblock On the lie enveloping algebra of a pre-lie algebra.
\newblock {\em Journal of K-theory}, 2(1):147--167, 2008.

\bibitem{popa2012realization}
Mihai Popa.
\newblock Realization of conditionally monotone independence and monotone
  products of completely positive maps.
\newblock {\em Journal of Operator Theory}, pages 257--274, 2012.

\bibitem{popa2013non}
Mihai Popa and Victor Vinnikov.
\newblock Non-commutative functions and the non-commutative free
  l{\'e}vy--hin{\v{c}}in formula.
\newblock {\em Advances in Mathematics}, 236:131--157, 2013.

\bibitem{popa2015multiplication}
Mihai Popa, Victor Vinnikov, and Jiun-Chiau Wang.
\newblock On the multiplication of operator-valued c-free random variables.
\newblock In {\em Colloquium Mathematicum}, volume 153, pages 241--260.
  Instytut Matematyczny Polskiej Akademii Nauk, 2018.

\bibitem{quillen1969rational}
Daniel Quillen.
\newblock Rational homotopy theory.
\newblock {\em Annals of Mathematics}, 90(2):205--295, 1969.

\bibitem{semenov2003classical}
Michael~Arsenevich Semenov-Tian-Shansky.
\newblock Integrable systems and factorization problems.
\newblock In Israel Gohberg, Nenad Manojlovic, and Ant\'onio~Ferreira dos
  Santos, editors, {\em Factorization and Integrable Systems}, volume 141 of
  {\em Operator Theory: Advances and Applications}, pages 155--218.
  Birkh\"auser, Basel, 2003.

\bibitem{semenov1983classical}
MA~Semenov-Tyan-Shanskii.
\newblock What is a classical r-matrix?
\newblock {\em Functional analysis and its applications}, 17(4):259--272, 1983.

\bibitem{speicher1998combinatorial}
Roland Speicher.
\newblock {\em Combinatorial theory of the free product with amalgamation and
  operator-valued free probability theory}, volume 627.
\newblock American Mathematical Soc., 1998.

\bibitem{voiculescu1987multiplication}
Dan Voiculescu.
\newblock Multiplication of certain non-commuting random variables.
\newblock {\em Journal of Operator Theory}, pages 223--235, 1987.

\end{thebibliography}
\end{document}